\documentclass[a4paper, reqno, twoside]{amsart}
\usepackage{etex}
\usepackage{amsfonts,amssymb,graphics,amsthm,amsmath}
\usepackage{latexsym}
\usepackage[2emode]{psfrag}
\usepackage{mathrsfs}
\usepackage[all]{xy}
\usepackage{enumerate}
\usepackage[latin1]{inputenc}
\usepackage{lscape}
\usepackage{a4wide}
\usepackage[bookmarks=false]{hyperref}

\usepackage{txfonts}

\usepackage[all]{pstricks}
\usepackage{pst-text,pst-node,pst-coil}
\usepackage{tikz}

\usepackage{cancel}
\usepackage[normalem]{ulem}
\usetikzlibrary{arrows,shapes,backgrounds}

\usepgflibrary{arrows}
\usetikzlibrary{topaths}
\usetikzlibrary{er}

\newtheorem{theorem}{\sc Theorem}[section]
\newtheorem{proposition}[theorem]{\sc Proposition}

\newtheorem{lemma}[theorem]{\sc Lemma}
\newtheorem{corollary}[theorem]{\sc Corollary}
\theoremstyle{definition}
\newtheorem{definition}[theorem]{\sc Definition}

\newtheorem{example}[theorem]{\sc Example}

\theoremstyle{remark}
\newtheorem{remark}[theorem]{\sc Remark}
\newtheorem{remarks}[theorem]{\sc Remarks}

\newcommand{\tensor}[1]{\otimes_{\scriptscriptstyle{#1}}}
\newcommand{\tril}[1]{\triangleleft_{\scriptscriptstyle{#1}}}
\newcommand{\trir}[1]{\triangleright_{\scriptscriptstyle{#1}}}

\newcommand{\Sf}[1]{\mathsf{#1}}
\newcommand{\fk}[1]{\mathfrak{#1}}
\newcommand{\cat}[1]{\mathcal{#1}}
\newcommand{\rmod}[1]{\Sf{Mod}_{#1}}

\renewcommand{\hom}[3]{\mathrm{Hom}_{#1}\left(#2,\,#3\right)}

\newcommand{\td}[1]{\widetilde{#1}}

\newcommand{\bara}[1]{\overline{#1}}

\newcommand{\Lr}[1]{\left[\underset{}{} #1 \right]}

\newcommand{\rcomod}[1]{ \mathsf{Comod}{}_{#1}}

\newcommand{\rhom}[3]{\mathrm{Hom}_{\scriptscriptstyle{\text{-}#1}}(#2,#3)}
\newcommand{\lhom}[3]{\mathrm{Hom}_{\scriptscriptstyle{#1\text{-}}}(#2,#3)}

\newcommand{\LR}[1]{\left\{\underset{}{} #1 \right\}}

\newcommand{\I}{\mathbb{I}}
\newcommand{\Z}[1]{\mathcal{Z}(#1)}
\newcommand{\ZR}[1]{\mathcal{Z}_{\scriptscriptstyle{R}}(#1)}
\newcommand{\ZS}[1]{\mathcal{Z}_{\scriptscriptstyle{S}}(#1)}
\newcommand{\Inv}[2]{\mathrm{Inv}_{#1}(#2)}
\newcommand{\Aut}[2]{\mathrm{Aut}_{#1}(#2)}
\newcommand{\B}[1]{\boldsymbol{#1}}
 
 \newcommand{\id}{\mathrm{Id}}
\newcommand{\mas}{m^{\scriptscriptstyle{S}}_{\scriptscriptstyle{A}}}
\newcommand{\mar}{m^{\scriptscriptstyle{R}}_{\scriptscriptstyle{A}}}

\newcommand{\lxr}{l^{\scriptscriptstyle{R}}_{\scriptscriptstyle{X}}}

\newcommand{\ryr}{r^{\scriptscriptstyle{R}}_{\scriptscriptstyle{X}}}

\newcommand{\lys}{l^{\scriptscriptstyle{S}}_{\scriptscriptstyle{Y}}}
\newcommand{\rxs}{r^{\scriptscriptstyle{S}}_{\scriptscriptstyle{X}}}

\newcommand{\lar}{l^{\scriptscriptstyle{R}}_{\scriptscriptstyle{A}}}

\newcommand{\rar}{r^{\scriptscriptstyle{R}}_{\scriptscriptstyle{A}}}
\newcommand{\mrrp}{m^{\scriptscriptstyle{R}}_{\scriptscriptstyle{R'}}}
\newcommand{\rrrp}{r^{\scriptscriptstyle{R}}_{\scriptscriptstyle{R'}}}
\newcommand{\lrrp}{l^{\scriptscriptstyle{R}}_{\scriptscriptstyle{R'}}}
\newcommand{\murpx}{\mu^{\scriptscriptstyle{R'}}_{\scriptscriptstyle{X}}}

\begin{document}
\allowdisplaybreaks

\title[Invertible Bimodules, Miyashita action and Azumaya Monoids.]{Invertible Bimodules, Miyashita Action in Monoidal Categories and Azumaya Monoids.}
\author{Alessandro Ardizzoni}
\address{University of Turin, Department of Mathematics ``Giuseppe Peano'', via Carlo Alberto 10, I-10123 Torino, Italy}  \email{alessandro.ardizzoni@unito.it}
\urladdr{www.unito.it/persone/alessandro.ardizzoni}
\author{Laiachi El Kaoutit}
\address{Universidad de Granada, Departamento de \'{A}lgebra. Facultad de Educaci\'{o}n, Econon\'ia y Tecnolog\'ia de Ceuta. Cortadura del Valle, s/n. E-51001 Ceuta, Spain}
\email{kaoutit@ugr.es}
\urladdr{http://www.ugr.es/~kaoutit/}
\date{\today}
\thanks{This paper was written while the first author was member of
GNSAGA and partially supported by the research grant ``Progetti di
Eccellenza 2011/2012'' from the ``Fondazione Cassa di Risparmio di Padova e
Rovigo''. The second author was supported by grant MTM2010-20940-C02-01 from the Ministerio de
Educaci\'{o}n y Ciencia of Spain.  His stay, as a visiting
professor at the University of Ferrara on 2010, was supported by INdAM}

\begin{abstract}
In this paper we introduce and study Miyashita action in the context of monoidal categories aiming by this to provide a common framework of previous studies in the literature. We make a special emphasis of this action on Azumaya monoids. To this end, we develop  the theory of invertible bimodules over different monoids (a sort of Morita contexts) in general monoidal categories  as well as their corresponding  Miyashita action.
Roughly speaking, a Miyashita action is a  homomorphism of groups from the group of all isomorphic classes of invertible subobjects of a given monoid to its group of automorphisms. In the symmetric case, we show that for certain Azumaya monoids, which are abundant in practice, the corresponding Miyashita action is always an isomorphism of groups.  This generalizes Miyashita's classical result and sheds light on other applications of geometric nature which can not be treated  using the classical theory.
In  order to  illustrate our methods, we give a concrete application to  the category of comodules over commutative (flat) Hopf algebroids.  This obviously includes the special cases of split Hopf algebroids (action groupoids), which for instance cover the situation of the action of an affine algebraic group on an affine algebraic variety.

\end{abstract}

\maketitle

\begin{small}
\tableofcontents
\end{small}

\pagestyle{headings}

\section*{Introduction}

\subsection{Motivation and overview}\label{ssec:0} The notion of what is nowadays  known as an Azumaya algebra  was first formulated by G. Azumaya in \cite[page 128]{Azumaya:51} where he introduced the Brauer group of a  local ring \cite[page 138]{Azumaya:51}, generalizing by this the classical notion of Brauer group of  a field  which was extremely important in developing the arithmetic study of fields.

For a general commutative base ring, this notion was recovered later on by  M. Auslander  and O. Goldman in \cite{Auslander-Goldman:60}, where several new properties of Azumaya algebras were displayed, see for instance \cite[\S 3]{Auslander-Goldman:60}.  In \cite{Auslander:66} Auslander extended this notion to ringed spaces,  in the case of a topological space endowed with its structural sheaf of rings of continuous  complex  valued functions;  Azumaya algebras\footnote{or Azumaya monoid in the abelian symmetric  monoidal category of quasi-coherent sheaves.} are interpreted as locally trivial algebra  bundles whose fibers are  central simple complex algebras (that is, square matrices over complex numbers).
As was shown by A. Grothendieck in \cite[1.1]{Grothendieck:BrauerI},  the set of isomorphic classes of Azumaya algebras with constant  rank $n^2$ over a topological space, is  identified with the set of isomorphic classes of $GP(n)$-principal bundles with base this space. Here $GP(n)$ is the projective group with $n$ variables over the complex numbers. By using the classifying space of this group, a  homotopic  interpretation of these Azumaya algebras over $CW$-complexes,   is also possible  \cite[1.1]{Grothendieck:BrauerI}.

Naturally, with a compact base space, the global sections of an Azumaya algebra bundle give rise to an Azumaya algebra over the ring of continuous complex valued functions (the base ring). In this way, those Azumaya algebra bundles of rank $n^2$ lead to  Azumaya algebras which are finitely generated and projective of locally constant rank $n^2$ as modules over the base ring. By  the classification theorem \cite[Th\'eor\`eme 6.6, Corollaire 6.7]{Knus-Ojanguren}, a given Azumaya algebra is of this form if and only if it is a twisted-form \cite[(c) p.29]{Knus-Ojanguren} of an $n$-square matrices algebra with coefficients in the base ring.  The latter condition can be interpreted in the case of smooth manifolds\footnote{or at least for almost complex manifolds.}, by saying that there is a surjective submersion to the base manifold such that the induced bundle of any Azumaya algebra bundle of rank $n^2$ is a trivial bundle.

\smallskip

There is no doubt then that  Azumaya algebras are extremely rich objects which, along the last decades, have attracted the attention of several mathematicians from different areas. Unfortunately, we have the feeling that Azumaya algebras have not been deeply investigated   in the general setting of  abelian monoidal categories.  However, the concept of an Azumaya monoid in symmetric monoidal categories has  been earlier introduced in the literature, \cite{Vitale,PareigisIV,Fisher-Palmquist}.

The main motivation of this paper is to try to fulfil this lack of investigation by introducing and studying Miyashita action on Azumaya monoids in abelian symmetric monoidal categories.
Although the results displayed here can be applied to other situations, we limit ourselves to a concrete  application concerning the category of comodules over commutative (flat) Hopf algebroids  which up to  our knowledge seems  not have been treated before.

\subsection{Description of the main results}\label{ssec:1}

In the first part of  this paper, that is Sections \ref{sec:biobjects} and \ref{sec:action}, we introduce and study  Miyashita actions in monoidal categories. Precisely,  consider  a Penrose abelian locally small monoidal category $(\cat{M},\tensor{},\I)$ whose   tensor products $\tensor{}$  are right exact functors (on both arguments) and  two morphisms of monoids $R \rightarrow A \leftarrow S$ which are monomorphisms in $\cat{M}$.  We first consider the set $\mathrm{Inv}_{S,R}\left( A\right)$ which consists of  two-sided invertible $(R,S)$-sub-bimodules of $A$. These are isomorphic classes $(X,i_X)$  of $(R,S)$-sub-bimodules with monomorphism $i_X:X  \hookrightarrow A$ such that there exists another sub-bimodule $(Y,i_Y)$ with compatible isomorphisms $X\tensor{R}Y \cong S$ and $Y\tensor{S}X \cong R$, each in  the appropriate category of bimodules,  which are defined via the multiplications of $A$ with respect to $R$ and $S$ (the pair $(X,Y)$ with the two isomorphisms is also referred to as a two-sided
dualizable
datum).
Section \ref{sec:biobjects} is entirely devoted to  the properties of these data  which will be used in the proofs of results stated in the forthcoming sections.
It is noteworthy to mention here that the set $\mathrm{Inv}_{S,R}\left( A\right)$ is quite different from the one already considered in the literature, see Remark \ref{rem:converse} and  Appendix \ref{ssec:appendix1} where  this difference is made clearer.

Our goal in Section \ref{sec:action} is to construct the map ${\Phi}$:
\begin{equation*}
{\Phi} ^{S,R}:\mathrm{Inv}_{S,R}\left( A\right) \longrightarrow \mathrm{Iso}_{\scriptscriptstyle{\Z{\I}}
\text{-alg}}\left(\ZS{A},\ZR{A} \right),
\end{equation*}
where $\cat{Z}$, $\cat{Z}_R$ and $\cat{Z}_S$ are the functors  $\hom{\cat{M}}{\I}{-}$,  $\hom{{}_R\cat{M}_R}{R}{-}$ and $\hom{{}_S\cat{M}_S}{S}{-}$, respectively (here ${}_R\cat{M}_R$ denotes the category of $R$-bimodules).
The codomain of ${\Phi} ^{S,R}$ is the set of $\Z{\I}$-algebra isomorphisms between the invariant algebras $\ZS{A}$ and $\ZR{A}$.
In particular, when $R=S$,  we show in Proposition \ref{pro: Phi} that the morphism $\Phi ^{R}:=\Phi^{R,R}$ factors as a composition of group homomorphisms:
\begin{equation*}
\mathrm{Inv}_{R}\left( A\right) \longrightarrow \mathrm{Aut}_{\scriptscriptstyle{\cat{Z}_{R}(R)}
\text{-alg}}\left( \ZR{A}\right) \hookrightarrow \mathrm{Aut}_{\scriptscriptstyle{\cat{Z}(\I)}\text{-alg}}\left( \ZR{A}\right).
\end{equation*}
The morphism $\Phi ^{R}$ is known in the literature as \emph{Miyashita action}. For $R=\I$, we clearly have a morphism of groups $\Omega: \mathrm{Aut}_{alg}(A)\rightarrow \mathrm{Aut}_{\scriptscriptstyle{\Z{\I}}\text{-}alg}(\Z{A}), \,\gamma \mapsto \Z{\gamma}$, where $\rm{Aut}_{alg}(A)$ is the group of monoid automorphisms of $A$.

The main aim of the second part of the paper, that is Sections  \ref{sec:azymaya} and \ref{sec:azymayacentral}, is to give conditions under which the map $\Phi^{\I}$, or some of its factors, becomes  bijective.  Explicitly in Theorem \ref{thm:mono}, we show that $\Phi^{\I}$ is injective when the base category $\cat{M}$ is bicomplete and $\I$ is isomorphic to a specific sub-monoid of $A$. If we further assume that $A\tensor{}A \sim A$ (i.e.~$A$ and $A\tensor{}A$ are direct summand of a product of finite copies of each other as $A$-bimodules), $\Omega$ is surjective and the functor $-\tensor{}A$ is exact, then $\Phi^{\I}$ comes out to be bijective, see Theorem \ref{thm:iso} where other bijections were also established,  when $\cat{Z}$ is faithful (i.e.~the unit object $\I$ is a generator).

In Section \ref{sec:azymayacentral} we study the behaviour of $\Phi^{\I}$ in the case when $\cat{M}$ is also symmetric. In Proposition \ref{prop:Gamma}, we construct a morphism of groups $\B{\Gamma}: \rm{Inv}_{\I}(A) \to \rm{Aut}_{alg}(A)$, and show that  $\Phi^{\I}$ decomposes as $\Phi^{\I}= \B{\Gamma} \circ \Omega$.
After that, in  Corollary \ref{thm:isoAzumaya},  we show  that $\B{\Gamma}$ is bijective  for any  Azumaya  monoid $A$ whose enveloping monoid $A^e=A\tensor{}A^o$ is, as an $A$-bimodule, a direct summand of a product of  finite copies of $A$,  and always under the hypothesis that $\cat{Z}$ is faithful.

Our main application, given in Section \ref{sec:application},  deals with the category of (right)  $H$-comodules over a commutative flat Hopf algebroid  $(R,H)$, where the  base ring $R$ is assumed to be a generator.  In fact, we show that the group of $H$-automorphisms of an Azumaya $H$-comodule $R$-algebra which satisfies the above condition  is isomorphic to the group of all invertible $H$-subcomodules, see Corollary \ref{coro:Halgd}.  The particular case  of split Hopf  algebroid is one of  the best places where this application could have some  geometric meaning. For instance,  let us consider  a compact Lie group $G$ acting freely and  smoothly on a  manifold $\mathfrak{M}$. Assume that this action converts  the  ring $\mathcal{C}^{\infty}(\mathfrak{M})$ of smooth (complex valued) functions into an  $\mathscr{R}_{\mathbb{C}}(G)$-comodule $\mathbb{C}$-algebra, where   $\mathscr{R}_{\mathbb{C}}(G)$ is the commutative  Hopf $\mathbb{C}$-algebra of representative smooth functions on $G$.  Now, consider the
Hopf
algebroid $(R,H)$ with $R=\mathcal{C}^{\infty}(\mathfrak{M})$ and
$H=\mathcal{C}^{\infty}(\mathfrak{M}) \tensor{\mathbb{C}}\mathscr{R}_{\mathbb{C}}(G)$.
In this way the $R$-module of smooth global sections of  any $G$-equivariant complex vector bundle turns out to be an $H$-comodule. Hence   an  interpretation of our result in this setting can be given as follows. Take  an Azumaya  $G$-equivariant algebra bundle $(\cat{E},\theta)$ of constant rank $n^2$  and consider its $G$-equivariant enveloping algebra bundle\footnote{The fibers of $\cat{E}^o$ are the opposite algebras of the fibers of $\cat{E}$ and the action $\theta\tensor{}\theta^o$ is the obvious one.}  $(\mathcal{E}\tensor{}\mathcal{E}^o, \theta\tensor{}\theta^o)$   such that the canonical splitting $\mathcal{E}\tensor{}\mathcal{E}^o | \mathcal{E}$ in vector bundles is also  $G$-equivariant\footnote{that is the canonical monomorphism $\mathcal{E}\tensor{}\mathcal{E}^o \hookrightarrow  \mathcal{E}^{(n^2)}$ is compatible with the action of $G$, i.e.~with the $\theta$'s.}. Then we can affirm that the group of $G$-equivariant algebra automorphisms of $\cat{E}$ is isomorphic to the group of (isomorphic
classes) of all invertible $G$-equivariant sub-bundles\footnote{
These are $G$-equivariant sub-bundles $(\cat{X}, \theta_{|\cat{X}})$ such that there exists another  $G$-equivariant sub-bundle $(\cat{Y}, \theta_{|\cat{Y}})$ with $\cat{X} \tensor{}\cat{Y} \cong \mathfrak{M}\times  \mathbb{C}$ to the trivial line  bundle endowed with the trivial $G$-action.}of $\cat{E}$.
Analogous  affirmations take place in the context of affine algebraic varieties. In fact, we have an analogue result when $G$ is an affine algebraic group acting freely (and algebraically) on an affine algebraic variety $\cat{X}$, by taking the split Hopf algebroid $H=\cat{P}(\cat{X}) \tensor{}\cat{P}(G)$, where $\cat{P}(\cat{X})$ is  the commutative algebra of polynomial functions on $\cat{X}$, and $\cat{P}(G)$ is the Hopf algebra of polynomial functions on $G$.

\subsection{Basic notions,  notations and general assumptions}\label{ssec:2}

Let $\cat{M}$ be an additive category. The notation $X \in \cat{M}$ means that $X$ is an object in $\cat{M}$. The identity  arrow $\id_X$ of $X\in \cat{M}$ will be denoted by the object itself if there is no danger of misunderstanding.
The sets of morphisms are denoted by $\hom{\cat{M}}{X}{Y}$, for $X,Y \in \cat{M}$. For two functors $\cat{F}$ and $\cat{G }$ the notation $\cat{F} \dashv \cat{G}$ means that $\cat{F}$ is a left adjoint to $\cat{G}$.

\emph{In this paper $(\cat{M},\tensor{}, \I,l,r)$ is  a Penrose  monoidal abelian category, where tensor products are right exact on both factors}.
Recall from \cite[page 5825]{Bruguieres:1994} that a monoidal additive category is said to be Penrose if the abelian groups of morphisms are central $\Z{\I}$-modules, where
$\Z{\I}=\hom{\cat{M}}{\I}{\I}$ is the commutative endomorphisms ring of the identity object $\I$. Notice that a braided  monoidal additive category is always Penrose, cf. \cite[remark on page 5825]{Bruguieres:1994}. \emph{We also assume that our category $\cat{M}$ is locally small}, that is the class of subobjects of any object is a set.

For a  monoid $(R,m_R, u_R)$ (or simply $(R,m,u)$ when no confusion can be made), we denote by ${}_R\cat{M}$, $\cat{M}_R$ and ${}_R\cat{M}_R$ its categories of left $R$-modules, right $R$-modules and $R$-bimodules respectively. The category ${}_R\cat{M}_R$ inherits then  a structure of monoidal abelian (also bicomplete  if $\cat{M}$ is) category with right exact tensor products denoted by $-\tensor{R}-$; the unit object is $R$ and the left, right constraints are denoted, respectively  by $l^R$ and $r^R$. Furthermore, the forgetful functor ${}_R\cat{M}_R \to \cat{M}$ is faithful and exact.  We denote by $\mathrm{End}_{alg}(R)$ the ring of monoid endomorphisms of $R$ and by $\mathrm{Aut}_{alg}(R)$ its group of units.

Given a second monoid $S$,  we use the classical notation for morphisms of left, right and bimodules. That is, we denote by $\lhom{S}{X}{Y}$ the set of left $S$-module morphisms, by $\rhom{R}{U}{V}$  the set of right $R$-module morphisms,  and by $\hom{S,R}{P}{Q}$ the set of $(S,R)$-bimodule morphisms ($S$ on the left and $R$ on the right).
We will use the  notation $\ZR{-}$  for the functor  $\hom{R,R}{R}{-}$ and similarly for $\ZS{-}$.

An object $X$ in $\cat{M}$ is called \emph{left} (resp. \emph{right}) \emph{flat}, if the functor $-\tensor{}X : \cat{M} \to \cat{M}$ (resp. $X\tensor{}-: \cat{M} \to \cat{M}$) is left exact. Obviously, if $\cat{M}$ is symmetric, then left flat is equivalent to right flat, and the adjective left or right is omitted.

\section{Invertible bimodules and dualizable  data, revisited}\label{sec:biobjects}

Let $A, R,S$ be monoids in $\left( \mathcal{M},\otimes ,\mathbb{I}\right)$ and let $\alpha :R\rightarrow A$, $\beta:S\rightarrow A$ be  morphisms of monoids.  One can  consider in a canonical way $A$ as a monoid simultaneously in the monoidal categories of bimodules $({}_R\cat{M}_R, \tensor{R}, R,l^{\scriptscriptstyle{R}},r^{\scriptscriptstyle{R}})$ and $({}_S\cat{M}_S, \tensor{R}, S,l^{\scriptscriptstyle{S}},r^{\scriptscriptstyle{S}})$. To distinguish the multiplications of $A$ in these different categories, we use the following notations: $\mar,\mas$.

Consider $A$ as an $(R,S)$-bimodule via $\alpha$ and $\beta$. By an \emph{$(R,S)$-sub-bimodule} of $A$, we mean a pair $(X,i_X)$ where $X$ is an $(R,S)$-bimodule and $i_X: X \to A$ a monomorphism of $(R,S)$-bimodules. Consider
\begin{equation*}
\mathscr{P}\left( _{R}A_{S}\right) :=\Big\{ (R,S)\text{-sub-bimodules }\left(
X,i_{X}\right) \text{ of }_{R}A_{S}\Big\}.
\end{equation*}
Since the base category $\cat{M}$ is locally small, $\mathscr{P}\left( _{R}A_{S}\right)$ is a skeletally small category, where a morphism $f:X \to X'$ is a morphism of $(R,S)$-bimodules satisfying $i_{X'} \circ f=i_X$. We will not make a difference between the category $\mathscr{P}\left( _{R}A_{S}\right)$ and its skeleton set, that is, between an object $X$ and its representing element $(X,i_X)$. In this way,
an element (or an object) $(X,i_X) \in \mathscr{P}\left( _{R}A_{S}\right)$ will be simply denoted by $X$, where the (fixed) monomorphism of $(R,S)$-bimodules $i_{X}$ is implicitly understood.  Similar conventions and considerations are applied to the set $\mathscr{P}({}_SA_R)$.

Given $X \in  \mathscr{P}\left( _{R}A_{S}\right)$ and $Y \in \mathscr{P}\left( _{S}A_{R}\right)$, one defines
\begin{eqnarray}
f_{X} &:=&\mas\circ \left( i_{X}\tensor{S}A\right) :X\tensor{S}A\longrightarrow A,\label{Eq:fx} \\
g_{Y} &:=&\mas\circ \left( A\tensor{S}i_{Y}\right) :A\tensor{S}Y\longrightarrow A, \label{Eq:gy}
\end{eqnarray}%
Using this time the multiplication $\mar$, one analogously defines $f_Y: Y\tensor{R}A \to A$ and $g_X: A\tensor{R}X \to A$.

Recall the following two definitions which will  play a central role in this section.

\begin{definition}\label{def:inverse}
A \emph{right inverse} for $X  \in \mathscr{P}\left( _{R}A_{S}\right)  $ consists of an element $Y \in \mathscr{P}\left( _{S}A_{R}\right) $ such that
\begin{itemize}
\item there are morphisms $m_{X}=m_{X,Y}:X\tensor{S}Y\rightarrow R$ in $
_{R}\mathcal{M}_{R}$ and $m_{Y}=m_{Y,X}:Y\tensor{R}X\rightarrow S$\ in $
_{S}\mathcal{M}_{S}$ fulfilling
\begin{eqnarray}
\alpha \circ m_{X} &=&\mas\circ \left( i_{X}\tensor{S}i_{Y}\right)
\label{def: mX} \\
\beta \circ m_{Y} &=&\mar\circ \left( i_{Y}\tensor{R}i_{X}\right)
\label{def: mY}
\end{eqnarray}
\item $m_{X}$ is an isomorphism.
\end{itemize}
We also say that $X$ is a \emph{right invertible} sub-bimodule.  \emph{Left} and \emph{two-sided inverses} are obviously defined.
\end{definition}

\begin{definition}\label{def:dualizable}
An $(R,S)$-bimodule $X$ is called \emph{right dualizable} if there
exist an $(S,R)$-bimodule $Y$ and morphisms
\begin{equation*}
\mathrm{ev}:Y\tensor{R}X\rightarrow S\qquad \text{and}\qquad \mathrm{coev}%
:R\rightarrow X\tensor{S}Y,
\end{equation*}%
of $S$-bimodules and $R$-bimodules respectively, such that the following
equalities hold true
\begin{eqnarray}
\rxs\circ \left( X\tensor{S}\mathrm{ev}\right) \circ \left( \mathrm{%
coev}\tensor{R}X\right) \circ \left( \lxr\right) ^{-1} &=& {X},  \label{form:ev1} \\
\lys\circ \left( \mathrm{ev}\tensor{S}Y\right) \circ \left( Y\tensor{R}\mathrm{coev}\right) \circ \left( \ryr\right) ^{-1} &=& {Y}.  \label{form:ev2}
\end{eqnarray}%
We will also say that $\left( X,Y,\mathrm{ev},\mathrm{coev}\right) $ is a
\emph{right dualizable datum} in this case. Notice, that the same definition was given in \cite[Definition 2.1]{PareigisV} with different terminology, where condition (ii) in that definition always hods true under our assumptions. If $R=S$, then of course we have that $Y$ is a (left) dual object of $X$ in the monoidal category of bimodules ${}_{R}\cat{M}_R$. In such case, if $Y$ exists, it is unique up to isomorphism.
\end{definition}

The following standard diagrammatic notation  will be used in the sequel.
\begin{center}
\begin{tikzpicture}[x=11pt,y=11pt,thick]\pgfsetlinewidth{0.5pt}
\node[inner sep=1pt] at (0,0) {$\,\text{ev}:=\,$};
\node at (0,-1) {};
\end{tikzpicture}%
\begin{tikzpicture}[x=11pt,y=11pt,thick]\pgfsetlinewidth{0.5pt}
\node(1) at (0,2){$\scriptstyle{Y}$};
\node(2) at (2,2) {$\scriptstyle{X}$};
\draw[-] (1) to [out=-90, in=180] (1,-0.5);
\draw[-] (1,-0.5) to [out=0, in=-90] (2);
\draw[-] (1,-0.5) to  [out=-90, in=90] (1,-1);
\end{tikzpicture}$\qquad $
\begin{tikzpicture}[x=11pt,y=11pt,thick]\pgfsetlinewidth{0.5pt}
\node[inner sep=1pt] at (0,0) {$\,\text{coev}:=\,$};
\node at (0,-1) {};
\end{tikzpicture}%
\begin{tikzpicture}[x=11pt,y=11pt,thick]\pgfsetlinewidth{0.5pt}
\node(1) at (0,-2) {$\scriptstyle{X}$};
\node(2) at (2,-2) {$\scriptstyle{Y}$};
\draw[-] (1) to  [out=90,in=180]  (1,0.5);
\draw[-] (1,0.5) to  [out=0,in=90]  (2);
\draw[-] (1,0.5) to  [out=90, in=-90] (1,1);
\end{tikzpicture}
\end{center}

In terms of these diagrams, equations \eqref{form:ev1} and \eqref{form:ev2} are represented as follows (notice that several kinds of tensor products are involved)

\begin{center}
\begin{tikzpicture}[x=9pt,y=9pt,thick]\pgfsetlinewidth{0.5pt}
\node(1) at (0,4){$\scriptstyle{Y}$};
\node(2) at (3,4) {$\scriptstyle{R}$};
\node(3) at (1,-4) {$\scriptstyle{S}$};
\node(4) at (4,-4) {$\scriptstyle{Y}$};

\draw[-] (1) to [out=-90, in=90] (0,-1);
\draw[-] (0,-1) to [out=-90, in=180] (1,-2);
\draw[-] (1,-2) to  [out=0, in=-90] (2,-1);
\draw[-] (2,-1) to  [out=90, in=-90] (2,0.5);
\draw[-] (2,0.5) to  [out=90, in=180] (3,1.5);
\draw[-] (3,1.5) to  [out=0, in=90] (4,0.5);
\draw[-] (4,0.5) to  [out=-90, in=90] (4);
\draw[-] (2) to  [out=-90, in=90] (3,1.5);
\draw[-] (3) to  [out=90, in=-90] (1,-2);
\end{tikzpicture}%
\begin{tikzpicture}[x=9pt,y=9pt,thick]\pgfsetlinewidth{0.5pt}
\node[inner sep=1pt] at (0,0) {$\,=\,$};
\node at (0,-4) {};
\end{tikzpicture}%
\begin{tikzpicture}[x=9pt,y=9pt,thick]\pgfsetlinewidth{0.5pt}
\node(1) at (0,4){$\scriptstyle{Y}$};
\node(2) at (1,4) {$\scriptstyle{R}$};
\node(3) at (-1,-4) {$\scriptstyle{S}$};
\node(4) at (0,-4) {$\scriptstyle{Y}$};

\draw[-] (2) to [out=-90, in=0] (0,1);
\draw[-] (3) to [out=90, in=180] (0,-1);
\draw[-] (1) to [out=-90, in=90] (4);
\end{tikzpicture}%
\begin{tikzpicture}[x=9pt,y=9pt,thick]\pgfsetlinewidth{0.5pt}
\node[inner sep=1pt] at (0,0) {$\,=\,$};
\node at (0,-4) {};
\end{tikzpicture}%
\begin{tikzpicture}[x=9pt,y=9pt,thick]\pgfsetlinewidth{0.5pt}
\node(1) at (0,4){$\scriptstyle{Y}$};
\node(4) at (0,-4) {$\scriptstyle{Y}$};

\draw[-] (1) to [out=-90, in=90] (4);
\end{tikzpicture}$\qquad $
\begin{tikzpicture}[x=9pt,y=9pt,thick]\pgfsetlinewidth{0.5pt}
\node(1) at (0,-4){$\scriptstyle{X}$};
\node(2) at (1,4) {$\scriptstyle{R}$};
\node(3) at (3,-4) {$\scriptstyle{S}$};
\node(4) at (4,4) {$\scriptstyle{X}$};

\draw[-] (2) to [out=-90, in=90] (1,1.5);
\draw[-] (1,1.5) to [out=180, in=90] (0,0);
\draw[-] (1) to [out=90, in=-90] (0,0);
\draw[-] (1,1.5) to [out=0, in=90] (2,0);
\draw[-] (2,0) to [out=-90, in=90] (2,-1);
\draw[-] (2,-1) to [out=-90, in=180] (3,-2);
\draw[-] (3,-2) to [out=0, in=-90] (4,-1);
\draw[-] (4,-1) to [out=90, in=-90] (4);
\draw[-] (3) to  [out=90, in=-90] (3,-2);
\end{tikzpicture}%
\begin{tikzpicture}[x=9pt,y=9pt,thick]\pgfsetlinewidth{0.5pt}
\node[inner sep=1pt] at (0,0) {$\,=\,$};
\node at (0,-4) {};
\end{tikzpicture}%
\begin{tikzpicture}[x=9pt,y=9pt,thick]\pgfsetlinewidth{0.5pt}
\node(1) at (0,4){$\scriptstyle{Y}$};
\node(2) at (-1,4) {$\scriptstyle{R}$};
\node(3) at (1,-4) {$\scriptstyle{S}$};
\node(4) at (0,-4) {$\scriptstyle{Y}$};

\draw[-] (2) to [out=-90, in=180] (0,1);
\draw[-] (3) to [out=90, in=0] (0,-1);
\draw[-] (1) to [out=-90, in=90] (4);
\end{tikzpicture}%
\begin{tikzpicture}[x=9pt,y=9pt,thick]\pgfsetlinewidth{0.5pt}
\node[inner sep=1pt] at (0,0) {$\,=\,$};
\node at (0,-4) {};
\end{tikzpicture}%
\begin{tikzpicture}[x=9pt,y=9pt,thick]\pgfsetlinewidth{0.5pt}
\node(1) at (0,4){$\scriptstyle{X}$};
\node(4) at (0,-4) {$\scriptstyle{X}$};

\draw[-] (1) to [out=-90, in=90] (4);
\end{tikzpicture}
\end{center}

which we use in the following form, where, as customary, we get rid of the unit constraints.

\begin{equation}
\begin{tikzpicture}[x=8pt,y=8pt,thick]\pgfsetlinewidth{0.5pt}
\node(1) at (0,4){$\scriptstyle{Y}$};
\node(4) at (4,-4) {$\scriptstyle{Y}$};

\draw[-] (1) to [out=-90, in=90] (0,-1);
\draw[-] (2,-1) to  [out=90, in=-90] (2,0.5);
\draw[-] (0,-1) to [out=-90, in=-90] (2,-1);
\draw[-] (2,0.5) to  [out=90, in=90] (4,0.5);
\draw[-] (4,0.5) to  [out=-90, in=90] (4);
\end{tikzpicture}%
\begin{tikzpicture}[x=8pt,y=8pt,thick]\pgfsetlinewidth{0.5pt}
\node[inner sep=1pt] at (0,0) {\,=\,};
\node at (0,-4) {};
\end{tikzpicture}%
\begin{tikzpicture}[x=8pt,y=8pt,thick]\pgfsetlinewidth{0.5pt}
\node(1) at (0,4){$\scriptstyle{Y}$};
\node(4) at (0,-4) {$\scriptstyle{Y}$};

\draw[-] (1) to [out=-90, in=90] (4);
\end{tikzpicture}\qquad \qquad
\begin{tikzpicture}[x=8pt,y=8pt,thick]\pgfsetlinewidth{0.5pt}
\node(1) at (0,-4){$\scriptstyle{X}$};
\node(4) at (4,4) {$\scriptstyle{X}$};

\draw[-] (1) to [out=90, in=-90] (0,0.5);
\draw[-] (2,0.5) to [out=-90, in=90] (2,-1);
\draw[-] (0,0.5) to [out=90,in=90] (2,0.5);
\draw[-] (2,-1) to [out=-90, in=-90] (4,-1);
\draw[-] (4,-1) to [out=90, in=-90] (4);
\end{tikzpicture}%
\begin{tikzpicture}[x=8pt,y=8pt,thick]\pgfsetlinewidth{0.5pt}
\node[inner sep=1pt] at (0,0) {\,=\,};
\node at (0,-4) {};
\end{tikzpicture}%
\begin{tikzpicture}[x=8pt,y=8pt,thick]\pgfsetlinewidth{0.5pt}
\node(1) at (0,4){$\scriptstyle{X}$};
\node(4) at (0,-4) {$\scriptstyle{X}$};

\draw[-] (1) to [out=-90, in=90] (4);
\end{tikzpicture}\label{Eq:R}
\end{equation}

\subsection{Right inverse versus right dualizable datum, and adjunctions}\label{ssec:invdual}
The main aim of this subsection is to check that the right inverse, if it exists, is unique up to isomorphism. To this end, we first show that the existence of a  right inverse leads  in fact to a  right dualizable datum. The converse holds true under some more assumptions, see  subsection \ref{ssec:converse}. Secondly we show, as it might be expected, that a dualizable datum entails adjunctions.

\begin{proposition}\label{pro:invdual}
Let $Y \in  \mathscr{P}\left( _{S}A_{R}\right) $
be a right inverse of $X \in \mathscr{P}\left( _{R}A_{S}\right) $, as in Definition \ref{def:inverse}.
Set%
\begin{equation*}
\mathrm{ev}:=m_{Y}:Y\tensor{R}X\rightarrow S\qquad \text{and}\qquad
\mathrm{coev}:=\left( m_{X}\right) ^{-1}:R\rightarrow X\tensor{S}Y.
\end{equation*}%
Then
\begin{enumerate}[$(i)$]
\item $\left( X,Y,\mathrm{ev},\mathrm{coev}\right) $ is a right dualizable
datum;
\item we have that
\begin{eqnarray}
\mas\circ \left( i_{X}\tensor{S}i_{Y}\right) \circ \mathrm{coev}
&=&\alpha ,  \label{form:ev3} \\
\mar\circ \left( i_{Y}\tensor{R}i_{X}\right) &=&\beta \circ \mathrm{ev%
}\text{.}  \label{form:ev4}
\end{eqnarray}
\end{enumerate}
\end{proposition}
\begin{proof}
$(i)$.  Assume that there is another $X' \in \mathscr{P}({}_RA_S)$ with right inverse $Y' \in  \mathscr{P}({}_SA_R)$. Then, for any pair of morphisms  $(f,g) \in\mathrm{Hom}_{\scriptstyle{\mathscr{P}\left( _{R}A_{S}\right)}}\left(X',X\right) \times  \mathrm{Hom}_{\scriptstyle{\mathscr{P}\left( _{S}A_{R}\right)}}\left(Y,Y'\right)$, we define
\begin{equation}\label{Eq:phipsi}
\begin{tikzpicture}[x=11pt,y=11pt,thick]\pgfsetlinewidth{0.5pt}
\node[inner sep=1pt] at (0,-1) {$\B{\phi }(f)\,:=\,$};
\node at (0,-4) {};
\end{tikzpicture}%
\begin{tikzpicture}[x=11pt,y=11pt,thick]\pgfsetlinewidth{0.5pt}
\node(1) at (0,2.5) {$\scriptstyle{Y}$};
\node(2) at (2,-2.5) {$\scriptstyle{Y'}$};
\draw[-] (1) to [out=-90,in=90] (0,-1.5);
\node[circle,draw, inner sep=0.5pt](3) at (1,0) {$\scriptstyle{f}$};
\draw[-] (1,-0.4) to [out=-90,in=-90] (1,-1.5);
\draw[-] (0,-1.5) to [out=-90,in=-90] (1,-1.5);
\draw[-] (1,1.5) to [out=90,in=-90] (1,0.5);
\draw[-] (1,1.5) to [out=90,in=90] (2,1.5);
\draw[-] (2,1.5) to [out=-90,in=90] (2);
\end{tikzpicture}
\qquad
\begin{tikzpicture}[x=11pt,y=11pt,thick]\pgfsetlinewidth{0.5pt}
\node[inner sep=1pt] at (0,-1) {$\B{\psi}(g)\,:=\,$};
\node at (0,-4) {};
\end{tikzpicture}
\begin{tikzpicture}[x=11pt,y=11pt,thick]\pgfsetlinewidth{0.5pt}
\node(1) at (0,-2.5) {$\scriptstyle{X}$};
\node(2) at (2,2.5) {$\scriptstyle{X'}$};
\draw[-] (0,1.5) to [out=-90,in=90] (1);
\node[circle,draw, inner sep=1pt](3) at (1,0) {$\scriptstyle{g}$};
\draw[-] (1,-0.5) to [out=90,in=90] (1,-1.5);
\draw[-] (0,1.5) to [out=90,in=90] (1,1.5);
\draw[-] (1,1.5) to [out=-90,in=90] (1,0.4);
\draw[-] (1,-1.5) to [out=-90,in=-90] (2,-1.5);
\draw[-] (2) to [out=-90,in=90] (2,-1.5);
\end{tikzpicture}
\end{equation}
Explicitly, $\B{\phi}$ and $\B{\psi}$ are  given by
\begin{eqnarray}
\B{\phi}_{\scriptscriptstyle{Y,Y'}}(f)&:=&  l_{\scriptscriptstyle{Y'}}^{\scriptscriptstyle{S}}\circ \left( m_{Y,X}\tensor{S}Y'\right) \circ (Y\tensor{R} f\tensor{S} Y')\circ\left(
Y\tensor{R}\left( m_{X',Y'}\right) ^{-1}\right) \circ \left(
r_{\scriptscriptstyle{Y}}^{\scriptscriptstyle{R}}\right) ^{-1}, \label{eq: phi-ij}\\
 \B{\psi}_{\scriptscriptstyle{X',X}}(g)&:=& r_{\scriptscriptstyle{X}}^{\scriptscriptstyle{S}}\circ(X\tensor{S}m_{Y',X'})
\circ(X\tensor{S}g\tensor{R}X')\circ(m_{X,Y}^{-1}\tensor{R}X')\circ (l_{\scriptscriptstyle{X'}}^{\scriptscriptstyle{R}})^{-1}.\label{eq: psi-ij}
\end{eqnarray}
Therefore, we have
\begin{center}
\begin{tikzpicture}[x=11pt,y=11pt,thick]\pgfsetlinewidth{0.5pt}
\node[inner sep=1pt] at (0,0) {$\,i_{Y'} \circ \B{\phi}(f)\,=\,$};
\node at (0,-5) {};
\end{tikzpicture}%
\begin{tikzpicture}[x=11pt,y=11pt,thick]\pgfsetlinewidth{0.5pt}
\node(1) at (0,4.5) {$\scriptstyle{Y}$};
\node(2) at (2.5,-4.5) {$\scriptstyle{A}$};
\node[circle,draw, inner sep=0.5pt](3) at (2.5,-2) {${\scriptstyle{i_{Y'}}}$};
\node[circle,draw, inner sep=0.5pt](4) at (1,0.5) {$\scriptstyle{f}$};

\draw[-] (1) to [out=-90,in=90] (0,-1);
\draw[-] (0,-1) to [out=-90,in=-90] (1,-1);
\draw[-] (1,-1) to [out=90,in=-90] (4);
\draw[-] (4) to [out=90,in=-90] (1,2.5);
\draw[-] (1,2.5) to [out=90,in=90] (2.5,2.5);
\draw[-] (2.5,2.5) to [out=-90,in=90] (3);
\draw[-] (3) to [out=-90,in=90] (2);
\end{tikzpicture}%
\begin{tikzpicture}[x=11pt,y=11pt,thick]\pgfsetlinewidth{0.5pt}
\node[inner sep=1pt] at (0,0) {$\,=\,$};
\node at (0,-5) {};
\end{tikzpicture}%
\begin{tikzpicture}[x=11pt,y=11pt,thick]\pgfsetlinewidth{0.5pt}
\node(1) at (0,4.5) {$\scriptstyle{Y}$};
\node(2) at (1,-4.5) {$\scriptstyle{A}$};
\node[circle,draw, inner sep=0.5pt](3) at (2,-1) {$\scriptstyle{i_{Y'}}$};
\node[circle,draw, inner sep=0.5pt](4) at (0,-1) {$\scriptstyle{\beta}$};
\node[circle,draw, inner sep=0.5pt](5) at (1,2) {$\scriptstyle{f}$};

\draw[-] (1) to [out=-90,in=90] (0,0.5);
\draw[-] (0,0.5) to [out=-90,in=-90] (1,0.5);
\draw[-] (1,0.5) to [out=90,in=-90] (5);
\draw[-] (5) to [out=90,in=-90] (1,3.5);
\draw[-] (1,3.5) to [out=90,in=90] (2,3.5);
\draw[-] (2,3.5) to [out=-90,in=90] (3);
\draw[-] (2,1) to [out=-90,in=90] (3);
\draw[-] (3) to [out=-90,in=90] (2,-2.5);
\draw[-] (4) to [out=-90,in=90] (0,-2.5);
\draw[-] (2,-2.5) to [out=-90,in=-90] (0,-2.5);
\draw[-] (1,-3.1) to [out=-90,in=90] (2);
\end{tikzpicture}%
\begin{tikzpicture}[x=11pt,y=11pt,thick]\pgfsetlinewidth{0.5pt}
\node[inner sep=1pt] at (0,0) {$\,\overset{\eqref{def: mY}}{=}\,$};
\node at (0,-5) {};
\end{tikzpicture}%
\begin{tikzpicture}[x=11pt,y=11pt,thick]\pgfsetlinewidth{0.5pt}
\node[circle,draw, inner sep=0.5pt](1) at (0,1) {$\scriptstyle{i_{Y}}$};
\node[circle,draw, inner sep=0.5pt](2) at (2,1) {$\scriptstyle{i_{X}}$};
\node[circle,draw, inner sep=0.5pt](3) at (4,1) {$\scriptstyle{i_{Y'}}$};
\node[circle,draw, inner sep=0.5pt](6) at (2,3) {$\scriptstyle{f}$};
\node(4) at (0,4.5) {$\scriptstyle{Y}$};
\node(5) at (2.5,-4.5) {$\scriptstyle{A}$};

\draw[-] (4) to [out=-90,in=90] (1);
\draw[-] (1) to [out=-90,in=90] (0,-0.5);
\draw[-] (2) to [out=-90,in=90] (2,-0.5);
\draw[-] (0,-0.5) to [out=-90,in=-90] (2,-0.5);
\draw[-] (1,-1.1) to [out=-90,in=90] (1,-1.7);
\draw[-] (1,-1.7) to [out=-90,in=-90] (4,-1.7);
\draw[-] (4,-1.7) to [out=90,in=-90] (3);
\draw[-] (5) to [out=90,in=-90] (2.5,-2.6);
\draw[-] (2) to [out=90,in=-90] (6);
\draw[-] (6) to [out=90,in=-90] (2,4);
\draw[-] (2,4) to [out=90,in=90] (4,4);
\draw[-] (4,4) to [out=-90,in=90] (3);
\end{tikzpicture}%
\begin{tikzpicture}[x=11pt,y=11pt,thick]\pgfsetlinewidth{0.5pt}
\node[inner sep=1pt] at (0,0) {$\,=\,$};
\node at (0,-5) {};
\end{tikzpicture}%
\begin{tikzpicture}[x=11pt,y=11pt,thick]\pgfsetlinewidth{0.5pt}
\node[circle,draw, inner sep=0.5pt](1) at (0,1) {$\scriptstyle{i_{Y}}$};
\node[circle,draw, inner sep=0.5pt](2) at (2,1) {$\scriptstyle{i_{X'}}$};
\node[circle,draw, inner sep=0.5pt](3) at (4,1) {$\scriptstyle{i_{Y'}}$};
\node(4) at (0,4.5) {$\scriptstyle{Y}$};
\node(5) at (1.5,-4.5) {$\scriptstyle{A}$};

\draw[-] (4) to [out=-90,in=90] (1);
\draw[-] (2) to [out=90,in=-90] (2,2);
\draw[-] (3) to [out=90,in=-90] (4,2);
\draw[-] (4,2) to [out=90,in=90] (2,2);

\draw[-] (2) to [out=-90,in=90] (2,0);
\draw[-] (3) to [out=-90,in=90] (4,0);
\draw[-] (4,0) to [out=-90,in=-90] (2,0);

\draw[-] (3,-0.5) to [out=-90,in=90] (3,-2);

\draw[-] (1) to [out=-90,in=90] (0,-2);
\draw[-] (3,-2) to [out=-90,in=-90] (0,-2);
\draw[-] (5) to [out=90,in=-90] (1.5,-2.9);
\end{tikzpicture}
\begin{tikzpicture}[x=11pt,y=11pt,thick]\pgfsetlinewidth{0.5pt}
\node[inner sep=1pt] at (0,0) {$\,\overset{\eqref{def: mX}}{=}\,$};
\node at (0,-5) {};
\end{tikzpicture}%
\begin{tikzpicture}[x=11pt,y=11pt,thick]\pgfsetlinewidth{0.5pt}
\node[circle,draw, inner sep=0.5pt](1) at (0,1) {$\scriptstyle{i_{Y}}$};
\node[circle,draw, inner sep=1pt](2) at (2,1) {$\scriptstyle{\alpha}$};
\node(3) at (1,-4.5) {$\scriptstyle{A}$};
\node(4) at (0,4.5) {$\scriptstyle{Y}$};

\draw[-] (4) to [out=-90,in=90] (1);
\draw[-] (1) to [out=-90,in=90] (0,-1);
\draw[-] (2) to [out=-90,in=90] (2,-1);
\draw[-] (0,-1) to [out=-90,in=-90] (2,-1);
\draw[-] (3) to [out=90,in=-90] (1,-1.6);
\end{tikzpicture}%
\begin{tikzpicture}[x=11pt,y=11pt,thick]\pgfsetlinewidth{0.5pt}
\node[inner sep=1pt] at (0,0) {$\,=\, i_{Y}$};
\node at (0,-5) {};
\end{tikzpicture}
\end{center}
Similarly, one obtains the equality $i_{X} \circ \B{\psi}(g) = i_{X'}$. Now, for $X=X'$, $Y=Y'$ and $f=\id_X$, we get $i_Y \circ \B{\phi}(\id_X)= i_Y$ which implies that  $ \B{\phi}(\id_X)= \id_Y$, since $i_Y$ is a monomorphism. Taking now $g=\id_Y$, we obtain $ \B{\psi}(\id_Y)= \id_X$.  Both equalities $ \B{\psi}(\id_Y)= \id_X$ and $ \B{\phi}(\id_X)= \id_Y$ form precisely equation \eqref{Eq:R}, and this finishes the proof of this item.\\
$(ii)$. Equalities (\ref{form:ev3}) and (\ref{form:ev4}) are just (\ref{def: mX})
and (\ref{def: mY}) rewritten with respect to $\mathrm{ev}$ and $\mathrm{coev%
}$.
\end{proof}

The following result is inspired by the equivalence between (ii) and (iii) in \cite[Theorem 2.6]{May-Picard}, see also \cite[Proposition 2.2]{PareigisV}. For the reader's sake we give here a diagrammatic proof.

\begin{proposition}\label{prop:adjunction}
Let $\left( X,Y,\mathrm{ev},\mathrm{coev}\right) $ be
a right dualizable datum, as in Definition \ref{def:dualizable}. Then the assignments%
\begin{eqnarray*}
\phi &:&\left( f:V\tensor{R}X\longrightarrow W\right) \longmapsto \Big( \xymatrix@C=35pt{V\ar@{->}^-{( r_{V}^{R}) ^{-1}}[r] & V\tensor{R}R\ar@{->}^-{V\tensor{R}\mathrm{coev}}[r] & V\tensor{R}X\tensor{S}Y\ar@{->}^-{f\tensor{S}Y}[r] & W\tensor{S}Y }\Big) \\
 \psi &:&\left( g:V\longrightarrow W\tensor{S}Y\right) \longmapsto \Big( \xymatrix@C=35pt{V\tensor{R}X \ar@{->}^-{g\tensor{R}X}[r] & W\tensor{S}Y\tensor{R}X
\ar@{->}^-{W\tensor{S}\mathrm{ev}}[r] & W\tensor{S}S \ar@{->}^-{r_{W}^{S}}[r] &  W}\Big)
\end{eqnarray*}%
yield a natural isomorphism%
\begin{equation*}
\rhom{S}{V\tensor{R}X}{W} \cong \rhom{R}{V}{W\tensor{S}Y}.
\end{equation*}%
In other words the functor $\left( -\right) \tensor{R}X:\mathcal{M}%
_{R}\rightarrow \mathcal{M}_{S}$ is left adjoint to  the functor $%
\left( -\right) \tensor{S}Y:\mathcal{M}_{S}\rightarrow \mathcal{M}_{R}$. We also have  that $Y\tensor{R} (-):{_R\cat{M}}\rightarrow {_S\cat{M}}$ is a left adjoint of $X\tensor{S}(-):{}_{S}\mathcal{M}\rightarrow {}_{R}\mathcal{M}$.
\end{proposition}
\begin{proof} The naturality of both $\phi$ and $\psi$ is clear. Now, for $f$ as above we have%
\begin{center}
\begin{tikzpicture}[x=9pt,y=9pt,thick]\pgfsetlinewidth{0.5pt}
\node[inner sep=1pt] at (0,0) {$\,\psi\phi(f)\,=\,\,$};
\node at (0,-5) {};
\end{tikzpicture}%
\begin{tikzpicture}[x=9pt,y=9pt,thick]\pgfsetlinewidth{0.5pt}
\node(1) at (-0.7,4.5) {$\scriptstyle{V}$};
\node(2) at (-1.5,-4.5) {$\scriptstyle{W}$};
\node(3) at (2,4.5) {$\scriptstyle{X}$};
\node[rectangle,draw, inner sep=6pt, text width=18pt, text centered](4) at (-0.5,0) {$\scriptstyle{\phi(f)}$};

\draw[-] (1) to [out=-90,in=90] (-0.7,1);
\draw[-] (2) to [out=90,in=-90] (-1.5,-1);
\draw[-] (3) to [out=-90,in=90] (2,-2.5);
\draw[-] (2,-2.5) to [out=-90,in=-90] (0.3,-2.5);
\draw[-] (0.3,-2.5) to [out=90,in=-90] (0.3,-1);
\end{tikzpicture}%
\begin{tikzpicture}[x=9pt,y=9pt,thick]\pgfsetlinewidth{0.5pt}
\node[inner sep=1pt] at (0,0) {$\,=\,\;$};
\node at (0,-5) {};
\end{tikzpicture}%
\begin{tikzpicture}[x=9pt,y=9pt,thick]\pgfsetlinewidth{0.5pt}
\node(1) at (-1,4.5) {$\scriptstyle{V}$};
\node(2) at (-0.3,-4.5) {$\scriptstyle{W}$};
\node(3) at (3,4.5) {$\scriptstyle{X}$};
\node[rectangle,draw, inner sep=4pt, text width=14pt, text centered](4) at (-0.5,0) {$\scriptstyle{f}$};

\draw[-] (1) to [out=-90,in=90] (-1,0.8);
\draw[-] (2) to [out=90,in=-90] (-0.3,-0.8);
\draw[-] (3) to [out=-90,in=90] (3,-1.5);
\draw[-] (3,-1.5) to [out=-90,in=-90] (1.5,-1.5);
\draw[-] (1.5,-1.5) to [out=90,in=-90] (1.5,2.5);
\draw[-] (0,2.5) to [out=90,in=90] (1.5,2.5);
\draw[-] (0,2.5) to [out=-90,in=90] (0,0.8);
\end{tikzpicture}
\begin{tikzpicture}[x=9pt,y=9pt,thick]\pgfsetlinewidth{0.5pt}
\node[inner sep=1pt] at (0,0) {$\,\overset{\eqref{Eq:R}}{=}\;\,$};
\node at (0,-5) {};
\end{tikzpicture}%
\begin{tikzpicture}[x=9pt,y=9pt,thick]\pgfsetlinewidth{0.5pt}
\node(1) at (-1,4.5) {$\scriptstyle{V}$};
\node(2) at (0,-4.5) {$\scriptstyle{W}$};
\node(3) at (1,4.5) {$\scriptstyle{X}$};
\node[rectangle,draw, inner sep=4pt, text width=14pt, text centered](4) at (0,0) {$\scriptstyle{f}$};

\draw[-] (1) to [out=-90,in=90] (-1,0.8);
\draw[-] (2) to [out=90,in=-90] (0,-0.8);
\draw[-] (3) to [out=-90,in=90] (1,0.8);
\end{tikzpicture}
\begin{tikzpicture}[x=9pt,y=9pt,thick]\pgfsetlinewidth{0.5pt}
\node[inner sep=1pt] at (0,0) {$\,=\, f$};
\node at (0,-5) {};
\end{tikzpicture}%
\end{center}
If we reflect horizontally the diagrams above and we apply the substitutions $\psi\leftrightarrow\phi,f\mapsto g,V\leftrightarrow W,X\mapsto Y$, we get the diagrammatic proof for the equality $\phi(\psi(g))=g$. The last adjunction is similarly proved.
\end{proof}

As a consequence of Propositions \ref{pro:invdual} and \ref{prop:adjunction}, we obtain the desired uniqueness of the right inverse, since we know that a right adjoint functor is  unique up to a natural isomorphism.
\begin{corollary}\label{lem: rightinverse} Let $ X \in \mathscr{P}\left(
_{R}A_{S}\right) $. Then a right inverse of $X$, if it exists, is unique up to an isomorphism in $\mathscr{P}({}_SA_R)$.  More precisely, assume that $ Y , Y' \in
\mathscr{P}\left( _{S}A_{R}\right)$ are two right inverses of $ X$.  Then
$$
i_{\scriptscriptstyle{Y,Y'}}:= \B{\phi}_{\scriptscriptstyle{Y,Y'}} (\id_X): Y \to Y', \text{ and }\;  i_{\scriptscriptstyle{Y',Y}}:=\B{\phi}_{\scriptscriptstyle{Y',Y}} (\id_X): Y' \to Y
$$
are mutual inverses in $\mathscr{P}\left( _{S}A_{R}\right)$, where $\B{\phi}$ is as in \eqref{eq: phi-ij}.
Moreover, if $\alpha $ and $\beta$ are monomorphisms, we also have
\begin{eqnarray}
m_{Y',X}\circ \left( i_{Y,Y'}\tensor{R}X\right) &=& m_{Y,X},  \label{form:iYZ1} \\
m_{X,Y'}\circ \left( X\tensor{S}i_{Y,Y'}\right) &=&m_{X,Y}.  \label{form:iYZ2}
\end{eqnarray}
\end{corollary}
\begin{proof}
We only show the last statement. The equality  \eqref{form:iYZ1} is shown as follows:
\begin{multline*}
\beta \circ m_{Y',X}  \circ (i_{Y,Y'}\tensor{R}X)\,\overset{\eqref{def: mY}}{=}\, \mar \circ (i_{Y'}\tensor{R}i_{X}) \circ  (\B{\phi}_{\scriptscriptstyle{Y,Y'}} (\id_X)\tensor{R}X) \\ \,=\, \mar \circ \big((i_{Y'}\circ \B{\phi}_{\scriptscriptstyle{Y,Y'}} (\id_X) )\tensor{R}i_{X} \big) \,=\, \mar \circ (i_Y\tensor{R}i_X) \,\overset{\eqref{def: mY}}{=}\,  \beta \circ m_{Y,X}.
\end{multline*}
Equality \eqref{form:iYZ2} follows similarly.
\end{proof}

We finish this subsection by giving more properties of  dualizable datum which  in fact  lead to a characterization of right invertible subbimodules.

\begin{proposition}\label{pro:fX}
Let $\left( X,Y,\mathrm{ev},\mathrm{coev}\right) $ be a right
dualizable datum such that $X \in \mathscr{P}({}_RA_S)$ and $Y \in \mathscr{P}({}_SA_R)$ with associated monomorphisms $i_X, i_Y$ satisfying equations \eqref{form:ev3} and \eqref{form:ev4}.
Then the morphisms $f_X$ and $g_Y$ of equations \eqref{Eq:fx} and \eqref{Eq:gy}
are isomorphisms with inverses%
\begin{eqnarray*}
f_{X}^{-1} &=&\left( X\tensor{S}\mar\right) \circ \left( X\tensor{S}i_{Y}\tensor{R}A\right) \circ \left( \mathrm{coev}\tensor{R}A\right)
\circ \left(\lar\right) ^{-1}, \\
g_{Y}^{-1} &=&\left( \mar\tensor{S}Y\right) \circ \left( A\tensor{R}i_{X}\tensor{S}Y\right) \circ \left( A\tensor{R}\mathrm{coev}\right)
\circ \left( \rar\right) ^{-1}.
\end{eqnarray*}
\end{proposition}
\begin{proof}
Let us prove that $f_{X}$ and $\left( X\tensor{S}\mar\right) \circ
\left( X\tensor{S}i_{Y}\tensor{R}A\right) \circ \left( \mathrm{coev}%
\tensor{R}A\right) \circ \left( \lar\right) ^{-1}$ are mutual
inverses (the proof for $g_{Y}$ is analogous). One composition can be
computed by means of  the following tangle diagrams.
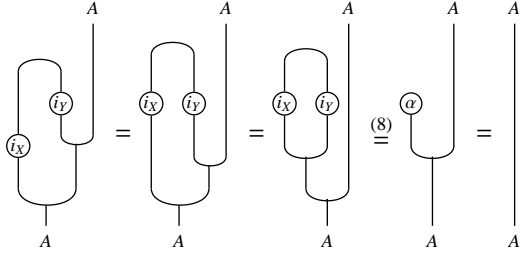
\begin{center}
\begin{tikzpicture}[x=8pt,y=8pt,thick]\pgfsetlinewidth{0.5pt}
\node[circle,draw, inner sep=0.2pt](1) at ( 1,1)  {$\scriptstyle{i_Y}$};
\node[circle,draw, inner sep=0.2pt](2) at (-1,-1) {$\scriptstyle{i_X}$};
\node(3) at (0.3,-5.5) {$\scriptstyle{A}$};
\node(4) at (2.5,5.5) { $\scriptstyle{A}$};

\draw[-]  (2) to [out=90,in=-90]  (-1,2.5);
\draw[-]  (2) to [out=-90,in=90]  (-1,-3);
\draw[-]  (3) to [out=90,in=-90]  (0.3,-3.8);
\draw[-]  (1) to [out=-90,in=90]  (1,-0.5);
\draw[-]  (4) to [out=-90,in=90]  (2.5,-0.5);
\draw[-]  (1,-0.5) to [out=-90,in=-90]  (2.5,-0.5);
\draw[-]  (1) to [out=90,in=-90]  (1,2.5);
\draw[-]  (1,2.5) to [out=90,in=90]  (-1,2.5);
\draw[-]  (1.7,-0.9) to [out=-90,in=90]  (1.7,-3);
\draw[-]  (-1,-3) to [out=-90,in=-90]  (1.7,-3);
\end{tikzpicture}%
\begin{tikzpicture}[x=8pt,y=8pt,thick]\pgfsetlinewidth{0.5pt}
\node[inner sep=1pt] at (0,0) {$\,\overset{}{=}\,$};
\node at (0,-5.5) {};
\end{tikzpicture}%
\begin{tikzpicture}[x=8pt,y=8pt,thick]\pgfsetlinewidth{0.5pt}
\node[circle,draw, inner sep=0.2pt](41) at ( 4,1)  {$\scriptstyle{i_X}$};
\node[circle,draw, inner sep=0.2pt](61) at ( 6,1)  {$\scriptstyle{i_Y}$};
\node(81) at (7.5,5.5) {$\scriptstyle{A}$};
\node(71) at (5.3,-5.5) {$\scriptstyle{A}$};

\draw[-]  (41) to [out=90,in=-90]  (4,3.3);
\draw[-]  (61) to [out=90,in=-90]  (6,3.3);
\draw[-]  (6,3.3) to [out=90,in=90]  (4,3.3);
\draw[-]  (41) to [out=-90,in=90]  (4,-3);
\draw[-]  (81) to [out=-90,in=90]  (7.5,-1.5);
\draw[-]  (6,-1.5) to [out=-90,in=-90]  (7.5,-1.5);
\draw[-]  (61) to [out=-90,in=90]  (6,-1.5);
\draw[-]  (6.7,-1.9) to [out=-90,in=90]  (6.7,-3);
\draw[-]  (6.7,-3) to [out=-90,in=-90]  (4,-3);
\draw[-]  (71) to [out=90,in=-90]  (5.3,-3.8);
\end{tikzpicture}%
\begin{tikzpicture}[x=8pt,y=8pt,thick]\pgfsetlinewidth{0.5pt}
\node[inner sep=1pt] at (0,0) {$\,\overset{}{=}\,$};
\node at (0,-5.5) {};
\end{tikzpicture}%
\begin{tikzpicture}[x=8pt,y=8pt,thick]\pgfsetlinewidth{0.5pt}
\node[circle,draw, inner sep=0.2pt](42) at ( 9,1)  {$\scriptstyle{i_X}$};
\node[circle,draw, inner sep=0.2pt](62) at ( 11,1)  {$\scriptstyle{i_Y}$};
\node(72) at (11,-5.5) {$\scriptstyle{A}$};
\node(82) at (12,5.5) {$\scriptstyle{A}$};

\draw[-]  (42) to [out=90,in=-90]  (9,3);
\draw[-]  (62) to [out=90,in=-90]  (11,3);
\draw[-]  (9,3) to [out=90,in=90]  (11,3);
\draw[-]  (42) to [out=-90,in=90]  (9,-1);
\draw[-]  (62) to [out=-90,in=90]  (11,-1);
\draw[-]  (9,-1) to [out=-90,in=-90]  (11,-1);

\draw[-]  (10,-1.5) to [out=-90,in=90]  (10,-3);

\draw[-]  (72) to [out=90,in=90]  (11,-3.6);
\draw[-]  (82) to [out=-90,in=90]  (12,-3);
\draw[-]  (12,-3) to [out=-90,in=-90]  (10,-3);
\end{tikzpicture}%
\begin{tikzpicture}[x=8pt,y=8pt,thick]\pgfsetlinewidth{0.5pt}
\node[inner sep=1pt] at (0,0) {$\,\overset{\eqref{form:ev3}}{=}\,$};
\node at (0,-5.5) {};
\end{tikzpicture}%
\begin{tikzpicture}[x=8pt,y=8pt,thick]\pgfsetlinewidth{0.5pt}
\node[circle, draw, inner sep=1pt](43) at ( 14,1)  {$\scriptstyle{\alpha}$};
\node(73) at (15,-5.5) {$\scriptstyle{A}$};
\node(83) at (16,5.5) {$\scriptstyle{A}$};

\draw[-]  (43) to [out=-90,in=90]  (14,-1);
\draw[-]  (73) to [out=90,in=-90]  (15,-1.5);
\draw[-]  (14,-1) to [out=-90,in=-90]  (16,-1);
\draw[-]  (16,-1) to [out=90,in=-90]  (83);
\end{tikzpicture}%
\begin{tikzpicture}[x=8pt,y=8pt,thick]\pgfsetlinewidth{0.5pt}
\node[inner sep=1pt] at (0,0) {$\,\overset{}{=}\,$};
\node at (0,-5.5) {};
\end{tikzpicture}%
\begin{tikzpicture}[x=8pt,y=8pt,thick]\pgfsetlinewidth{0.5pt}
\node(74) at (18,-5.5) {$\scriptstyle{A}$};
\node(84) at (18,5.5) {$\scriptstyle{A}$};
\draw[-]  (74) to [out=90,in=-90]  (84);
\end{tikzpicture}
\end{center}

The other composition is computed as follows.
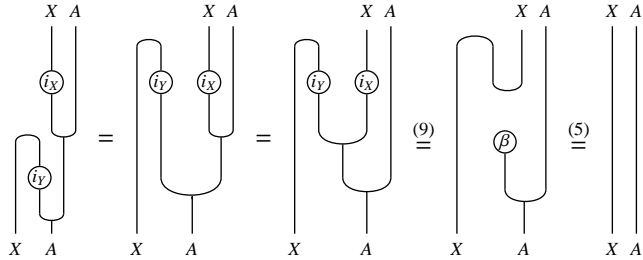
\begin{center}
\begin{tikzpicture}[x=9pt,y=9pt,thick]\pgfsetlinewidth{0.5pt}
\node(1) at (0,-5) {$\scriptstyle{X}$};
\node(3) at (1.5,-5) {$\scriptstyle{A}$};
\node[circle,draw, inner sep=0.2pt](2) at (1,-2) {$\scriptstyle{i_Y}$} ;
\node[circle,draw, inner sep=0.2pt](4) at (1.5,2) {$\scriptstyle{i_X}$} ;
\node(5) at (2.5,5) {$\scriptstyle{A}$};
\node(6) at (1.5,5) {$\scriptstyle{X}$};

\draw[-] (6) to [out=-90, in=90] (4);
\draw[-] (4) to [out=-90, in=90] (1.5, 0);
\draw[-] (1.5,0) to [out=-90, in=-90] (2.5,0);
\draw[-] (5) to [out=-90, in=90] (2.5,0);

\draw[-] (2,-0.3) to [out=-90, in=90] (2,-3.5);
\draw[-] (2,-3.5) to [out=-90, in=-90] (1,-3.5);
\draw[-] (1,-3.5) to [out=90, in=-90] (2);
\draw[-] (2) to [out=90, in=-90] (1,-0.5);
\draw[-] (1,-0.5) to [out=90, in=90] (0,-0.5);
\draw[-] (0,-0.5) to [out=-90, in=90] (1);
\draw[-] (3) to [out=90, in=-90] (1.5,-3.8);
\end{tikzpicture}%
\begin{tikzpicture}[x=9pt,y=9pt,thick]\pgfsetlinewidth{0.5pt}
\node[inner sep=1pt] at (0,0) {$\,\overset{}{=}\,$};
\node at (0,-5) {};
\end{tikzpicture}%
\begin{tikzpicture}[x=9pt,y=9pt,thick]\pgfsetlinewidth{0.5pt}
\node(7) at (0,-5) {$\scriptstyle{X}$};
\node[circle,draw, inner sep=0.2pt](8) at (1,2) {$\scriptstyle{i_Y}$} ;
\node(9) at (2.3,-5) {$\scriptstyle{A}$};
\node[circle,draw, inner sep=0.2pt](10) at (3,2) {$\scriptstyle{i_X}$} ;
\node(51) at (4,5) {$\scriptstyle{A}$};
\node(61) at (3,5) {$\scriptstyle{X}$};

\draw[-] (7) to [out=90, in=-90] (0,3.5);
\draw[-] (0,3.5) to [out=90, in=90] (1,3.5);
\draw[-] (8) to [out=90, in=-90] (1,3.5);
\draw[-] (8) to [out=-90, in=90] (1,-2);
\draw[-] (9) to [out=90, in=90] (2.3,-2.9);
\draw[-] (3.5,-2) to [out=-90, in=-90] (1,-2);
\draw[-] (3.5,-2) to [out=90, in=-90] (3.5,-0.3);
\draw[-] (10) to [out=-90, in=90] (3,0);
\draw[-] (3,0) to [out=-90, in=-90] (4,0);
\draw[-] (51) to [out=-90, in=90] (4,0);
\draw[-] (10) to [out=90, in=-90] (61);
\end{tikzpicture}%
\begin{tikzpicture}[x=9pt,y=9pt,thick]\pgfsetlinewidth{0.5pt}
\node[inner sep=1pt] at (0,0) {$\,\overset{}{=}\,$};
\node at (0,-5) {};
\end{tikzpicture}%
\begin{tikzpicture}[x=9pt,y=9pt,thick]\pgfsetlinewidth{0.5pt}
\node(71) at (0,-5) {$\scriptstyle{X}$};
\node[circle,draw, inner sep=0.2pt](82) at (1,2) {$\scriptstyle{i_Y}$} ;
\node[circle,draw, inner sep=0.2pt](101) at (3,2) {$\scriptstyle{i_X}$} ;
\node(52) at (3,-5) {$\scriptstyle{A}$};
\node(62) at (3,5) {$\scriptstyle{X}$};
\node(92) at (4,5) {$\scriptstyle{A}$};

\draw[-] (71) to [out=90, in=-90] (0,3.5);
\draw[-] (0,3.5) to [out=90, in=90] (1,3.5);
\draw[-] (1,3.5) to [out=-90, in=90] (82);
\draw[-] (82) to [out=-90, in=90] (1,0);
\draw[-] (101) to [out=-90, in=90] (3,0);
\draw[-] (1,0) to [out=-90, in=-90] (3,0);
\draw[-] (2,-0.6) to [out=-90, in=90] (2,-2);
\draw[-] (52) to [out=90, in=-90] (3,-2.6);
\draw[-] (92) to [out=-90, in=90] (4,-2);
\draw[-] (4,-2) to [out=-90, in=-90] (2,-2);
\draw[-] (101) to [out=90, in=-90] (3,4.2);
\end{tikzpicture}%
\begin{tikzpicture}[x=9pt,y=9pt,thick]\pgfsetlinewidth{0.5pt}
\node[inner sep=1pt] at (0,0) {$\,\overset{\eqref{form:ev4}}{=}\,$};
\node at (0,-5) {};
\end{tikzpicture}%
\begin{tikzpicture}[x=9pt,y=9pt,thick]\pgfsetlinewidth{0.5pt}
\node(1) at (0,-5) {$\scriptstyle{X}$};
\node(2) at (2.7,5) {$\scriptstyle{X}$};
\node(3) at (3.7,5) {$\scriptstyle{A}$};
\node(4) at (2.8,-5) {$\scriptstyle{A}$};
\node[circle,draw, inner sep=0.5pt](5) at (2,-0.5) {$\scriptstyle{\beta}$} ;

\draw[-] (1) to [out=90, in=-90] (0,3.5);
\draw[-] (0,3.5) to [out=90, in=90] (1.5,3.5);
\draw[-] (1.5,2) to [out=-90, in=-90] (2.7,2);
\draw[-] (2.7,2) to [out=90, in=-90] (2.7,4.2);
\draw[-] (1.5,2) to [out=90, in=-90] (1.5,3.5);
\draw[-] (5) to [out=-90, in=90] (2,-2.5);
\draw[-] (4) to [out=90, in=-90] (2.8,-3);
\draw[-] (3.7,-2.5) to [out=-90, in=-90] (2,-2.5);
\draw[-] (3) to [out=-90, in=90] (3.7,-2.5);
\end{tikzpicture}%
\begin{tikzpicture}[x=9pt,y=9pt,thick]\pgfsetlinewidth{0.5pt}
\node[inner sep=1pt] at (0,0) {$\,\overset{\eqref{form:ev1}}{=}\,$};
\node at (0,-5) {};
\end{tikzpicture}%
\begin{tikzpicture}[x=9pt,y=9pt,thick]\pgfsetlinewidth{0.5pt}
\node(1) at (0,-5) {$\scriptstyle{X}$};
\node(2) at (0,5) {$\scriptstyle{X}$};
\node(3) at (1,5) {$\scriptstyle{A}$};
\node(4) at (1,-5) {$\scriptstyle{A}$};
\draw[-] (1) to [out=90, in=-90] (2);
\draw[-] (3) to [out=-90, in=90] (4);
\end{tikzpicture}
\end{center}
\end{proof}

Combining Proposition \ref{pro:invdual} with Proposition  \ref{pro:fX}, we obtain
\begin{corollary}\label{coro:conaq}
Let $Y \in  \mathscr{P}\left( _{S}A_{R}\right) $
be a right inverse of $X \in \mathscr{P}\left( _{R}A_{S}\right) $. Then the morphisms
$f_{X}: X\tensor{S}A \to A$ and $g_Y:A\tensor{S}Y \to A$ defined in \eqref{Eq:fx} and \eqref{Eq:gy} are isomorphisms.
\end{corollary}

\begin{remark}\label{rem:converse} Corollary \ref{coro:conaq} says that, if $X \in\mathscr{P}\left( _{R}A_{S}\right) $ has a right inverse, then $f_X$ is an isomorphism. Implicitly, an analogue statement holds true for left invertible elements in $\mathscr{P}\left( _{S}A_{R}\right)$.  At this level of generality the converse  of this implication is not at all a  trivial question.
However,  once assumed that $X$ belongs to  a right dualizable datum $(X, Y, {\rm ev}, {\rm coev})$ with $Y \in \mathscr{P}({}_RA_S)$ such that $f_X$ is an isomorphism and that  $\text{ev}, \text{coev}$ satisfy equations \eqref{form:ev3} and \eqref{form:ev4}, then one can construct a new sub-monoid $R'$ of $A$ in such a way that $X \in \mathscr{P}\left( _{R'}A_{S}\right)$ is right invertible. The complete   proof  of this fact  is included in  the Appendix, see Theorem \ref{thm:Rp}.
On the other hand,  one can define a right  invertible $(R,S)$-sub-bimodule as a sub-bimodule $X$ of $A$ whose associated morphism $f_X$ is an isomorphism. This was in fact the approach adopted in  \cite{Gomez-Mesablishvili, Masuoka-Corings, Kaoutit-Gomez}. It is noteworthy to mention that  these arguments show in fact that our approach runs in a different direction.
\end{remark}

\subsection{The group of invertible subbimodules}\label{ssec:invsubbimod} Before introducing the set of two-sided invertible subbimodules, which will be our main object of study in the forthcoming  sections, this lemma is needed:
\begin{lemma}\label{lem: rightinverse-gen}
Let $X,X' \in \mathscr{P}\left(
_{R}A_{S}\right) $ have right  inverses  $ Y, Y'  \in
\mathscr{P}\left( _{S}A_{R}\right) $ respectively, as in Definition \ref{def:inverse}. Then the  maps $\B{\phi}_{\scriptscriptstyle{Y,Y'}}$ and
$\B{\psi}_{\scriptscriptstyle{X',X}}$ given by the formulae  \eqref{eq: phi-ij}, \eqref{eq: psi-ij}, yield an isomorphism%
\begin{equation*}
\mathrm{Hom}_{\scriptscriptstyle{\mathscr{P}\left(
_{R}A_{S}\right)}}\left(X',X\right) \cong \mathrm{Hom}{\scriptscriptstyle{\mathscr{P}\left(
_{S}A_{R}\right)}}\left(Y,Y'\right) .
\end{equation*}%
Moreover, we have
\begin{gather*}
\B{\phi}_{\scriptscriptstyle{X,X}}(\id_{X})=\id_{Y},\quad \B{\psi}_{\scriptscriptstyle{Y,Y}}(\id_{Y})=\id_{X},\\
\B{\phi}(f)\circ\B{\phi}(f')=\B{\phi}(f'\circ f),\quad\B{\psi}(g)\circ\B{\psi}(g')=\B{\psi}(g'\circ g),
\end{gather*}
where in the last two equations different $\B{\phi}$'s and $\B{\psi}$'s are employed.
\end{lemma}
\begin{proof}
 First note that the required isomorphism is obtained, using the adjunctions of Proposition \ref{prop:adjunction}, as follows:
 \begin{equation*}
\mathrm{Hom}_{\scriptscriptstyle{R,S}}\left(X',X\right) \cong\mathrm{Hom}_{\scriptscriptstyle{R,S}}\left(R\tensor{R}X',X\right)
\cong\mathrm{Hom}_{\scriptscriptstyle{R,R}}\left(R,X\tensor{S}Y'\right) \cong\mathrm{Hom}_{\scriptscriptstyle{S,R}}\left(Y\tensor{R}R,Y'\right)\cong \mathrm{Hom}_{\scriptscriptstyle{S,R}}\left(Y,Y'\right) .
\end{equation*}
By \eqref{Eq:R}, it is clear that $\B{\phi}_{\scriptscriptstyle{X,X}}(\id_{X})=\id_{Y}$ and $\B{\psi}_{\scriptscriptstyle{Y,Y}}(\id_{Y})=\id_{X}.$
The following diagrams show that
\begin{center}
\begin{tikzpicture}[x=8pt,y=8pt,thick]\pgfsetlinewidth{0.5pt}
\node[inner sep=1pt] at (0,0) {$\,\B{\phi}(f) \circ \B{\phi}(f')\,=\,$};
\node at (0,-5) {};
\end{tikzpicture}%
\begin{tikzpicture}[x=8pt,y=8pt,thick]\pgfsetlinewidth{0.5pt}
\node(1) at (0,4.5) {$\scriptstyle{Y}$};
\node(2) at (6,-4.5) {$\scriptstyle{Y}$};
\node[circle,draw, inner sep=0.5pt](4) at (1.5,0.5) {$\scriptstyle{f'}$};
\node[circle,draw, inner sep=0.5pt](5) at (4.5,-1.5) {$\scriptstyle{f}$};

\draw[-] (1) to [out=-90,in=90] (0,-1);
\draw[-] (0,-1) to [out=-90,in=-90] (1.5,-1);
\draw[-] (1.5,-1) to [out=90,in=-90] (4);
\draw[-] (4) to [out=90,in=-90] (1.5,2.5);
\draw[-] (1.5,2.5) to [out=90,in=90] (3,2.5);
\draw[-] (3,2.5) to [out=-90,in=90] (3,-3.5);
\draw[-] (3,-3.5) to [out=-90,in=-90] (4.5,-3.5);
\draw[-] (5) to [out=-90,in=90] (4.5,-3.5);
\draw[-] (5) to [out=90,in=-90] (4.5,0.5);
\draw[-] (4.5,0.5) to [out=90,in=90] (6,0.5);
\draw[-] (6,0.5) to [out=-90,in=90] (2);
\end{tikzpicture}%
\begin{tikzpicture}[x=8pt,y=8pt,thick]\pgfsetlinewidth{0.5pt}
\node[inner sep=1pt] at (0,0) {$\,\overset{\eqref{Eq:R}}{=}\,$};
\node at (0,-5) {};
\end{tikzpicture}%
\begin{tikzpicture}[x=8pt,y=8pt,thick]\pgfsetlinewidth{0.5pt}
\node(1) at (0,4.5) {$\scriptstyle{Y'}$};
\node(2) at (3,-4.5) {$\scriptstyle{Y'}$};
\node[circle,draw, inner sep=0.5pt](4) at (1.5,-2) {$\scriptstyle{f'}$};
\node[circle,draw, inner sep=0.5pt](5) at (1.5,1) {$\scriptstyle{f}$};

\draw[-] (1) to [out=-90,in=90] (0,-4);
\draw[-] (0,-4) to [out=-90,in=-90] (1.5,-4);
\draw[-] (4) to [out=90,in=-90] (5);
\draw[-] (1.5,-4) to [out=90,in=-90] (4);
\draw[-] (5) to [out=90,in=-90] (1.5,3);
\draw[-] (1.5,3) to [out=90,in=90] (3,3);
\draw[-] (3,3) to [out=-90,in=90] (2);
\end{tikzpicture}
\begin{tikzpicture}[x=8pt,y=8pt,thick]\pgfsetlinewidth{0.5pt}
\node[inner sep=1pt] at (0,0) {$\,=\, \B{\phi}(f'\circ f)$.};
\node at (0,-5) {};
\end{tikzpicture}%
\end{center}
If we reflect horizontally the diagrammatic proof above we get $\B{\psi}(g)\circ\B{\psi}(g')=\B{\psi}(g'\circ g)$.
\end{proof}

\begin{definition}\label{def:invrs}
Let $(A,m,u)$ be a monoid in $\cat{M}$ and $\alpha: R \rightarrow A \leftarrow S: \beta$ two morphisms of monoids. We define
\begin{equation*}
\mathrm{Inv}_{R,S}^{r}\left( A\right) :=\Big\{ X \in
\mathscr{P}\left( _{R}A_{S}\right) \mid X \text{ has a right
inverse, as in Definition \ref{def:inverse}}\Big\} .
\end{equation*}%
Similarly, one can define $\mathrm{Inv}_{R,S}^{l}(A)$;  interchanging $R$ by $S$, one defines $\mathrm{Inv}_{S,R}^{r}(A)$ and $\mathrm{Inv}_{S,R}^{l}(A)$.
\end{definition}
The following lemma establishes a functorial relation between the left and right hand sides versions of the previous definition.
\begin{lemma}\label{lema:leftright}
Consider $\mathrm{Inv}_{R,S}^{r}\left( A\right)$ and $\mathrm{Inv}_{S,R}^{l}\left( A\right)$ as full sub-categories of $\mathscr{P}({}_RA_S)$ and $\mathscr{P}({}_SA_R)$, respectively. Then there is an isomorphism of categories
$$ \B{\phiup}: \mathrm{Inv}_{R,S}^{r}\left( A\right) \overset{\cong}{\longrightarrow} \mathrm{Inv}_{S,R}^{l}\left( A\right)^{\scriptscriptstyle{op}}.$$
\end{lemma}
\begin{proof}
First notice that both $\mathrm{Inv}_{R,S}^{r}\left( A\right)$ and $\mathrm{Inv}_{S,R}^{l}\left( A\right)$ are regarded as skeletal categories, in the sense that any two isomorphic objects are identical \cite[page 91]{MacLane-Categories}.
Therefore in view of Lemma \ref{lem: rightinverse}, if $X \in
\mathscr{P}\left( _{R}A_{S}\right) $ has a right inverse, then this inverse is unique and denoted by $X^{r}$. Thus to each object $X \in \mathrm{Inv}_{R,S}^{r}\left( A\right)$ it corresponds a unique object $X^{r} \in \mathrm{Inv}_{S,R}^{l}\left( A\right)$.  This establishes a bijection at the level of objects. By Lemma \ref{lem: rightinverse-gen} we know that the maps $\B{\phi}$ of equation \eqref{eq: phi-ij} induce
isomorphisms $\mathrm{Hom}_{\scriptscriptstyle{\mathrm{Inv}_{R,S}^{r}\left( A\right)}}\left(X,U\right) \cong \mathrm{Hom}{\scriptscriptstyle{\mathrm{Inv}_{S,R}^{l}\left( A\right)}}\left(U^r,X^r\right)$, which, by the last equations of that Lemma, give the desired contravariant category isomorphism $\B{\phiup}$.
\end{proof}

The image of an element $X \in \mathrm{Inv}^r_{R,S}\left( A\right)$ by the functor $\B{\phiup}$ of Lemma \ref{lema:leftright} will be denoted by $X^r$.  The set of two-sided invertible sub-bimodules is then defined in the following way:
\begin{equation}\label{Eq:inv}
\mathrm{Inv}_{R,S}\left( A\right) :=
\Big\{  X \in
\mathrm{Inv}_{R,S}^{r}\left( A\right) \mid m_{X^{r}}\text{ is an isomorphism, see Definition \ref{def:inverse}}\Big\} .
\end{equation}
One shows the  following equivalent description of this set
$$
\mathrm{Inv}_{R,S}\left( A\right) \,=\,  \mathrm{Inv}_{R,S}^{r}\left( A\right) \bigcap  \mathrm{Inv}_{R,S}^{l}\left( A\right),
$$
where the intersection is that of two subsets of $\mathscr{P}({}_RA_S)$.

Let $X\in \mathscr{P}\left( _{R}A_{S}\right)$ and $X' \in\mathscr{P}\left( _{S}A_{T}\right)$, where $\gamma: T \to A$ is another morphism of monoids. The image of the morphism $$\mas \circ(i_X\tensor{S}i_{X'}):X\tensor{S}X'\rightarrow A,$$ will be denoted by $(XX',i_{XX'}:XX'\rightarrow A)$, which will be considered as an element in $\mathscr{P}({}_RA_T)$.

\begin{proposition}\label{pro: InvTens} Let $\alpha:R\rightarrow A,\beta:S\rightarrow A$ and $\gamma:T\rightarrow A$ be morphisms of monoids in $\cat{M}$.  Let $X \in \mathrm{Inv}^r_{R,S}\left( A\right)$ and $X'\in\mathrm{Inv}^r_{S,T}\left( A\right)$ with right inverses $Y$ and $Y'$ respectively. Then
 \begin{enumerate}
  \item $\mas \circ(i_X\tensor{S}i_{X'})$ and $\mas\circ(i_{Y'}\tensor{S}i_Y)$ are monomorphisms i.e. $X\tensor{S}X'\cong XX'$ and $Y'\tensor{S}Y\cong Y'Y$, as $(R,T)$-bimodules.
  \item $X\tensor{S}X'\in \mathrm{Inv}^r_{R,T}\left( A\right)$ and its right inverse is $Y'\tensor{S}Y$.
  \end{enumerate}
Moreover, we have a functor $\mathrm{Inv}^r_{R,S}\left( A\right) \times  \mathrm{Inv}^r_{S,T}\left( A\right) \to  \mathrm{Inv}^r_{R,T}\left( A\right)$.
\end{proposition}
\begin{proof}
 $(1)$. By Proposition \ref{prop:adjunction}, the functor $X\tensor{S}(-)$ is a right adjoint and hence left exact so that $X\tensor{S}i_{X'}$ is a monomorphism. By Corollary \ref{coro:conaq}, the morphism $f_{X}: X\tensor{S}A\rightarrow A$ is an isomorphism. Thus $\mas \circ(i_X\tensor{S}i_{X'})=f_X\circ (X\tensor{S}i_{X'}):X\tensor{S}X'\rightarrow A$ is a monomorphism. Similarly, using the fact that the functor $(-)\tensor{S}Y$ is a right adjoint (Proposition \ref{prop:adjunction}) and that $g_Y:=m_A^S\circ(A\tensor{S}i_Y)$ is an isomorphism (Corollary \ref{coro:conaq}) one gets that $\mas\circ(i_{Y'}\tensor{S}i_Y)$ is a monomorphism too.

 $(2)$. Define $m_{X\tensor{S}X'}$ and $m_{Y'\tensor{S}Y}$ diagrammatically by setting
\begin{center}
\begin{tikzpicture}[x=9pt,y=9pt,thick]\pgfsetlinewidth{0.5pt}
\node[inner sep=1pt] at (0,0) {$\,m_{X\tensor{S}X'}\,=\,$};
\node at (0,-1) {};
\end{tikzpicture}%
\begin{tikzpicture}[x=9pt,y=9pt,thick]\pgfsetlinewidth{0.5pt}
\node(1) at (-3,2) {$\scriptstyle{X}$};
\node(2) at (-1.5,2) {$\scriptstyle{X'}$};
\node(3) at (1.5,2) {$\scriptstyle{Y'}$};
\node(4) at (3,2) {$\scriptstyle{Y}$};

\draw[-] (1) to [out=-90,in=90] (-3, 1);
\draw[-] (4) to [out=-90,in=90] (3, 1);
\draw[-] (3,1) to [out=-90,in=-90] (-3, 1);

\draw[-] (2) to [out=-90,in=90] (-1.5, 1);
\draw[-] (3) to [out=-90,in=90] (1.5, 1);
\draw[-] (1.5,1) to [out=-90,in=-90] (-1.5, 1);
\end{tikzpicture} $\qquad$
\begin{tikzpicture}[x=9pt,y=9pt,thick]\pgfsetlinewidth{0.5pt}
\node[inner sep=1pt] at (0,0) {$\,m_{Y'\tensor{S}Y}\,=\,$};
\node at (0,-1) {};
\end{tikzpicture}%
\begin{tikzpicture}[x=9pt,y=9pt,thick]\pgfsetlinewidth{0.5pt}
\node(1) at (-3,2) {$\scriptstyle{Y'}$};
\node(2) at (-1.5,2) {$\scriptstyle{Y}$};
\node(3) at (1.5,2) {$\scriptstyle{X}$};
\node(4) at (3,2) {$\scriptstyle{X'}$};

\draw[-] (1) to [out=-90,in=90] (-3, 1);
\draw[-] (4) to [out=-90,in=90] (3, 1);
\draw[-] (3,1) to [out=-90,in=-90] (-3, 1);

\draw[-] (2) to [out=-90,in=90] (-1.5, 1);
\draw[-] (3) to [out=-90,in=90] (1.5, 1);
\draw[-] (1.5,1) to [out=-90,in=-90] (-1.5, 1);

\end{tikzpicture}
\end{center}

Now, using equation \eqref{def: mX} for $X'$ and $X$, one gets the same equality for $X\tensor{S}X'$ as follows.
\begin{center}
\begin{tikzpicture}[x=9pt,y=9pt,thick]\pgfsetlinewidth{0.5pt}
\node(1) at (-3.5,5) {$\scriptstyle{X}$};
\node(2) at (-1.5,5) {$\scriptstyle{X'}$};
\node(3) at (1.5,5) {$\scriptstyle{Y'}$};
\node(4) at (3.5,5) {$\scriptstyle{Y}$};
\node(5) at (0,-5) {$\scriptstyle{A}$};
\node[circle, draw,inner sep=0.5pt](11) at (-3.5,3) {$\scriptstyle{\,i_{X}}$};
\node[circle, draw,inner sep=0.5pt](22) at (-1.5,3) {$\scriptstyle{\,i_{X'}}$};
\node[circle, draw,inner sep=0.5pt](33) at (1.5,3) {$\scriptstyle{\,i_{Y'}}$};
\node[circle, draw,inner sep=0.5pt](44) at (3.5,3) {$\scriptstyle{\,i_{Y}}$};

\draw[-] (1) to [out=-90,in=90] (11);
\draw[-] (2) to [out=-90,in=90] (22);
\draw[-] (3) to [out=-90,in=90] (33);
\draw[-] (4) to [out=-90,in=90] (44);
\draw[-] (11) to [out=-90,in=90] (-3.5,1);
\draw[-] (22) to [out=-90,in=90] (-1.5,1);
\draw[-] (33) to [out=-90,in=90] (1.5,1);
\draw[-] (44) to [out=-90,in=90] (3.5,1);
\draw[-] (-3.5,1) to [out=-90,in=-90] (-1.5,1);
\draw[-] (3.5,1) to [out=-90,in=-90] (1.5,1);
\draw[-] (-2.5,-1) to [out=90,in=-90] (-2.5,0.4);
\draw[-] (2.5,-1) to [out=90,in=-90] (2.5,0.4);
\draw[-] (-2.5,-1) to [out=-90,in=-90] (2.5,-1);
\draw[-] (5) to [out=90,in=-90] (0,-2.4);
\end{tikzpicture}
\begin{tikzpicture}[x=9pt,y=9pt,thick]\pgfsetlinewidth{0.5pt}
\node[inner sep=1pt] at (0,0) {$\,=\,$};
\node at (0,-5) {};
\end{tikzpicture}%
\begin{tikzpicture}[x=9pt,y=9pt,thick]\pgfsetlinewidth{0.5pt}
\node(1) at (-3.5,5) {$\scriptstyle{X}$};
\node(2) at (-1.5,5) {$\scriptstyle{X'}$};
\node(3) at (1.5,5) {$\scriptstyle{Y'}$};
\node(4) at (3.5,5) {$\scriptstyle{Y}$};
\node(5) at (-0.5,-5) {$\scriptstyle{A}$};
\node[circle, draw,inner sep=0.5pt](11) at (-3.5,3) {$\scriptstyle{\,i_{X}}$};
\node[circle, draw,inner sep=0.5pt](22) at (-1.5,3) {$\scriptstyle{\,i_{X'}}$};
\node[circle, draw,inner sep=0.5pt](33) at (1.5,3) {$\scriptstyle{i_{Y'}}$};
\node[circle, draw,inner sep=0.5pt](44) at (3.5,3) {$\scriptstyle{\,i_{Y}}$};

\draw[-] (1) to [out=-90,in=90] (11);
\draw[-] (2) to [out=-90,in=90] (22);
\draw[-] (3) to [out=-90,in=90] (33);
\draw[-] (4) to [out=-90,in=90] (44);
\draw[-] (11) to [out=-90,in=90] (-3.5,-2);
\draw[-] (22) to [out=-90,in=90] (-1.5,2);
\draw[-] (33) to [out=-90,in=90] (1.5,2);
\draw[-] (44) to [out=-90,in=90] (3.5,0);
\draw[-] (-1.5,2) to [out=-90,in=-90] (1.5,2);
\draw[-] (0,0) to [out=90,in=-90] (0,1);
\draw[-] (0,0) to [out=-90,in=-90] (3.5,0);
\draw[-] (1.7,-2) to [out=90,in=-90] (1.7,-1);
\draw[-] (1.7,-2) to [out=-90,in=-90] (-3.5,-2);
\draw[-] (5) to [out=90,in=-90] (-0.5,-3.5);
\end{tikzpicture}
\begin{tikzpicture}[x=9pt,y=9pt,thick]\pgfsetlinewidth{0.5pt}
\node[inner sep=1pt] at (0,0) {$\,=\,$};
\node at (0,-5) {};
\end{tikzpicture}%
\begin{tikzpicture}[x=9pt,y=9pt,thick]\pgfsetlinewidth{0.5pt}
\node(1) at (-3.5,5) {$\scriptstyle{X}$};
\node(2) at (-1.5,5) {$\scriptstyle{X'}$};
\node(3) at (1.5,5) {$\scriptstyle{Y'}$};
\node(4) at (3.5,5) {$\scriptstyle{Y}$};
\node(5) at (-0.7,-5) {$\scriptstyle{A}$};
\node[circle, draw,inner sep=0.5pt](11) at (-3.5,3) {$\scriptstyle{\,i_{X}}$};
\node[circle, draw,inner sep=0.5pt](6) at (0,2) {$\scriptstyle{\beta}$};
\node[circle, draw,inner sep=0.5pt](44) at (3.5,3) {$\scriptstyle{\,i_{Y}}$};

\draw[-] (1) to [out=-90,in=90] (11);

\draw[-] (2) to [out=-90,in=90] (-1.5,4);
\draw[-] (3) to [out=-90,in=90] (1.5,4);
\draw[-] (-1.5, 4) to [out=-90,in=-90] (1.5,4);

\draw[-] (4) to [out=-90,in=90] (44);
\draw[-] (11) to [out=-90,in=90] (-3.5,-1.5);
\draw[-] (44) to [out=-90,in=90] (3.5,0);
\draw[-] (0,0) to [out=90,in=-90] (6);
\draw[-] (0,0) to [out=-90,in=-90] (3.5,0);
\draw[-] (1.7,-1.5) to [out=90,in=-90] (1.7,-1);
\draw[-] (1.7,-1.5) to [out=-90,in=-90] (-3.5,-1.5);
\draw[-] (5) to [out=90,in=-90] (-0.7,-3);
\end{tikzpicture}
\begin{tikzpicture}[x=9pt,y=9pt,thick]\pgfsetlinewidth{0.5pt}
\node[inner sep=1pt] at (0,0) {$\,=\,$};
\node at (0,-5) {};
\end{tikzpicture}%
\begin{tikzpicture}[x=9pt,y=9pt,thick]\pgfsetlinewidth{0.5pt}
\node(1) at (-3.5,5) {$\scriptstyle{X}$};
\node(2) at (-1.5,5) {$\scriptstyle{X'}$};
\node(3) at (1.5,5) {$\scriptstyle{Y'}$};
\node(4) at (3.5,5) {$\scriptstyle{Y}$};
\node(5) at (0,-5) {$\scriptstyle{A}$};
\node[circle, draw,inner sep=1pt](6) at (0,0) {$\scriptstyle{\alpha}$};

\draw[-] (1) to [out=-90,in=90] (-3.5,4);
\draw[-] (4) to [out=-90,in=90] (3.5,4);
\draw[-] (-3.5,4) to [out=-90,in=-90] (3.5,4);
\draw[-] (2) to [out=-90,in=-90] (3);
\draw[-] (5) to [out=90,in=-90] (6);
\end{tikzpicture}
\end{center}

The same diagrammatic proof, once applied the substitutions $X\leftrightarrow Y',Y\leftrightarrow X'$ and $\alpha\leftrightarrow\gamma$, yields \eqref{def: mY} for $Y'\tensor{S}Y$ using the corresponding equality for $Y$ and $Y'$. The last statement is clear.
\end{proof}

\begin{corollary}\label{coro:InvGrp} Assume that $\beta:S \to A$ is a monomorphism. Then
 $\mathrm{Inv}^r_{S,S}(A)$ is a monoid where the multiplication of $X,X' \in \mathrm{Inv}_{S,S}^r(A)$ is given by $X\tensor{S}X'$ with monomorphism $i_{X\tensor{S}X'}:= \mas \circ(i_X\tensor{S}i_{X'})$.
The neutral element is $S \in \mathrm{Inv}^r_{S,S}(A)$ via $\beta$. Moreover $\mathrm{Inv}_{S,S}(A)$ is the group of units of this monoid.
\end{corollary}
\begin{proof}
Let $X, X'\in \mathrm{Inv}^r_{S,S}(A)$. By Proposition \ref{pro: InvTens}, we have $X\tensor{S}X' \in \mathrm{Inv}^r_{S,S}(A)$, with monomorphism $\mas\circ(i_X\tensor{S}i_{X'})$,  so that $\tensor{S}$ is a well-defined multiplication for $\mathrm{Inv}_{S,S}^r(A)$. It is clearly associative. Note that $X\tensor{S}S = X = S\tensor{S}X$ as sub-bimodules of $A$
(these equalities make sense as the definition of subobject is given up to isomorphism, cf.~\cite[page 122]{MacLane-Categories}). Thus $S$ is the neutral element for this operation.

Let us check the last statement. For this consider $X\in \mathrm{Inv}_{S,S}^r(A)$ and $Y$ its right inverse (as in Definition \ref{def:inverse}).  Then, by equation \eqref{def: mX}, we have
$$
i_{X\tensor{S}Y}\,:=\, \mas \circ(i_X\tensor{S} i_Y)\,=\, \beta \circ m_{X}.
$$
This means that $X\tensor{S}Y =S \in \mathrm{Inv}_{S,\,S}^r(A)$, as $m_X:X\tensor{S}Y\rightarrow S$ is an isomorphism. Therefore, any element in $\mathrm{Inv}_{S,\,S}^r(A)$ is already right invertible w.r.t.~to the multiplication $\tensor{S}$.  Thus an element $X  \in \mathrm{Inv}_{S,\,S}^r(A)$ belongs to the group of units if and only if it has a left $\tensor{S}$-inverse. Note that, in this case, the left and the right inverses should coincide as we are in a monoid. Therefore,  the element $X$ is $\tensor{S}$-invertible if and only if $Y\tensor{S}X= S$. Now, $Y\tensor{S}X$ is sub-bimodule of $A$ via $\mas \circ (i_Y\tensor{S}i_X) = \beta\circ m_{Y}$  by equation \eqref{def: mY}. Thus, the equality $Y\tensor{S}X=S$ holds if and only if there is an isomorphism $\xi:Y\tensor{S}X\to S$ of bimodules over $S$ such that $\beta\circ\xi=\beta\circ m_Y$. Since we are assuming $\beta$ to be a monomorphism, this equality is equivalent to say that $m_Y$ is an isomorphism. This entails that $X$ is $\tensor{S}$-
invertible if and only if $X \in \mathrm{Inv}_{S,S}(A)$.
\end{proof}

The group $\mathrm{Inv}_{S,S}(A)$, considered in Corollary \ref{coro:InvGrp}, will also be denoted by $\mathrm{Inv}_{S}(A)$.

\begin{proposition}\label{prop:monoiso}
Let $X, X' \in \mathrm{Inv}_{R,\,S}(A)$ be two-sided invertible sub-bimodules of $A$ with inverses $Y,Y' \in \mathrm{Inv}_{S,\,R}(A)$ respectively. Consider the category isomorphism $\B{\phiup}$ of Lemma \ref{lema:leftright}, and assume that  there is a monomorphism  $i: X \hookrightarrow X'$ of bimodules in $\mathscr{P}({}_RA_S)$ (i.e.~satisfying $i_{X'} \circ i \,=\, i_X$).
\begin{enumerate}[(1)]
\item The morphism $i$ is an isomorphism  if and only if $\B{\phiup}(i)$ is.
\item Assume further that there is a monomorphism  of bimodules $j: Y \hookrightarrow Y'$ in $\mathscr{P}({}_SA_R)$ (i.e.~satisfying $i_{Y'} \circ j \,=\, i_Y$). Then both $i$ and $j$ are isomorphisms i.e. $X= X'$ and  $Y= Y'$ as elements in $\mathrm{Inv}_{R,S}\left( A\right)$ and $\mathrm{Inv}_{S,R}\left( A\right)$, respectively.
\end{enumerate}
\end{proposition}

\begin{proof}
$(1)$ is trivial.
$(2)$. From Lemma \ref{lema:leftright} we know that $i_Y \circ \B{\phiup}(i)= i_{Y'}$. Thus $i_{Y'}\circ j \circ \B{\phiup}(i)=i_Y \circ \B{\phiup}(i)=i_{Y'}$. Since $i_{Y'}$ is a monomorphism, we get $j \circ \B{\phiup}(i)=\id_{Y'}.$ Similarly one proves that $\B{\phiup}(i)\circ j=\id_Y$. Thus  $j$ and  $\B{\phiup}(i)$ are isomorphisms. By  $(1)$, the morphism $i$ is an isomorphism too, and this finishes the proof.
\end{proof}

\section{Miyashita action  and invariant subobjects}\label{sec:action}
In this section  we assume that our base monoidal category $(\cat{M},\tensor{}, \I, l,r)$ is also bicomplete.
In what follows let $A, R$ and $S$ be three monoids in $\cat{M}$ and $\alpha: R \rightarrow A \leftarrow S: \beta$ be morphisms of monoids which are monomorphisms in $\cat{M}$.
We denote by  $\cat{Z}: \cat{M} \to \rmod{\Z{\I}}$  the functor $\hom{\cat{M}}{\I}{-}$. Analogously,   we denote by $\ZR{-}: {}_R\cat{M}_R \to \rmod{\ZR{R}}$ the functor $\hom{R,R}{R}{-}$, and similarly we consider $\ZS{-}$.

\subsection{Miyashita action: Definition}\label{ssec:miyaaction}
In this subsection we introduce the \emph{Miyashita action} in the context of monoidal categories. This is a map (group homomorphism) which connects   $\mathrm{Inv}_{S,R}^{r}\left( A\right) $ (resp.~$\mathrm{Inv}_{S,R}\left( A\right) $)  with   the set of (iso)morphisms of $\Z{\I}$-algebras from $\ZS{A}$ to $\ZR{A}$. The latter can be seen as the  invariants (subalgebra) of $A$ with respect to $S$ and $R$, respectively  (usually they are denoted by $A^S$ and $A^R$ respectively in the classical case).

Let $M$ be a left $S$-module and let $N$ be an $A$-bimodule. Consider the action $\ZS{A} \times \lhom{S}{M}{N} \rightarrow\lhom{S}{M}{N}:(z,h)\mapsto z\trir{S} h$ defined by setting
\begin{equation}\label{Eq:triangles}
z\trir{S} h:=\Big(\xymatrix@C=40pt{  M\ar@{->}^{(l_{\scriptscriptstyle{M}}^{\scriptscriptstyle{S}})^{-1}}[r] &S\tensor{S}M \ar[r]^-{z\tensor{S}h} & A\tensor{S}N \ar@{->}^{\rho^S_N}[r] & N   } \Big)
\end{equation}

Similarly one defines the action $\rhom{R}{M}{N} \times \ZR{A}\rightarrow\rhom{R}{M}{N}:(f,z)\mapsto f\tril{R} z$. When $S=R=\I$ we omit the subscripts and write $f\triangleleft z$ and $z\triangleright f$.

\begin{lemma}\label{lema:convolution}
Let $\left( \mathcal{M},\otimes ,\mathbb{I}\right) $ be a monoidal category.
Let $\left( C,\Delta _{C},\varepsilon _{C}\right) $ be a comonoid  in $%
\left( \mathcal{M},\otimes ,\mathbb{I}\right) $ and let $\left(
A,m_{A},u_{A}\right) $ be a monoid in $\left( \mathcal{M},\otimes ,\mathbb{%
I}\right) .$ Then
\begin{equation*}
\left( B=\mathrm{Hom}_{\mathcal{M}}\left( C,A\right) ,m_{B},1_{B}\right)
\end{equation*}%
is a $\mathbb{Z}$-algebra where, for all $f,g\in B$
\begin{equation*}
f\ast g:=m_{B}\circ \left( f\otimes g\right) :=m_{A} \circ \left( f\otimes g\right) \circ
\Delta _{C}\qquad \text{and}\qquad 1_{B}:=u_{A}\circ \varepsilon _{C}.
\end{equation*}
\end{lemma}
\begin{proof}
Straightforward.
\end{proof}

Since $( R,\left(l^{\scriptscriptstyle{R}}_{\scriptscriptstyle{R}}\right) ^{-1},\mathrm{Id}_{R}) $ is a comonoid and $\left(A,\mar,\alpha \right) $ is a monoid, both in $\left( _{R}\mathcal{M}_{R},\tensor{R},R\right) $, we have that $\ZR{A}$ is a  $\mathbb{Z}$-algebra by Lemma \ref{lema:convolution}. Furthermore, since  the base category is assumed to be Penrose, we have that $R\tensor{}t=t\tensor{}R$ for every element $t$ in $\Z{\I}$ which in fact defines an algebra map $\Z{\I} \to \ZR{R}$ so that $\ZR{R}$ becomes  a commutative $\Z{\I}$-algebra. In this way, the map $\ZR{\alpha}$  clearly induces a structure of   $\Z{\I}$-algebra on
$\ZR{A}$. Explicitly, the unit $\td{\alpha}:\Z{\I}\to \ZR{A}$ of this algebra maps an element $t \in \Z{\I}$ to $\alpha \circ (R\tensor{}t) \,=\, \alpha \circ (t\tensor{}R)$ which is an element in $\ZR{A}$.  Similarly, one constructs $\td{\beta}: \Z{\I} \to \ZS{A}$. On the other hand, one shows that the map
\begin{equation}\label{Eq:varphi}
\varphi_R:\ZR{A} \to \cat{Z}(A), \quad f \mapsto f \circ u_R
\end{equation}
is multiplicative and satisfies $\varphi_R \circ \td{\alpha} =\cat{Z}(u)$, so that it is a $\Z{\I}$-algebras map. Moreover, by using that morphisms in $\ZR{A}$ are of $R$-bimodules, one also gets that $\varphi_R$ is always injective.  In this way, both $\ZR{A}$ and $\ZS{A}$ become $\Z{\I}$-subalgebras of $\cat{Z}(A)$.

\begin{proposition}\label{pro: right Miyashita}
For any $g\in \ZS{A}$ and $X \in \mathrm{Inv}_{R,S}^{r}\left( A\right) $ let $\B{\sigma}
_{X}^{S,R}\left( g\right) :=\B{\sigma} _{X}\left( g\right) :R\rightarrow A$ be
defined by
\begin{equation}\label{form:defsigmaX}
\B{\sigma} _{X}\left( g\right) =\mas\circ \left( i_X\tensor{S}(g\trir{S}i_{Y})\right) \circ \left( m_{X}\right) ^{-1},
\end{equation}
where $Y$ is a right inverse of $X$.
Then, $\B{\sigma} _{X}$ does not depend on the choice of $Y$ and  the map
$$
\B{\sigma} ^{S,R} \,:=\, \B{\sigma} :\mathrm{Inv}_{R,S}^{r}\left( A\right) \longrightarrow
\mathrm{Hom}_{\scriptscriptstyle{\Z{\I}}\text{-alg}}\Big(\ZS{A},\ZR{A}\Big), \;\; \Big\{
X\longmapsto \B{\sigma} _{X}\Big\}
$$
is well-defined.
\end{proposition}
\begin{proof}
First note that the right inverse $Y$ of $X$ is unique up to isomorphism as shown in Lemma \ref{lem: rightinverse}. Let us check that $\B{\sigma} _{X}\left( g\right) $ does not depend on this isomorphism. Thus, given another right inverse $Z$ of $X$, we want to check  that the formula \eqref{form:defsigmaX} is the same for both $Y$ and $Z$. That is,  we need to check the following equality
$$\B{\sigma} _{X,Y}\left( g\right) := \mas\circ \left( i_X\tensor{S}(g\trir{S}i_{Y})\right) \circ \left( m_{X}\right) ^{-1}\,\,=\,\, \mas\circ \left( i_X\tensor{S}(g\trir{S}i_{Z})\right) \circ \left( m_{X}\right) ^{-1}:= \B{\sigma} _{X,Z}\left( g\right).$$
The map $\B{\sigma} _{X,Y}(g)$ can be represented by the following diagram:
\begin{center}
\begin{tikzpicture}[x=9pt,y=9pt,thick]\pgfsetlinewidth{0.5pt}
\node[inner sep=1pt] at (0,0) {$\B{\sigma}_{X,\, Y}(g)\,:=\,$};
\node at (0,-5) {};
\end{tikzpicture}
\begin{tikzpicture}[x=9pt,y=9pt,thick]\pgfsetlinewidth{0.5pt}
\node[circle,draw, inner sep=0.6pt](1) at (0,0) {$\scriptstyle{i_X}$};
\node[circle,draw, inner sep=0.6pt](2) at (3,0) {$\scriptstyle{i_Y}$};
\node[circle,draw, inner sep=1pt](3) at (1.5,2) {$\scriptstyle{g}$};
\node(4) at (1.3,-5) {$\scriptstyle{A}$};

\draw[-] (1) to [out=90,in=-90] (0,2);
\draw[-] (2) to [out=90,in=-90] (3,2);
\draw[-] (0,2) to [out=90,in=180] (1.5,3.5);
\draw[-] (1.5,3.5) to [out=0,in=90] (3,2);
\draw[-] (3) to [out=-90,in=90] (1.5,-1.5);
\draw[-] (1) to [out=-90,in=90] (0,-2.5);
\draw[-] (2) to [out=-90,in=90] (3,-1.5);
\draw[-] (3,-1.5) to [out=-90,in=-90] (1.5,-1.5);
\draw[-] (2.3,-2) to [out=-90,in=90] (2.3,-2.55);
\draw[-] (2.3,-2.5) to [out=-90,in=-90] (0,-2.5);
\draw[-] (4) to [out=90,in=-90] (1.3,-3.2);
\end{tikzpicture}
\end{center}
Using the isomorphisms stated in Corollary \ref{lem: rightinverse},  we then have
\begin{center}
\begin{tikzpicture}[x=9pt,y=9pt,thick]\pgfsetlinewidth{0.5pt}
\node[inner sep=1pt] at (0,0.5) {$\B{\sigma}_{X,\, Z}(g)\,=\,$};
\node at (0,-5) {};
\end{tikzpicture}
\begin{tikzpicture}[x=9pt,y=9pt,thick]\pgfsetlinewidth{0.5pt}
\node[circle,draw, inner sep=0.6pt](1) at (0,0) {$\scriptstyle{i_X}$};
\node[circle,draw, inner sep=0.6pt](2) at (3,0) {$\scriptstyle{i_Z}$};
\node[circle,draw, inner sep=1pt](3) at (1.5,2) {$\scriptstyle{g}$};
\node(4) at (1.3,-5.2) {$\scriptstyle{A}$};

\draw[-] (1) to [out=90,in=-90] (0,4.5);
\draw[-] (2) to [out=90,in=-90] (3,4.5);
\draw[-] (0,4.5) to [out=90,in=180] (1.5,5.5);
\draw[-] (1.5,5.5) to [out=0,in=90] (3,4.5);
\draw[-] (3) to [out=-90,in=90] (1.5,-1.5);
\draw[-] (1) to [out=-90,in=90] (0,-2.5);
\draw[-] (2) to [out=-90,in=90] (3,-1.5);
\draw[-] (3,-1.5) to [out=-90,in=-90] (1.5,-1.5);
\draw[-] (2.3,-2) to [out=-90,in=90] (2.3,-2.5);
\draw[-] (2.3,-2.5) to [out=-90,in=-90] (0,-2.5);
\draw[-] (4) to [out=90,in=-90] (1.3,-3.2);
\end{tikzpicture}
\begin{tikzpicture}[x=9pt,y=9pt,thick]\pgfsetlinewidth{0.5pt}
\node[inner sep=1pt] at (0,0.5) {$\,\overset{\eqref{form:iYZ2}}{=}\,$};
\node at (0,-5) {};
\end{tikzpicture}
\begin{tikzpicture}[x=9pt,y=9pt,thick]\pgfsetlinewidth{0.5pt}
\node[circle,draw, inner sep=0.6pt](1) at (0,0) {$\scriptstyle{i_X}$};
\node[circle,draw, inner sep=0.6pt](2) at (3,0) {$\scriptstyle{i_Z}$};
\node[circle,draw, inner sep=1pt](3) at (1.5,2) {$\scriptstyle{g}$};
\node(4) at (1.3,-5.2) {$\scriptstyle{A}$};
\node[circle,draw, inner sep=0pt](5) at (3,3.5) {$\scriptstyle{i_{Y, Z}}$};

\draw[-] (1) to [out=90,in=-90] (0,4.3);
\draw[-] (2) to [out=90,in=-90] (5);
\draw[-] (0,4.3) to [out=90,in=180] (1.5,5.5);
\draw[-] (1.5,5.5) to [out=0,in=90] (5);
\draw[-] (3) to [out=-90,in=90] (1.5,-1.5);
\draw[-] (1) to [out=-90,in=90] (0,-2.5);
\draw[-] (2) to [out=-90,in=90] (3,-1.5);
\draw[-] (1.5,-1.5) to [out=-90,in=-90] (3,-1.5);
\draw[-] (2.3,-2) to [out=-90,in=90] (2.3,-2.5);
\draw[-] (2.3,-2.5) to [out=-90,in=-90] (0,-2.5);
\draw[-] (4) to [out=90,in=-90] (1.3,-3.2);
\end{tikzpicture}
\begin{tikzpicture}[x=9pt,y=9pt,thick]\pgfsetlinewidth{0.5pt}
\node[inner sep=1pt] at (0,0.5) {$\,=\,$};
\node at (0,-5) {};
\end{tikzpicture}
\begin{tikzpicture}[x=9pt,y=9pt,thick]\pgfsetlinewidth{0.5pt}
\node[circle,draw, inner sep=0.6pt](1) at (0,0) {$\scriptstyle{i_X}$};
\node[circle,draw, inner sep=0.6pt](2) at (3,0) {$\scriptstyle{i_Z}$};
\node[circle,draw, inner sep=1pt](3) at (1.5,4) {$\scriptstyle{g}$};
\node(4) at (1.3,-5.2) {$\scriptstyle{A}$};
\node[circle,draw, inner sep=0pt](5) at (3,2.5) {$\scriptstyle{i_{Y,Z}}$};

\draw[-] (1) to [out=90,in=-90] (0,4.3);
\draw[-] (2) to [out=90,in=-90] (5);
\draw[-] (0,4.3) to [out=90,in=180] (1.5,5.5);
\draw[-] (1.5,5.5) to [out=0,in=90] (3,4.5);
\draw[-] (3,4.5) to [out=-90,in=90] (5);
\draw[-] (3) to [out=-90,in=90] (1.5,-1.5);
\draw[-] (1) to [out=-90,in=90] (0,-2.5);
\draw[-] (2) to [out=-90,in=90] (3,-1.5);
\draw[-] (3,-1.5) to [out=-90,in=-90] (1.5,-1.5);
\draw[-] (2.3,-2) to [out=-90,in=90] (2.3,-2.5);
\draw[-] (2.3,-2.5) to [out=-90,in=-90] (0,-2.5);
\draw[-] (4) to [out=90,in=-90] (1.3,-3.2);
\end{tikzpicture}
\begin{tikzpicture}[x=9pt,y=9pt,thick]\pgfsetlinewidth{0.5pt}
\node[inner sep=1pt] at (0,0.5) {$\,=\,$};
\node at (0,-5) {};
\end{tikzpicture}
\begin{tikzpicture}[x=9pt,y=9pt,thick]\pgfsetlinewidth{0.5pt}
\node[circle,draw, inner sep=0.6pt](1) at (0,0) {$\scriptstyle{i_X}$};
\node[circle,draw, inner sep=0.6pt](2) at (3,0) {$\scriptstyle{i_Y}$};
\node[circle,draw, inner sep=1pt](3) at (1.5,3) {$\scriptstyle{g}$};
\node(4) at (1.3,-5.2) {$\scriptstyle{A}$};

\draw[-] (1) to [out=90,in=-90] (0,4.5);
\draw[-] (2) to [out=90,in=-90] (3,4.5);
\draw[-] (0,4.5) to [out=90,in=180] (1.5,5.5);
\draw[-] (1.5,5.5) to [out=0,in=90] (3,4.5);
\draw[-] (3) to [out=-90,in=90] (1.5,-1.5);
\draw[-] (1) to [out=-90,in=90] (0,-2.5);
\draw[-] (2) to [out=-90,in=90] (3,-1.5);
\draw[-] (3,-1.5) to [out=-90,in=-90] (1.5,-1.5);
\draw[-] (2.3,-2) to [out=-90,in=90] (2.3,-2.5);
\draw[-] (2.3,-2.5) to [out=-90,in=-90] (0,-2.5);
\draw[-] (4) to [out=90,in=-90] (1.3,-3.2);
\end{tikzpicture}
\begin{tikzpicture}[x=9pt,y=9pt,thick]\pgfsetlinewidth{0.5pt}
\node[inner sep=1pt] at (0,0.5) {$\,=\,\B{\sigma}_{X,\, Y}(g).$};
\node at (0,-5) {};
\end{tikzpicture}
\end{center}

By a similar argument, one proves that the definition of this map does not depend on the representative of the equivalence class $(X,i_X)$. In other words if $X=X'$ (i.e.~there is an isomorphism $f:X'\to X$ such that $i_X\circ f=i_{X'}$), then $\B{\sigma} _{X}\left(
g\right) =\B{\sigma} _{X'}\left(
g\right)$. On the other hand,  the morphism $\B{\sigma} _{X}\left(
g\right) $ is  a morphism in $_{R}\mathcal{M}_{R}$ being the composition of morphisms in $_{R}\mathcal{M}_{R}$. Therefore, $\B{\sigma} _{X}\left( g\right) $ is an element in $\ZR{A}$.

We  have to check that $\B{\sigma} _{X}$ is a $\Z{\I}$-algebra map. To this aim, given $g,h\in \ZS{A}$, we first show that $\B{\sigma} _{X}\left( g\right) \ast
\B{\sigma} _{X}\left( h\right) =\B{\sigma} _{X}\left( g\ast h\right)$. Using the diagrammatic notation,  we have
\begin{center}
\begin{tikzpicture}[x=8pt,y=8pt,thick]\pgfsetlinewidth{0.5pt}
\node[ellipse,draw, inner sep=1pt](1) at (0,0) {$\scriptstyle{\B{\sigma}_{X}(g)}$};
\node[circle,draw, inner sep=0.6pt](2) at (4,0) {$\scriptstyle{i_X}$};
\node(3) at (4,5) {$\scriptstyle{X}$};
\node(4) at (2,-5) {$\scriptstyle{A}$};
\draw[-] (1) to [out=-90,in=90] (0,-2);
\draw[-] (0,-2) to [out=-90,in=180] (2,-3);
\draw[-] (4,-2) to [out=90,in=-90] (2);
\draw[-] (2,-3) to [out=0,in=-90] (4,-2);
\draw[-] (4) to [out=90,in=-90] (2,-3);
\draw[-] (3) to [out=-90,in=90] (2);
\end{tikzpicture}
\begin{tikzpicture}[x=8pt,y=8pt,thick]\pgfsetlinewidth{0.5pt}
\node[inner sep=1pt] at (0,0) {$\,=\,$};
\node at (0,-5) {};
\end{tikzpicture}
\begin{tikzpicture}[x=8pt,y=8pt,thick]\pgfsetlinewidth{0.5pt}
\node[circle,draw, inner sep=0.5pt](1) at (0,1) {$\scriptstyle{i_X}$};
\node[circle,draw, inner sep=0.5pt](2) at (3,1) {$\scriptstyle{i_Y}$};
\node[circle,draw, inner sep=1pt](3) at (1.5,3) {$\scriptstyle{g}$};
\node(4) at (3,-5) {$\scriptstyle{A}$};
\node[circle,draw, inner sep=0.6pt](5) at (5,0) {$\scriptstyle{i_X}$};
\node(6) at (5,5) {$\scriptstyle{X}$};

\draw[-] (1) to [out=90,in=-90] (0,4);
\draw[-] (2) to [out=90,in=-90] (3,4);
\draw[-] (0,4) to [out=90,in=90] (3,4);
\draw[-] (3) to [out=-90,in=90] (1.5,-0.5);
\draw[-] (1) to [out=-90,in=90] (0,-1.5);
\draw[-] (2) to [out=-90,in=90] (3,-0.5);
\draw[-] (3,-0.5) to [out=-90,in=-90] (1.5,-0.5);
\draw[-] (2.3,-1) to [out=-90,in=90] (2.3,-1.5);
\draw[-] (2.3,-1.5) to [out=-90,in=-90] (0,-1.5);
\draw[-] (1.2,-2.2) to [out=-90,in=90] (1.2,-3);
\draw[-] (5,-3) to [out=-90,in=-90] (1.2,-3);
\draw[-] (6) to [out=-90,in=90] (5);
\draw[-] (5) to [out=-90,in=90] (5,-3);
\draw[-] (4) to [out=90,in=-90] (3,-4.5);
\end{tikzpicture}
\begin{tikzpicture}[x=8pt,y=8pt,thick]\pgfsetlinewidth{0.5pt}
\node[inner sep=1pt] at (0,0) {$\,=\,$};
\node at (0,-5) {};
\end{tikzpicture}
\begin{tikzpicture}[x=8pt,y=8pt,thick]\pgfsetlinewidth{0.5pt}
\node[circle,draw, inner sep=0.5pt](1) at (0,1) {$\scriptstyle{i_X}$};
\node[circle,draw, inner sep=0.5pt](2) at (3,1) {$\scriptstyle{i_Y}$};
\node[circle,draw, inner sep=1pt](3) at (1.5,3) {$\scriptstyle{g}$};
\node(4) at (2.3,-5) {$\scriptstyle{A}$};
\node[circle,draw, inner sep=0.5pt](5) at (5,0) {$\scriptstyle{i_X}$};
\node(6) at (5,5) {$\scriptstyle{X}$};

\draw[-] (1) to [out=90,in=-90] (0,4);
\draw[-] (2) to [out=90,in=-90] (3,4);
\draw[-] (0,4) to [out=90,in=90] (3,4);

\draw[-] (3) to [out=-90,in=90] (1.5,-1);
\draw[-] (1) to [out=-90,in=90] (0,-1);
\draw[-] (0,-1) to [out=-90,in=-90] (1.5,-1);
\draw[-] (2) to [out=-90,in=90] (3,-1.5);
\draw[-] (6) to [out=-90,in=90] (5);
\draw[-] (5) to [out=-90,in=90] (5,-1.5);
\draw[-] (3,-1.5) to [out=-90,in=-90]  (5,-1.5);
\draw[-] (0.7,-1.5) to [out=-90,in=90] (0.7,-2.5);
\draw[-] (4,-2) to [out=-90,in=90] (4,-2.5);
\draw[-] (0.7,-2.5) to [out=-90,in=-90] (4,-2.5);
\draw[-] (4) to [out=90,in=-90] (2.3,-3.5);
\end{tikzpicture}
\begin{tikzpicture}[x=8pt,y=8pt,thick]\pgfsetlinewidth{0.5pt}
\node[inner sep=1pt] at (0,0) {$\,=\,$};
\node at (0,-5) {};
\end{tikzpicture}
\begin{tikzpicture}[x=8pt,y=8pt,thick]\pgfsetlinewidth{0.5pt}
\node[circle,draw, inner sep=0.5pt](1) at (0,1) {$\scriptstyle{i_X}$};
\node[circle,draw, inner sep=1pt](3) at (1.5,3) {$\scriptstyle{g}$};
\node(4) at (2.3,-5) {$\scriptstyle{A}$};
\node(6) at (5,5) {$\scriptstyle{X}$};
\node[circle,draw, inner sep=1pt](7) at (4,-0.5) {$\scriptstyle{\beta}$};

\draw[-] (1) to [out=90,in=-90] (0,4);
\draw[-] (3,4) to [out=-90,in=90] (3,1.5);
\draw[-] (3,4) to [out=90,in=90] (0,4);

\draw[-] (3,1.5) to [out=-90,in=-90] (5,1.5);
\draw[-] (3) to [out=-90,in=90] (1.5,-1);
\draw[-] (1) to [out=-90,in=90] (0,-1);
\draw[-] (0,-1) to [out=-90,in=-90] (1.5,-1);
\draw[-] (6) to [out=-90,in=90] (5,1.5);
\draw[-] (0.7,-1.5) to [out=-90,in=90] (0.7,-2.5);
\draw[-] (7) to [out=-90,in=90] (4,-2.5);
\draw[-] (0.7,-2.5) to [out=-90,in=-90]  (4,-2.5);
\draw[-] (4) to [out=90,in=-90] (2.3,-3.5);
\end{tikzpicture}
\begin{tikzpicture}[x=8pt,y=8pt,thick]\pgfsetlinewidth{0.5pt}
\node[inner sep=1pt] at (0,0) {$\,=\,$};
\node at (0,-5) {};
\end{tikzpicture}
\begin{tikzpicture}[x=8pt,y=8pt,thick]\pgfsetlinewidth{0.5pt}
\node[circle,draw, inner sep=0.6pt](1) at (0,-1) {$\scriptstyle{i_X}$};
\node(4) at (1.5,-5) {$\scriptstyle{A}$};
\node(6) at (5,5) {$\scriptstyle{X}$};
\node[circle,draw, inner sep=1pt](7) at (3,-1) {$\scriptstyle{g}$};

\draw[-] (1) to [out=90,in=-90] (0,3);
\draw[-] (0,3) to [out=90,in=90] (3,3);
\draw[-] (1) to [out=-90,in=90] (0,-2.5);
\draw[-] (7) to [out=-90,in=90] (3,-2.5);
\draw[-] (0,-2.5) to [out=-90,in=-90] (3,-2.5);
\draw[-] (4) to [out=90,in=-90] (1.5,-3.4);
\draw[-] (3,3) to [out=-90,in=90] (3,1.5);
\draw[-] (3,1.5) to [out=-90,in=-90] (5,1.5);
\draw[-] (6) to [out=-90,in=90] (5,1.5);
\end{tikzpicture}
\begin{tikzpicture}[x=8pt,y=8pt,thick]\pgfsetlinewidth{0.5pt}
\node[inner sep=1pt] at (0,0) {$\,\overset{\eqref{form:ev1}}{=}\,$};
\node at (0,-5) {};
\end{tikzpicture}
\begin{tikzpicture}[x=8pt,y=8pt,thick]\pgfsetlinewidth{0.5pt}
\node[circle,draw, inner sep=0.6pt](1) at (0,-1) {$\scriptstyle{i_X}$};
\node(4) at (1.5,-5) {$\scriptstyle{A}$};
\node(6) at (0,5) {$\scriptstyle{X}$};
\node[circle,draw, inner sep=1pt](7) at (3,-1) {$\scriptstyle{g}$};

\draw[-] (1) to [out=90,in=-90] (0,4.5);
\draw[-] (1) to [out=-90,in=90] (0,-2.5);
\draw[-] (7) to [out=-90,in=90] (3,-2.5);
\draw[-] (0,-2.5) to [out=-90,in=180] (1.5,-3.5);
\draw[-] (1.5,-3.5) to [out=0,in=-90] (3,-2.5);
\draw[-] (4) to [out=90,in=-90] (1.5,-3.5);
\end{tikzpicture}
\end{center}
so that we get%
\begin{equation}\label{form:sigmaWiX}
\B{\sigma} _{X}(g)\trir{R} i_{X}= i_{X}\tril{S}g.
\end{equation}
Using this equality, we obtain
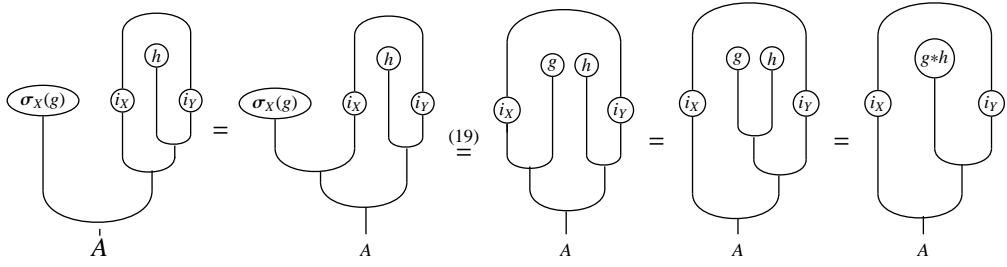
\begin{center}
\begin{tikzpicture}[x=8.5pt,y=8.5pt,thick]\pgfsetlinewidth{0.5pt}
\node[circle,draw, inner sep=0.2pt](1) at (0,1) {$\scriptstyle{i_X}$};
\node[circle,draw, inner sep=0.2pt](2) at (3,1) {$\scriptstyle{i_Y}$};
\node[circle,draw, inner sep=1pt](3) at (1.5,3) {$\scriptstyle{h}$};
\node(4) at (-1,-5.5) {$A$};
\node[ellipse,draw, inner sep=1pt](5) at (-3.5,1) {$\scriptstyle{\B{\sigma}_X(g)}$};

\draw[-] (1) to [out=90,in=-90] (0,4);
\draw[-] (2) to [out=90,in=-90] (3,4);
\draw[-] (0,4) to [out=90,in=90] (3,4);
\draw[-] (3) to [out=-90,in=90] (1.5,-0.5);
\draw[-] (1) to [out=-90,in=90] (0,-1.5);
\draw[-] (2) to [out=-90,in=90] (3,-0.5);
\draw[-] (1.5,-0.5) to [out=-90,in=-90] (3,-0.5);
\draw[-] (2.3,-1) to [out=-90,in=90] (2.3,-1.55);
\draw[-] (2.3,-1.5) to [out=-90,in=-90] (0,-1.5);
\draw[-] (5) to [out=-90,in=90] (-3.5,-3);
\draw[-] (1.3,-3) to [out=90,in=-90] (1.3,-2.1);
\draw[-] (1.3,-3) to [out=-90,in=-90] (-3.5,-3);
\draw[-] (4) to [out=90,in=-90] (-1,-5);
\end{tikzpicture}
\begin{tikzpicture}[x=8.5pt,y=8.5pt,thick]\pgfsetlinewidth{0.5pt}
\node[inner sep=1pt] at (0,0) {$=$};
\node at (0,-5.5) {};
\end{tikzpicture}
\begin{tikzpicture}[x=8.5pt,y=8.5pt,thick]\pgfsetlinewidth{0.5pt}
\node[circle,draw, inner sep=0.2pt](1) at (0,1) {$\scriptstyle{i_X}$};
\node[circle,draw, inner sep=0.2pt](2) at (3,1) {$\scriptstyle{i_Y}$};
\node[circle,draw, inner sep=1pt](3) at (1.5,3) {$\scriptstyle{h}$};
\node(4) at (0.5,-5.5) {$\scriptstyle{A}$};
\node[ellipse,draw, inner sep=1pt](5) at (-3.5,1) {$\scriptstyle{\B{\sigma}_X(g)}$};

\draw[-] (1) to [out=90,in=-90] (0,4);
\draw[-] (2) to [out=90,in=-90] (3,4);
\draw[-] (0,4) to [out=90,in=90] (3,4);
\draw[-] (3) to [out=-90,in=90] (1.5,-0.5);
\draw[-] (1) to [out=-90,in=90] (0,-1);
\draw[-] (2) to [out=-90,in=90] (3,-0.5);
\draw[-] (1.5,-0.5) to [out=-90,in=-90] (3,-0.5);
\draw[-] (5) to [out=-90,in=90] (-3.5,-1);
\draw[-] (-3.5,-1) to [out=-90,in=-90] (0,-1);
\draw[-] (-1.5,-2.5) to [out=90,in=-90] (-1.5,-2);
\draw[-] (2.3,-1) to [out=-90,in=90] (2.3,-2.5);
\draw[-] (2.3,-2.5) to [out=-90,in=-90] (-1.5,-2.5);
\draw[-] (4) to [out=90,in=-90] (0.5,-3.6);
\end{tikzpicture}
\begin{tikzpicture}[x=8.5pt,y=8.5pt,thick]\pgfsetlinewidth{0.5pt}
\node[inner sep=1pt] at (0,0) {$\,\overset{\eqref{form:sigmaWiX}}{=}\,$};
\node at (0,-5) {};
\end{tikzpicture}
\begin{tikzpicture}[x=8.5pt,y=8.5pt,thick]\pgfsetlinewidth{0.5pt}
\node[circle,draw, inner sep=0.6pt](1) at (-2,1) {$\scriptstyle{i_X}$};
\node[circle,draw, inner sep=0.6pt](2) at (3,1) {$\scriptstyle{i_Y}$};
\node[circle,draw, inner sep=1pt](3) at (1.5,3) {$\scriptstyle{h}$};
\node[circle,draw, inner sep=1pt](5) at (0,3) {$\scriptstyle{g}$};
\node(4) at (0.6,-5.2) {$\scriptstyle{A}$};

\draw[-] (1) to [out=90,in=-90] (-2,4);
\draw[-] (2) to [out=90,in=-90] (3,4);
\draw[-] (3,4) to [out=90,in=90] (-2,4);

\draw[-] (1) to [out=-90,in=90] (-2,0);
\draw[-] (3) to [out=-90,in=90] (1.5,-1);
\draw[-] (5) to [out=-90,in=90] (0,-1);
\draw[-] (2) to [out=-90,in=90] (3,-1);
\draw[-] (1) to [out=-90,in=90] (-2,-1);
\draw[-] (0,-1) to [out=-90,in=-90] (-2,-1);

\draw[-] (1.5,-1) to [out=-90,in=-90] (3,-1);
\draw[-] (-1,-1.5) to [out=-90,in=90] (-1,-2.5);
\draw[-] (2.3,-1.5) to [out=-90,in=90] (2.3,-2.5);
\draw[-] (2.3,-2.5) to [out=-90,in=-90]  (-1,-2.5);
\draw[-] (4) to [out=90,in=-90] (0.6,-3.5);
\end{tikzpicture}
\begin{tikzpicture}[x=8.5pt,y=8.5pt,thick]\pgfsetlinewidth{0.5pt}
\node[inner sep=1pt] at (0,0) {$\,=\,$};
\node at (0,-5) {};
\end{tikzpicture}
\begin{tikzpicture}[x=8.5pt,y=8.5pt,thick]\pgfsetlinewidth{0.5pt}
\node[circle,draw, inner sep=0.6pt](1) at (-2,1) {$\scriptstyle{i_X}$};
\node[circle,draw, inner sep=0.6pt](2) at (3,1) {$\scriptstyle{i_Y}$};
\node[circle,draw, inner sep=1pt](3) at (1.5,3) {$\scriptstyle{h}$};
\node[circle,draw, inner sep=1pt](5) at (0,3) {$\scriptstyle{g}$};
\node(4) at (0,-5.5) {$\scriptstyle{A}$};

\draw[-] (1) to [out=90,in=-90] (-2,4);
\draw[-] (2) to [out=90,in=-90] (3,4);
\draw[-] (3,4) to [out=90,in=90] (-2,4);
\draw[-] (1) to [out=-90,in=90] (-2,-3);
\draw[-] (3) to [out=-90,in=90] (1.5,0);
\draw[-] (5) to [out=-90,in=90] (0,0);
\draw[-] (2) to [out=-90,in=90] (3,-1.5);
\draw[-] (1.5,0) to [out=-90,in=-90] (0,0);
\draw[-] (0.7,-0.5) to [out=-90,in=90] (0.7,-1.5);
\draw[-] (0.7,-1.5) to [out=-90,in=-90] (3,-1.5);
\draw[-] (1.8,-2.2) to [out=-90,in=90] (1.8,-3);
\draw[-] (1.8,-3) to [out=-90,in=-90] (-2,-3);
\draw[-] (4) to [out=90,in=-90] (0,-4.1);
\end{tikzpicture}
\begin{tikzpicture}[x=8.5pt,y=8.5pt,thick]\pgfsetlinewidth{0.5pt}
\node[inner sep=1pt] at (0,0) {$\,=\,$};
\node at (0,-5) {};
\end{tikzpicture}
\begin{tikzpicture}[x=8.5pt,y=8.5pt,thick]\pgfsetlinewidth{0.5pt}
\node[circle,draw, inner sep=0.6pt](1) at (-2,1) {$\scriptstyle{i_X}$};
\node[circle,draw, inner sep=0.6pt](2) at (3,1) {$\scriptstyle{i_Y}$};
\node[circle,draw, inner sep=1pt](3) at (0.5,3) {$\scriptstyle{g * h}$};
\node(4) at (0,-5.5) {$\scriptstyle{A}$};

\draw[-] (1) to [out=90,in=-90] (-2,4);
\draw[-] (2) to [out=90,in=-90] (3,4);
\draw[-] (3,4) to [out=90,in=90] (-2,4);
\draw[-] (1) to [out=-90,in=90] (-2,-3);
\draw[-] (3) to [out=-90,in=90] (0.5,-1);
\draw[-] (1.7,-1.7) to [out=-90,in=90] (1.7,-3);
\draw[-] (-2,-3) to [out=-90,in=-90] (1.7,-3);
\draw[-] (2) to [out=-90,in=90] (3,-1);
\draw[-] (3,-1) to [out=-90,in=-90] (0.5,-1);

\draw[-] (4) to [out=90,in=-90] (0,-4);
\end{tikzpicture}
\end{center}
which means that $\B{\sigma} _{X}(g)\ast \B{\sigma} _{X}(h)=\B{\sigma} _{X}(g\ast h)$.
Moreover we have
\begin{equation*}
   \B{\sigma} _X( 1_{\scriptscriptstyle{\ZS{A}}}) =  \B{\sigma} _X(\beta)=\mas\circ(i_X\tensor{S}(\beta\trir{S} i_Y))\circ \mathrm{coev}=\mas\circ(i_X\tensor{S} i_Y)\circ \mathrm{coev}\overset{\eqref{form:ev3}}{=}\alpha =1_{\scriptscriptstyle{\ZR{A}}}.
\end{equation*}
We still need to check that $\B{\sigma}_{X} \circ \td{\beta} \,=\, \td{\alpha}$. To this aim, given  $t \in \Z{\I}$, we have
\begin{equation}\label{diag:mah}
\begin{tikzpicture}[x=8pt,y=8pt,thick]\pgfsetlinewidth{0.5pt}
\node[inner sep=1pt] at (0,0) {$\,\B{\sigma}_{X} \circ \td{\beta}(t)\,=\,$};
\node at (0,-5.5) {};
\end{tikzpicture}
\begin{tikzpicture}[x=8pt,y=8pt,thick]\pgfsetlinewidth{0.5pt}
\node[circle,draw, inner sep=0.5pt](1) at (-2,2.5) {$\scriptstyle{i_X}$};
\node[circle,draw, inner sep=0.5pt](2) at (2,2.5) {$\scriptstyle{i_Y}$};
\node[circle,draw, inner sep=0.5pt](3) at (0,0) {$\scriptstyle{\td{\beta}(t)}$};
\node(4) at (-0.5,-5.5) {$\scriptstyle{A}$};

\draw[-] (1) to [out=90,in=-90] (-2,4);
\draw[-] (2) to [out=90,in=-90] (2,4);
\draw[-] (2,4) to [out=90,in=90] (-2,4);

\draw[-] (2) to [out=-90,in=90] (2,-2);
\draw[-] (3) to [out=-90,in=90] (0,-2);
\draw[-] (0,-2) to [out=-90,in=-90] (2,-2);
\draw[-] (1) to [out=-90,in=90] (-2,-3.5);
\draw[-] (1,-3.5) to [out=90,in=-90] (1,-2.6);
\draw[-] (1,-3.5) to [out=-90,in=-90] (-2,-3.5);
\draw[-] (4) to [out=90,in=-90] (-0.5,-4.3);
\end{tikzpicture}
\begin{tikzpicture}[x=8pt,y=8pt,thick]\pgfsetlinewidth{0.5pt}
\node[inner sep=1pt] at (0,0) {$\,=\,$};
\node at (0,-5.5) {};
\end{tikzpicture}
\begin{tikzpicture}[x=8pt,y=8pt,thick]\pgfsetlinewidth{0.5pt}
\node[circle,draw, inner sep=0.5pt](1) at (-2,2.5) {$\scriptstyle{i_X}$};
\node[circle,draw, inner sep=0.5pt](2) at (2,2.5) {$\scriptstyle{i_Y}$};
\node[circle,draw, inner sep=1pt](3) at (0,1) {$\scriptstyle{\beta}$};
\node[circle,draw, inner sep=1pt](4) at (-1,-0.5) {$\scriptstyle{t}$};
\node(4) at (-0.5,-5.5) {$\scriptstyle{A}$};

\draw[-] (1) to [out=90,in=-90] (-2,4);
\draw[-] (2) to [out=90,in=-90] (2,4);
\draw[-] (2,4) to [out=90,in=90] (-2,4);

\draw[-] (2) to [out=-90,in=90] (2,-1);
\draw[-] (3) to [out=-90,in=90] (0,-1);
\draw[-] (0,-1) to [out=-90,in=-90] (2,-1);
\draw[-] (1) to [out=-90,in=90] (-2,-2.5);
\draw[-] (1,-2.5) to [out=90,in=-90] (1,-1.6);
\draw[-] (1,-2.5) to [out=-90,in=-90] (-2,-2.5);
\draw[-] (4) to [out=90,in=-90] (-0.5,-3.5);
\end{tikzpicture}
\begin{tikzpicture}[x=8pt,y=8pt,thick]\pgfsetlinewidth{0.5pt}
\node[inner sep=1pt] at (0,0) {$\,=\,$};
\node at (0,-5.5) {};
\end{tikzpicture}
\begin{tikzpicture}[x=8pt,y=8pt,thick]\pgfsetlinewidth{0.5pt}
\node[circle,draw, inner sep=0.5pt](1) at (-2,2.5) {$\scriptstyle{i_X}$};
\node[circle,draw, inner sep=0.5pt](2) at (2,2.5) {$\scriptstyle{i_Y}$};
\node[circle,draw, inner sep=1pt](4) at (0,-0.5) {$\scriptstyle{t}$};
\node(4) at (0,-5.5) {$\scriptstyle{A}$};

\draw[-] (1) to [out=90,in=-90] (-2,4);
\draw[-] (2) to [out=90,in=-90] (2,4);
\draw[-] (2,4) to [out=90,in=90] (-2,4);

\draw[-] (2) to [out=-90,in=90] (2,-2);
\draw[-] (1) to [out=-90,in=90] (-2,-2);
\draw[-] (-2,-2) to [out=-90,in=-90] (2,-2);
\draw[-] (4) to [out=90,in=-90] (0,-3.3);
\end{tikzpicture}
\begin{tikzpicture}[x=8pt,y=8pt,thick]\pgfsetlinewidth{0.5pt}
\node[inner sep=1pt] at (0,0) {$\,=\,$};
\node at (0,-5.5) {};
\end{tikzpicture}
\begin{tikzpicture}[x=8pt,y=8pt,thick]\pgfsetlinewidth{0.5pt}
\node[circle,draw, inner sep=0.5pt](1) at (-2,1) {$\scriptstyle{i_X}$};
\node[circle,draw, inner sep=0.5pt](2) at (2,1) {$\scriptstyle{i_Y}$};
\node[circle,draw, inner sep=1pt](4) at (3.7,-0.5) {$\scriptstyle{t}$};
\node(4) at (0,-5.5) {$\scriptstyle{A}$};

\draw[-] (1) to [out=90,in=-90] (-2,2.5);
\draw[-] (2) to [out=90,in=-90] (2,2.5);
\draw[-] (2,2.5) to [out=90,in=90] (-2,2.5);

\draw[-] (2) to [out=-90,in=90] (2,-2);
\draw[-] (1) to [out=-90,in=90] (-2,-2);
\draw[-] (-2,-2) to [out=-90,in=-90] (2,-2);
\draw[-] (4) to [out=90,in=-90] (0,-3.3);
\end{tikzpicture}
\begin{tikzpicture}[x=8pt,y=8pt,thick]\pgfsetlinewidth{0.5pt}
\node[inner sep=1pt] at (0,0) {$\,\overset{\eqref{eq: phi-ij}}{=}\,$};
\node at (0,-5.5) {};
\end{tikzpicture}
\begin{tikzpicture}[x=8pt,y=8pt,thick]\pgfsetlinewidth{0.5pt}
\node[circle,draw, inner sep=1pt](2) at (0,1) {$\scriptstyle{\alpha}$};
\node[circle,draw, inner sep=1pt](4) at (1.5,-0.5) {$\scriptstyle{t}$};
\node(4) at (0,-5.5) {$\scriptstyle{A}$};

\draw[-] (4) to [out=90,in=-90] (2);
\end{tikzpicture}
\begin{tikzpicture}[x=8pt,y=8pt,thick]\pgfsetlinewidth{0.5pt}
\node[inner sep=1pt] at (0,0) {${\,=\,}\,\td{\alpha}(t),\,$};
\node at (0,-5.5) {};
\end{tikzpicture}
\end{equation}
and this completes the proof.
\end{proof}

The map defined in the following proposition is an extension,  to the general framework of monoidal categories, of the so-called  \emph{Miyashita action} which was originally introduce by Miyashita in \cite[page 100]{Mi}.
Further developments  on this action appeared in various studies: Hopf Galois extensions, $H$-separable extensions, comodules over corings with grouplike elements etc, see  \cite{Masuoka-Corings,Kadison-Depth2, Kadison-H-sep,Kadison-NewExamples,Sch-Miya,Kaoutit-Gomez}.  Our general definition aims to provide a common and unifying context for all these studies.

\begin{proposition}\label{pro: Phi}
The map $$\B{\sigma} :\mathrm{Inv}_{R,S}^{r}\left( A\right)
\rightarrow \mathrm{Hom}_{\scriptscriptstyle{\Z{\I}}\text{-alg}}\left( \ZS{A},\ZR{A}\right) $$ of Proposition \ref{pro: right
Miyashita} induces a map%
\begin{equation*}
\Phi ^{S,R}:\mathrm{Inv}_{R,S}\left( A\right) \rightarrow \mathrm{Iso}_{\scriptscriptstyle{\Z{\I}}
\text{-alg}}\left(\ZS{A},\ZR{A} \right) .
\end{equation*}
In particular, when $R=S$ we have a morphism of groups $\Phi^R:=\Phi^{R,R}$. Moreover, $\Phi^R$ factors as the composition
\begin{equation}\label{Eq:PHI}
\mathrm{Inv}_{R}\left( A\right) \longrightarrow \mathrm{Aut}_{\scriptscriptstyle{\cat{Z}_{R}(R)}
\text{-alg}}\left( \ZR{A}\right) \hookrightarrow \mathrm{Aut}_{\scriptscriptstyle{\cat{Z}(\I)}
\text{-alg}}\left( \ZR{A}\right) .
\end{equation}
\end{proposition}

\begin{proof}
Consider both
\begin{eqnarray*}
\B{\sigma} ^{S,R} &:&\mathrm{Inv}_{R,S}^{r}\left( A\right) \rightarrow \mathrm{%
Hom}_{\scriptscriptstyle{\Z{\I}}\text{-alg}}\left( \ZS{A},\ZR{A} \right) \\
\B{\sigma} ^{R,S} &:&\mathrm{Inv}_{S,R}^{r}\left( A\right) \rightarrow \mathrm{%
Hom}_{\scriptscriptstyle{\Z{\I}}\text{-alg}}\left( \ZR{A},\ZS{A} \right).
\end{eqnarray*}%
Take $X \in \mathrm{Inv}_{R,S}\left( A\right)$ as in \eqref{Eq:inv}.  Since $m_{X^{r}}$ is an isomorphism, we have that $X^{r} \in \mathrm{Inv}_{S,R}^{r}\left( A\right) $. Hence we can consider both $\B{\sigma} _{X}^{S,R}$ and $\B{\sigma} _{
X^{r}}^{R,S}$. Let us check that these maps are mutual
inverses. For $g\in \ZS{A}$, we have $\left( \B{\sigma} _{X^{r}}^{R,S}\circ \B{\sigma} _{X}^{S,R}\right) \left( g\right)$ is equal to
\begin{equation}\label{diag:Boh}
\begin{tikzpicture}[x=8.5pt,y=8.5pt,thick]\pgfsetlinewidth{0.5pt}
\node[circle,draw, inner sep=0.5pt](1) at (0,1) {$\scriptstyle{i_X}$};
\node[circle,draw, inner sep=0.5pt](2) at (-7,1) {$\scriptstyle{i_{X^r}}$};
\node(0) at (0,-9) { };
\node(4) at (-4,-6.5) {$\scriptstyle{A}$};
\node[ellipse,draw, inner sep=1pt](5) at (-3.5,2) {$\scriptstyle{\B{\sigma}_X(g)}$};

\draw[-] (1) to [out=90,in=-90] (0,3);
\draw[-] (2) to [out=90,in=-90] (-7,3);
\draw[-] (-7,3) to [out=90,in=90] (0,3);
\draw[-] (1) to [out=-90,in=90] (0,-1.5);
\draw[-] (2) to [out=-90,in=90] (-7,-3.3);
\draw[-] (-7,-3.3) to [out=-90,in=-90] (-1.5,-3.3);
\draw[-] (0,-1.5) to [out=-90,in=0] (-1.5,-2.5);

\draw[-] (5) to [out=-90,in=90] (-3.5,-1.5);
\draw[-] (-3.5,-1.5) to [out=-90,in=180] (-1.5,-2.5);
\draw[-] (-1.5,-2.5) to [out=-90,in=90] (-1.5,-3.3);
\draw[-] (4) to [out=90,in=-90] (-4,-4.9);
\end{tikzpicture}
\begin{tikzpicture}[x=8.5pt,y=8.5pt,thick]\pgfsetlinewidth{0.5pt}
\node[inner sep=1pt] at (0,0) {$=$};
\node at (0,-9) {};
\end{tikzpicture}
\begin{tikzpicture}[x=8.5pt,y=8.5pt,thick]\pgfsetlinewidth{0.5pt}
\node[circle,draw, inner sep=0.5pt](1) at (4,0) {$\scriptstyle{i_X}$};
\node[circle,draw, inner sep=0.5pt](2) at (-4,0) {$\scriptstyle{i_{X^r}}$};
\node[circle,draw, inner sep=0.5pt](11) at (2,1.5) {$\scriptstyle{i_{X^r}}$};
\node[circle,draw, inner sep=0.5pt](22) at (-2,1.5) {$\scriptstyle{i_X}$};
\node[circle,draw, inner sep=1pt](3) at (0,2) {$\scriptstyle{g}$};
\node(4) at (-1.2,-9) {$\scriptstyle{A}$};

\draw[-] (1) to [out=90,in=-90] (4,3.5);
\draw[-] (2) to [out=90,in=-90] (-4,3.5);
\draw[-] (4,3.5) to [out=90,in=90] (-4,3.5);
\draw[-] (11) to [out=90,in=-90] (2,2.5);
\draw[-] (22) to [out=90,in=-90] (-2,2.5);
\draw[-] (2,2.5) to [out=90,in=90] (-2,2.5);

\draw[-] (1) to [out=-90,in=90] (4,-4);
\draw[-] (2) to [out=-90,in=90] (-4,-6);
\draw[-] (11) to [out=-90,in=90] (2,-0.5);
\draw[-] (22) to [out=-90,in=90] (-2,-2);
\draw[-] (3) to [out=-90,in=90] (0,-0.5);
\draw[-] (0,-0.5) to [out=-90,in=-90] (2,-0.5);
\draw[-] (1,-1) to [out=-90,in=90] (1,-2);
\draw[-] (-2,-2) to [out=-90,in=-90] (1,-2);
\draw[-] (-0.5,-3) to [out=90,in=-90] (-0.5,-4);
\draw[-] (4,-4) to [out=-90,in=-90] (-0.5,-4);
\draw[-] (2,-6) to [out=90,in=-90] (2,-5.2);
\draw[-] (2,-6) to [out=-90,in=-90] (-4,-6);
\draw[-] (4) to [out=90,in=-90] (-1.2,-7.8);
\end{tikzpicture}
\begin{tikzpicture}[x=8.5pt,y=8.5pt,thick]\pgfsetlinewidth{0.5pt}
\node[inner sep=1pt] at (0,0) {$\,=\,$};
\node at (0,-9) {};
\end{tikzpicture}
\begin{tikzpicture}[x=8.5pt,y=8.5pt,thick]\pgfsetlinewidth{0.5pt}
\node[circle,draw, inner sep=0.5pt](1) at (4,1) {$\scriptstyle{i_X}$};
\node[circle,draw, inner sep=0.5pt](2) at (-4,1) {$\scriptstyle{i_{X^r}}$};
\node[circle,draw, inner sep=0.5pt](11) at (2,2) {$\scriptstyle{i_{X^r}}$};
\node[circle,draw, inner sep=0.5pt](22) at (-2,2) {$\scriptstyle{i_X}$};
\node[circle,draw, inner sep=1pt](3) at (0,2.5) {$\scriptstyle{g}$};
\node(4) at (-0.7,-9) {$\scriptstyle{A}$};

\draw[-] (1) to [out=90,in=-90] (4,3.5);
\draw[-] (2) to [out=90,in=-90] (-4,3.5);
\draw[-] (4,3.5) to [out=90,in=90] (-4,3.5);
\draw[-] (11) to [out=90,in=-90] (2,3);
\draw[-] (22) to [out=90,in=-90] (-2,3);
\draw[-] (2,3) to [out=90,in=90] (-2,3);

\draw[-] (1) to [out=-90,in=90] (4,-0.5);
\draw[-] (2) to [out=-90,in=90] (-4,-0.5);
\draw[-] (11) to [out=-90,in=90] (2,-0.5);
\draw[-] (22) to [out=-90,in=90] (-2,-0.5);
\draw[-] (3) to [out=-90,in=90] (0,-2.5);

\draw[-] (4,-0.5) to [out=-90,in=-90] (2,-0.5);
\draw[-] (3,-1) to [out=-90,in=90] (3,-2.5);
\draw[-] (3,-2.5) to [out=-90,in=-90] (0,-2.5);
\draw[-] (-4,-0.5) to [out=-90,in=-90] (-2,-0.5);
\draw[-] (-3,-1) to [out=-90,in=90] (-3, -5);
\draw[-] (1.5,-3.3) to [out=-90,in=90] (1.5,-5);
\draw[-] (1.5,-5) to [out=-90,in=-90] (-3,-5);
\draw[-] (4) to [out=90,in=-90] (-0.7,-6.3);
\end{tikzpicture}
\begin{tikzpicture}[x=8.5pt,y=8.5pt,thick]\pgfsetlinewidth{0.5pt}
\node[inner sep=1pt] at (0,0) {$\,\overset{\eqref{def: mY}}{=}\,$};
\node at (0,-9) {};
\end{tikzpicture}
\begin{tikzpicture}[x=8.5pt,y=8.5pt,thick]\pgfsetlinewidth{0.5pt}
\node[circle,draw, inner sep=0.6pt](1) at (2,-0.5) {$\scriptstyle{\beta}$};
\node[circle,draw, inner sep=0.6pt](2) at (-2,-0.5) {$\scriptstyle{\beta}$};
\node[circle,draw, inner sep=1pt](3) at (0,2.5) {$\scriptstyle{g}$};
\node(4) at (-0.5,-9) {$\scriptstyle{A}$};

\draw[-] (3,2) to [out=90,in=-90] (3,4.5);
\draw[-] (-3,2) to [out=90,in=-90] (-3,4.5);
\draw[-] (3,4.5) to [out=90,in=90] (-3,4.5);
\draw[-] (1.5,2) to [out=90,in=-90] (1.5,3.5);
\draw[-] (1.5,3.5) to [out=90,in=90] (-1.5,3.5);

\draw[-] (3,2) to [out=-90,in=-90] (1.5,2);
\draw[-] (-3,2) to [out=-90,in=-90]  (-1.5,2);
\draw[-] (-1.5,2) to [out=-90,in=90] (-1.5,3.5);
\draw[-] (3) to [out=-90,in=90] (0,-3);

\draw[-] (1) to [out=-90,in=90] (2,-3);
\draw[-] (2) to [out=-90,in=90] (-2,-5);
\draw[-] (2,-3) to [out=-90,in=-90] (0,-3);
\draw[-] (1,-3.5) to [out=-90,in=90] (1,-5);
\draw[-] (-2,-5) to [out=-90,in=-90] (1, -5);
\draw[-] (4) to [out=90,in=-90] (-0.5,-5.9);
\end{tikzpicture}
\begin{tikzpicture}[x=8.5pt,y=8.5pt,thick]\pgfsetlinewidth{0.5pt}
\node[inner sep=1pt] at (0,0) {$\,\overset{\eqref{form:ev1}}{=}\,$};
\node at (0,-9) {};
\end{tikzpicture}
\begin{tikzpicture}[x=8.5pt,y=8.5pt,thick]\pgfsetlinewidth{0.5pt}
\node[circle,draw, inner sep=1pt](3) at (0,0.5) {$\scriptstyle{g}$};
\node(4) at (0,-4) {$\scriptstyle{A}$};
\node(1) at (0,-9) { };

\draw[-] (1,2) to [out=-90,in=-90] (-1,2);
\draw[-] (1,2) to [out=90,in=-90] (1,4);
\draw[-] (-1,2) to [out=90,in=-90] (-1,4);
\draw[-] (1,4) to [out=90,in=90] (-1,4);
\draw[-] (3) to [out=-90,in=90] (4);
\end{tikzpicture}
\begin{tikzpicture}[x=8.5pt,y=8.5pt,thick]\pgfsetlinewidth{0.5pt}
\node[inner sep=1pt] at (0,0) {$\,=\,$};
\node at (0,-9) {};
\end{tikzpicture}
\begin{tikzpicture}[x=8.5pt,y=8.5pt,thick]\pgfsetlinewidth{0.5pt}
\node[circle,draw, inner sep=1pt](3) at (0,2.5) {$\scriptstyle{g}$};
\node(4) at (0,-3) {$\scriptstyle{A}$};
\node(5) at (0,-9) {};

\draw[-] (3) to [out=-90,in=90] (4);
\end{tikzpicture}
\end{equation}
Thus $\B{\sigma} _{X^{r}}^{R,S}\circ \B{\sigma} _{
X}^{S,R}=\mathrm{Id}_{\ZS{A}}$. Similarly one gets $\B{\sigma}
_{X}^{S,R}\circ \B{\sigma} _{ X^{r}}^{R,S}=\mathrm{Id}_{\ZR{A}}$. Therefore $\B{\sigma} _{X}^{S,R}$ is an isomorphism. This proves that $\Phi^{S,R}$ is well-defined.

Now assume that $R=S$. By Corollary \ref{coro:InvGrp}, we know that $(\mathrm{Inv}_{R}\left( A\right),\tensor{R},R) $ is a group.
Let us check that $\Phi ^{R}$ is a group morphism. Take $X,X' \in \mathrm{Inv}_{R}\left( A\right)$ with two-sided inverses $Y$ and $Y'$, respectively.  We have
\begin{center}
\begin{tikzpicture}[x=8pt,y=8pt,thick]\pgfsetlinewidth{0.5pt}
\node[inner sep=1pt] at (0,0) {$\B{\sigma}_X(\B{\sigma}_{X'}(g))\,=\,$};
\node at (0,-9) {};
\end{tikzpicture}
\begin{tikzpicture}[x=8.5pt,y=8.5pt,thick]\pgfsetlinewidth{0.5pt}
\node[circle,draw, inner sep=0.7pt](1) at (0,1) {$\scriptstyle{i_Y}$};
\node[circle,draw, inner sep=0.7pt](2) at (-7,1) {$\scriptstyle{i_{X}}$};
\node(0) at (0,-9) { };
\node(4) at (-4,-9) {$\scriptstyle{A}$};
\node[ellipse,draw, inner sep=1pt](5) at (-3.5,2) {$\scriptstyle{\B{\sigma}_{X'}(g)}$};

\draw[-] (1) to [out=90,in=-90] (0,3);
\draw[-] (2) to [out=90,in=-90] (-7,3);
\draw[-] (-7,3) to [out=90,in=90] (0,3);
\draw[-] (1) to [out=-90,in=90] (0,-1.5);
\draw[-] (2) to [out=-90,in=90] (-7,-3.3);
\draw[-] (-7,-3.3) to [out=-90,in=-90] (-1.5,-3.3);
\draw[-] (0,-1.5) to [out=-90,in=0] (-1.5,-2.5);

\draw[-] (5) to [out=-90,in=90] (-3.5,-1.5);
\draw[-] (-3.5,-1.5) to [out=-90,in=180] (-1.5,-2.5);
\draw[-] (-1.5,-2.5) to [out=-90,in=90] (-1.5,-3.3);
\draw[-] (4) to [out=90,in=-90] (-4,-4.9);

\end{tikzpicture} $\qquad$
\begin{tikzpicture}[x=8.5pt,y=8.5pt,thick]\pgfsetlinewidth{0.5pt}
\node[inner sep=1pt] at (0,0) {$\B{\sigma}_{X\tensor{R}X'}(g)\,=\,$};
\node at (0,-9) {};
\end{tikzpicture}
\begin{tikzpicture}[x=8pt,y=8pt,thick]\pgfsetlinewidth{0.5pt}
\node[circle,draw, inner sep=0.5pt](1) at (4,1) {$\scriptstyle{i_Y}$};
\node[circle,draw, inner sep=0.5pt](2) at (-4,1) {$\scriptstyle{i_{X}}$};
\node[circle,draw, inner sep=0.5pt](11) at (2,2) {$\scriptstyle{i_{Y'}}$};
\node[circle,draw, inner sep=0.5pt](22) at (-2,2) {$\scriptstyle{i_{X'}}$};
\node[circle,draw, inner sep=1pt](3) at (0,2.5) {$\scriptstyle{g}$};
\node(4) at (-0.7,-9) {$\scriptstyle{A}$};

\draw[-] (1) to [out=90,in=-90] (4,3.5);
\draw[-] (2) to [out=90,in=-90] (-4,3.5);
\draw[-] (4,3.5) to [out=90,in=90] (-4,3.5);
\draw[-] (11) to [out=90,in=-90] (2,3);
\draw[-] (22) to [out=90,in=-90] (-2,3);
\draw[-] (2,3) to [out=90,in=90] (-2,3);

\draw[-] (1) to [out=-90,in=90] (4,-0.5);
\draw[-] (2) to [out=-90,in=90] (-4,-0.5);
\draw[-] (11) to [out=-90,in=90] (2,-0.5);
\draw[-] (22) to [out=-90,in=90] (-2,-0.5);
\draw[-] (3) to [out=-90,in=90] (0,-2.5);

\draw[-] (4,-0.5) to [out=-90,in=-90] (2,-0.5);
\draw[-] (3,-1) to [out=-90,in=90] (3,-2.5);
\draw[-] (3,-2.5) to [out=-90,in=-90] (0,-2.5);
\draw[-] (-4,-0.5) to [out=-90,in=-90] (-2,-0.5);
\draw[-] (-3,-1) to [out=-90,in=90] (-3, -5);
\draw[-] (1.5,-3.3) to [out=-90,in=90] (1.5,-5);
\draw[-] (1.5,-5) to [out=-90,in=-90] (-3,-5);
\draw[-] (4) to [out=90,in=-90] (-0.7,-6.3);
\end{tikzpicture}
\end{center}
where $g$ is as above. The second equality in diagram \eqref{diag:Boh} implies that $\B{\sigma}_X(\B{\sigma}_{X'}(g))=\B{\sigma}_{X\tensor{R}X'}(g)$.
In other words  $\Phi ^{R}(X\tensor{R}X')=\Phi ^{R}(X)\circ\Phi ^{R}(X')$.
Moreover $$\Phi ^{R}(R)(g)=\B{\sigma}_{R}(g)\overset{\eqref{form:defsigmaX}}{=}m_A^R\circ(\alpha\tensor{R}(g\trir{R}\alpha))\circ (r_R^R)^{-1}\overset{(*)}{=}m_A^R\circ(\alpha\tensor{R}g)\circ (r_R^R)^{-1}=\alpha\trir{R}g=g$$ so that $\Phi ^{R}(R)=\id_{\ZR{A}},$ where in $(*)$ we applied the definition of $\trir{R}$. \smallskip

Let us check that $\Phi^{R}$ factors as stated. Take $p \in \ZR{R}$ and set $g:= \ZR{\alpha}(p)=\alpha \circ p$. Since we already know that $\sigma_X$ is multiplicative, in order to conclude that $\sigma_X$ is $\ZR{R}$-bilinear, it suffices to check that $\sigma_X(g)=g$.  As $p$ is an endomorphism of the unit $R$ of the monoidal category ${}_R\cat{M}_R$, we have $R\tensor{R}p = p\tensor{R}R$. Hence, by using the same arguments of  diagram \eqref{diag:mah}, we get the desired equality.
\end{proof}

\begin{remark}\label{rem:bicat}
Sub-monoids of $A$ are $0$-cells in a bicategory $\mathrm{Inv}^r(A)$, where for any  two $0$-cells $R$ and $S$ the associated hom-category from $R$ to $S$ is given by  $\mathrm{Inv}^r_{R,\, S}(A)$ with horizontal and vertical multiplications given by the functors of Proposition \ref{pro: InvTens}.  Similarly, one defines the bicategories $\mathrm{Inv}^l(A)$ and $\mathrm{Inv}(A)$.
Now, if we restrict to sub-monoids $R$ of $A$ for which the functor $-\tensor{R}A$ reflects isomorphisms and denote this new bicategory by $\mathrm{Invf}^r(A)$, then any morphism in the hom-category $\mathrm{Invf}_{R,S}^r(A)$ is an isomorphism. Indeed, given a morphism $h: X \to X'$ in $\mathrm{Inv}^r_{R,\, S}(A)$, then $f_{X'} \circ (h\tensor{R}A) =f_{X}$, where $f_{-}$ is as in equation \eqref{Eq:fx}. Therefore,  by Corollary \ref{coro:conaq}, we have $h\tensor{R}A$ is an isomorphism and hence  $h$ is an isomorphism as well.

In this way,    the bicategory $\mathrm{Invf}(A)$, corresponding to $\mathrm{Inv}(A)$, can be regarded as a \emph{$2$-groupoid}, that is, a bicategory where $1$-cells and $2$-cells are invertible (i.e.~each $1$-cell is a member of an internal equivalence and any $2$-cell is an isomorphism). However, $\mathrm{Inv}(A)$ viewed as a small category with hom-sets $\mathrm{Inv}_{R,S}(A)$ and composition given by the tensor products $\tensor{R}$, is clearly a groupoid.

On the other hand, there is another $2$-category with $0$-cells given by $\Z{\I}$-subalgebras of $\Z{A}$, and hom-category from $E$ and $E'$ given by the set  $\mathrm{Hom}_{\scriptscriptstyle{\Z{\I}}\text{-alg}}\big(E,E'\big)$ (objects are edges and arrows are squares). Clearly, one assigns to any $1$-cell $R$ in $\mathrm{Inv}^r(A)$  a $1$-cell $\ZR{A}$ in this $2$-category, by using the map $\varphi_R$ of \eqref{Eq:varphi}. However,  it is not clear to us whether the family of maps $\{\B{\sigma}^{S,R}\}_{\scriptscriptstyle{S,\, R}}$ defined in  Proposition \ref{pro: right Miyashita} together with this assignment, give rise to a family of functors and thus to a morphism of bicategories. Nevertheless,  using the maps $\{\Phi^{S,R}\}_{\scriptscriptstyle{S,\, R}}$ given in Proposition \ref{pro: Phi}, one shows, as in the classical case \cite[Theorem 1.3]{Mi},   that $\Phi$ establishes a homomorphism of groupoids.
\end{remark}

\subsection{Invariant subobjects}\label{ssec:invariants}

Let $(A,m,u)$ be a monoid in $\cat{M}$ and consider  $\sigma, \delta : \Z{A}  \to \Z{A}$ two  $\Z{\I}$-linear maps.  For each element $t \in \Z{A}$, define the following equalizer:
$$\xymatrix@C=60pt{ 0 \ar[r] & Eq(\sigma(t)\triangleright A, A\triangleleft\delta(t)) \ar@{->}^-{\fk{eq}_{\sigma(t),\, \delta(t)}}[r] & A \ar@<0.5ex>@{->}^-{\sigma(t)\,\triangleright A}[r]  \ar@<-0.5ex>@{->}_-{A\triangleleft\,\delta(t)}[r] & A } $$
Now, since $\cat{M}$ is an abelian and bicomplete category, we can take the intersection of all those equalizes:
\begin{equation}\label{Eq:J}
{}_{\sigma}J_{\delta}\,:=\, \left(\bigcap_{t \in \Z{A}} Eq(\sigma(t)\triangleright A, A\triangleleft\delta(t))\right) \hookrightarrow Eq(\sigma(t)\triangleright A, A\triangleleft\delta(t)), \text{ for every } t \in \Z{A}
\end{equation}
with structure monomorphism denoted by
\begin{equation}\label{Eq:strmaps}
\xymatrix@C=50pt{ 0 \ar@{->}^-{}[r] & {}_{\sigma}J_{\delta}  \ar@{->}^-{\fk{eq}_{\sigma,\delta}}[r] & A.}
\end{equation} We view ${}_{\sigma}J_{\delta}$ as an element in $\mathscr{P}({}_{\mathbb{I}}A_{\I})$ with monomorphism $i_{{}_{\sigma}J_{\delta}}=\fk{eq}_{\sigma,\delta}$ as in Section \ref{sec:biobjects}.
The pair $({}_{\sigma}J_{\delta},\fk{eq}_{\sigma,\delta})$ is in fact an  universal object with respect to the following equations:
\begin{equation}\label{eq:equalizer}
(A\triangleleft\delta(t)) \circ \fk{eq}_{\sigma,\delta} \,\,=\,\,  (\sigma(t) \triangleright A)  \circ \fk{eq}_{\sigma,\delta}, \text{ for every } t \in \Z{A}.
\end{equation}
In other words  any other morphism which satisfies these equations (parameterized by elements  in $\Z{A}$) factors uniquely  throughout the monomorphism $\fk{eq}_{\sigma,\delta}$.
For every $t \in \Z{A}$, equation \eqref{eq:equalizer} will be used in the following  form
\begin{equation}\label{Eq:seraloqsera}
\begin{tikzpicture}[x=8pt,y=8pt,thick]\pgfsetlinewidth{0.5pt}
\node(1) at (0.5,6){$\scriptstyle{{}_{\sigma}J_{\delta}}$};
\node[circle,draw, inner sep=1pt](2) at (0.5,3) {$\scriptstyle{\fk{eq}_{\sigma,\delta}}$};
\node[circle,draw, inner sep=1pt] (3) at (2.5,0.5) {$\scriptstyle{\delta(t)}$};
\node(4) at (1.5,-5.5) {$\scriptstyle{A}$};

\draw[-] (1) to [out=-90, in=90] (2);
\draw[-] (2) to [out=-90, in=90] (0.5,-2.5);
\draw[-] (3) to [out=-90, in=90] (2.5,-2.5);
\draw[-] (0.5,-2.5) to [out=-90, in=-90] (2.5,-2.5);
\draw[-] (1.5,-3) to [out=-90, in=90] (4);
\end{tikzpicture}
\begin{tikzpicture}[x=8pt,y=8pt,thick]\pgfsetlinewidth{0.5pt}
\node[inner sep=1pt] at (0,0) {$\,\,=\,\,$};
\node at (0,-6) {};
\end{tikzpicture}
\begin{tikzpicture}[x=8pt,y=8pt,thick]\pgfsetlinewidth{0.5pt}
\node(1) at (0.5,6){$\scriptstyle{{}_{\sigma}J_{\delta}}$};
\node[circle,draw, inner sep=1pt](2) at (0.5,3) {$\scriptstyle{\fk{eq}_{\sigma,\delta}}$};
\node[circle,draw, inner sep=1pt] (3) at (-1.5,0.5) {$\scriptstyle{\sigma(t)}$};
\node(4) at (-0.5,-5.5) {$\scriptstyle{A}$};

\draw[-] (1) to [out=-90, in=90] (2);
\draw[-] (2) to [out=-90, in=90] (0.5,-2.5);
\draw[-] (3) to [out=-90, in=90] (-1.5,-2.5);
\draw[-] (0.5,-2.5) to [out=-90, in=-90] (-1.5,-2.5);
\draw[-] (-0.5,-3) to [out=-90, in=90] (4);
\end{tikzpicture}
\end{equation}

For simplicity we write ${}_1J_{1}$,  ${}_{\sigma}J_{1}$ and ${}_1J_{\delta}$ when one or both the involved $\Z{\I}$-linear maps are identity.
\begin{proposition}\label{lema:-1}
Let $(A,m,u)$ be a monoid in $\cat{M}$ and  let $\sigma$, $\delta$,  $\gamma$ and $\lambda$ be in $\rm{End}_{\scriptscriptstyle{\Z{\I}}}(\Z{A})$. Consider the family of subobjects  $({}_{x}J_{y}, \fk{eq}_{x,y})$ of $A$, for $x, y \in \{\sigma, \delta, \gamma, \lambda\}$.   Then
\begin{enumerate}[(i)]
\item There is a morphism
$$ m_{\sigma, \gamma}^{\delta}: {}_{\sigma}J_{\delta} \tensor{} {}_{\delta}J_{\gamma} \longrightarrow {}_{\sigma}J_{\gamma}$$ such that
\begin{equation}\label{Eq:mu3}
\fk{eq}_{\sigma,\gamma} \circ   m_{\sigma, \gamma}^{\delta} \,=\, m \circ  (\fk{eq}_{\sigma,\delta}\tensor{}\fk{eq}_{\delta,\gamma}).
\end{equation}
\item We have
\begin{equation}\label{Eq:mu4}
 m_{\sigma, \lambda}^{\gamma}\circ (m_{\sigma, \gamma}^{\delta}\tensor {} {_\gamma J_\lambda})=m_{\sigma, \lambda}^{\delta}\circ (_\sigma J_\delta\tensor {} m_{\delta, \lambda}^{\gamma}).
\end{equation}
\item There exists a unique morphism  $u_\sigma :\I\rightarrow {}_{\sigma}J_{\sigma}$ such that $\fk{eq}_{\sigma,\sigma}\circ u_\sigma=u.$ Furthermore $({}_{\sigma}J_{\sigma},m_\sigma,u_\sigma)$ is a monoid, for every $\Z{\I}$-linear map $\sigma$, where we set
$m_\sigma:=m_{\sigma, \sigma}^{\sigma}$.

\item The morphisms $m_{\sigma, \delta}^{\sigma}$ and $m_{\sigma, \delta}^{\delta}$ turn ${}_{\sigma}J_{\delta}$ into a $({}_{\sigma}J_{\sigma},{}_{\delta}J_{\delta})$-bimodule.
\item There is a unique morphism
\begin{equation}\label{Eq:mu5}
\overline{m}_{\sigma, \gamma}^{\delta}: ({}_{\sigma}J_{\delta} )\tensor{_\delta J_\delta} {}({_{\delta}J_{\gamma}}) \longrightarrow {}{_{\sigma}J_{\gamma}} \, \text{ such that } \,
 \overline{m}_{\sigma, \gamma}^{\delta}\circ \chi_{\sigma, \gamma}^{\delta} =m_{\sigma, \gamma}^{\delta}
\end{equation}
where $\chi_{\sigma, \gamma}^{\delta}:{}_{\sigma}J_{\delta} \tensor{} {}_{\delta}J_{\gamma}\rightarrow ({}_{\sigma}J_{\delta} )\tensor{_{\delta}J_{\delta}} {}(_{\delta}J_{\gamma})$ denotes the canonical morphism defining the tensor product over the monoid ${}_{\delta}J_{\delta}$.
\end{enumerate}
\end{proposition}
\begin{proof}
$(i)$.  Fixing an element $t \in \Z{A}$, we have
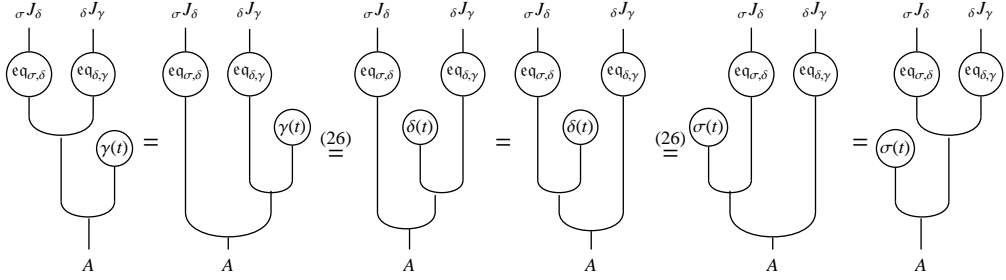
\begin{center}
\begin{tikzpicture}[x=8pt,y=8pt,thick]\pgfsetlinewidth{0.5pt}
\node(1) at (-3,6) {$\scriptstyle{{}_{\sigma}J_{\delta}}$};
\node(2) at (0,6) {$\scriptstyle{{}_{\delta}J_{\gamma}}$};
\node(5) at (-0.2,-6) {$\scriptstyle{A}$};
\node[circle, draw,inner sep=1pt](11) at (-3,3) {$\scriptstyle{\fk{eq}_{\sigma,\delta}}$};
\node[circle, draw,inner sep=1pt](22) at (0,3) {$\scriptstyle{\fk{eq}_{\delta,\gamma}}$};
\node[circle, draw,inner sep=0.5pt](33) at (1,-0.5) {$\scriptstyle{\gamma(t)}$};

\draw[-] (1) to [out=-90,in=90] (11);
\draw[-] (2) to [out=-90,in=90] (22);
\draw[-] (33) to [out=-90,in=90] (1,-3);

\draw[-] (11) to [out=-90,in=90] (-3,1);
\draw[-] (22) to [out=-90,in=90] (0,1);
\draw[-] (-3,1) to [out=-90,in=-90] (0,1);
\draw[-] (-1.5,0) to [out=-90,in=90] (-1.5,-3);
\draw[-] (-1.5,-3) to [out=-90,in=-90] (1,-3);
\draw[-] (5) to [out=90,in=-90] (-0.2,-3.8);
\end{tikzpicture}
\begin{tikzpicture}[x=8pt,y=8pt,thick]\pgfsetlinewidth{0.5pt}
\node[inner sep=1pt] at (0,0) {$=$};
\node at (0,-6) {};
\end{tikzpicture}%
\begin{tikzpicture}[x=8pt,y=8pt,thick]\pgfsetlinewidth{0.5pt}
\node(1) at (-3,6) {$\scriptstyle{{}_{\sigma}J_{\delta}}$};
\node(2) at (0,6) {$\scriptstyle{{}_{\delta}J_{\gamma}}$};
\node(5) at (-1,-6) {$\scriptstyle{A}$};
\node[circle, draw,inner sep=1pt](11) at (-3,3) {$\scriptstyle{\fk{eq}_{\sigma,\delta}}$};
\node[circle, draw,inner sep=1pt](22) at (0,3) {$\scriptstyle{\fk{eq}_{\delta,\gamma}}$};
\node[circle, draw,inner sep=0.5pt](33) at (2,0.5) {$\scriptstyle{\gamma(t)}$};

\draw[-] (1) to [out=-90,in=90] (11);

\draw[-] (2) to [out=-90,in=90] (22);
\draw[-] (33) to [out=-90,in=90] (2,-2);

\draw[-] (11) to [out=-90,in=90] (-3,-3.5);
\draw[-] (22) to [out=-90,in=90] (0,-2);
\draw[-] (0,-2) to [out=-90,in=-90] (2,-2);
\draw[-] (1,-2.5) to [out=-90,in=90] (1,-3.5);
\draw[-] (1,-3.5) to [out=-90,in=-90] (-3,-3.5);
\draw[-] (5) to [out=90,in=-90] (-1,-4.7);
\end{tikzpicture}
\begin{tikzpicture}[x=8pt,y=8pt,thick]\pgfsetlinewidth{0.5pt}
\node[inner sep=1pt] at (0,0) {$\overset{\eqref{Eq:seraloqsera}}{=}$};
\node at (0,-6) {};
\end{tikzpicture}%
\begin{tikzpicture}[x=8pt,y=8pt,thick]\pgfsetlinewidth{0.5pt}
\node(1) at (-2,6) {$\scriptstyle{{}_{\sigma}J_{\delta}}$};
\node(2) at (2,6) {$\scriptstyle{{}_{\delta}J_{\gamma}}$};
\node(5) at (-0.5,-6) {$\scriptstyle{A}$};
\node[circle, draw,inner sep=1pt](11) at (-2,3) {$\scriptstyle{\fk{eq}_{\sigma,\delta}}$};
\node[circle, draw,inner sep=1pt](22) at (2,3) {$\scriptstyle{\fk{eq}_{\delta,\gamma}}$};
\node[circle, draw,inner sep=0.5pt](33) at (0,0.5) {$\scriptstyle{\delta(t)}$};

\draw[-] (1) to [out=-90,in=90] (11);
\draw[-] (2) to [out=-90,in=90] (22);
\draw[-] (33) to [out=-90,in=90] (0,-2);

\draw[-] (11) to [out=-90,in=90] (-2,-3.5);
\draw[-] (22) to [out=-90,in=90] (2,-2);
\draw[-] (0,-2) to [out=-90,in=-90] (2,-2);
\draw[-] (0.7,-2.7) to [out=-90,in=90] (0.7,-3.5);
\draw[-] (0.7,-3.5) to [out=-90,in=-90] (-2,-3.5);
\draw[-] (5) to [out=90,in=-90] (-0.5,-4.3);
\end{tikzpicture}
\begin{tikzpicture}[x=8pt,y=8pt,thick]\pgfsetlinewidth{0.5pt}
\node[inner sep=1pt] at (0,0) {$=$};
\node at (0,-6) {};
\end{tikzpicture}%
\begin{tikzpicture}[x=8pt,y=8pt,thick]\pgfsetlinewidth{0.5pt}
\node(1) at (-2,6) {$\scriptstyle{{}_{\sigma}J_{\delta}}$};
\node(2) at (2,6) {$\scriptstyle{{}_{\delta}J_{\gamma}}$};
\node(5) at (0.5,-6) {$\scriptstyle{A}$};
\node[circle, draw,inner sep=1pt](11) at (-2,3) {$\scriptstyle{\fk{eq}_{\sigma,\delta}}$};
\node[circle, draw,inner sep=1pt](22) at (2,3) {$\scriptstyle{\fk{eq}_{\delta,\gamma}}$};
\node[circle, draw,inner sep=0.5pt](33) at (0,0.5) {$\scriptstyle{\delta(t)}$};

\draw[-] (1) to [out=-90,in=90] (11);
\draw[-] (2) to [out=-90,in=90] (22);
\draw[-] (33) to [out=-90,in=90] (0,-2);

\draw[-] (11) to [out=-90,in=90] (-2,-2);
\draw[-] (22) to [out=-90,in=90] (2,-3.5);
\draw[-] (-2,-2) to [out=-90,in=-90] (0,-2);
\draw[-] (-1,-2.7) to [out=-90,in=90] (-1,-3.5);
\draw[-] (-1,-3.5) to [out=-90,in=-90] (2,-3.5);
\draw[-] (5) to [out=90,in=-90] (0.5,-4.4);
\end{tikzpicture}
\begin{tikzpicture}[x=8pt,y=8pt,thick]\pgfsetlinewidth{0.5pt}
\node[inner sep=1pt] at (0,0) {$\overset{\eqref{Eq:seraloqsera}}{=}$};
\node at (0,-6) {};
\end{tikzpicture}%
\begin{tikzpicture}[x=8pt,y=8pt,thick]\pgfsetlinewidth{0.5pt}
\node(1) at (-3,6) {$\scriptstyle{{}_{\sigma}J_{\delta}}$};
\node(2) at (0,6) {$\scriptstyle{{}_{\delta}J_{\gamma}}$};
\node(5) at (-2,-6) {$\scriptstyle{A}$};
\node[circle, draw,inner sep=1pt](11) at (-3,3) {$\scriptstyle{\fk{eq}_{\sigma,\delta}}$};
\node[circle, draw,inner sep=1pt](22) at (0,3) {$\scriptstyle{\fk{eq}_{\delta,\gamma}}$};
\node[circle, draw,inner sep=0.5pt](33) at (-5,0.5) {$\scriptstyle{\sigma(t)}$};

\draw[-] (1) to [out=-90,in=90] (11);
\draw[-] (2) to [out=-90,in=90] (22);
\draw[-] (33) to [out=-90,in=90] (-5,-2);
\draw[-] (11) to [out=-90,in=90] (-3,-2);
\draw[-] (22) to [out=-90,in=90] (0,-3.5);
\draw[-] (-5,-2) to [out=-90,in=-90] (-3,-2);
\draw[-] (-4,-2.5) to [out=-90,in=90] (-4,-3.5);
\draw[-] (-4,-3.5) to [out=-90,in=-90] (0,-3.5);
\draw[-] (5) to [out=90,in=-90] (-2,-4.7);
\end{tikzpicture}
\begin{tikzpicture}[x=8pt,y=8pt,thick]\pgfsetlinewidth{0.5pt}
\node[inner sep=1pt] at (0,0) {$\,=\,$};
\node at (0,-6) {};
\end{tikzpicture}%
\begin{tikzpicture}[x=8pt,y=8pt,thick]\pgfsetlinewidth{0.5pt}
\node(1) at (-3,6) {$\scriptstyle{{}_{\sigma}J_{\delta}}$};
\node(2) at (0,6) {$\scriptstyle{{}_{\delta}J_{\gamma}}$};
\node(5) at (-2.8,-6) {$\scriptstyle{A}$};
\node[circle, draw,inner sep=1pt](11) at (-3,3) {$\scriptstyle{\fk{eq}_{\sigma,\delta}}$};
\node[circle, draw,inner sep=1pt](22) at (0,3) {$\scriptstyle{\fk{eq}_{\delta,\gamma}}$};
\node[circle, draw,inner sep=0.5pt](33) at (-4,-0.5) {$\scriptstyle{\sigma(t)}$};

\draw[-] (1) to [out=-90,in=90] (11);
\draw[-] (2) to [out=-90,in=90] (22);
\draw[-] (33) to [out=-90,in=90] (-4,-3);
\draw[-] (11) to [out=-90,in=90] (-3,1);
\draw[-] (22) to [out=-90,in=90] (0,1);

\draw[-] (-3,1) to [out=-90,in=-90] (0,1);
\draw[-] (-1.5,0) to [out=-90,in=90] (-1.5,-3);
\draw[-] (-1.5,-3) to [out=-90,in=-90] (-4,-3);
\draw[-] (5) to [out=90,in=-90] (-2.8,-3.8);
\end{tikzpicture}
\end{center}
which shows that
\begin{eqnarray*}
(\sigma(t) \triangleright A) \circ m \circ  (\fk{eq}_{\sigma,\delta}\tensor{}\fk{eq}_{\delta,\gamma})  &=&    (A\triangleleft \gamma(t) ) \circ m \circ (\fk{eq}_{\sigma,\delta}\tensor{}\fk{eq}_{\delta,\gamma}).
\end{eqnarray*}
Hence, by the universal property, the desired morphism is the one which turns commutative the diagrams
\begin{equation}\label{Eq:diag}
  \xymatrixcolsep{2.5cm}
\xymatrix{
  {}_{\sigma}J_{\delta} \tensor{} {}_{\delta}J_{\gamma}  \ar[r]^{\fk{eq}_{\sigma,\delta}\tensor{}\fk{eq}_{\delta,\gamma}}  \ar@{-->}_-{m^{\delta}_{\sigma,\gamma}}[d]& A\otimes A\ar[d]^{m} && \\
 {}_{\sigma}J_{\gamma}  \ar[r]^{\fk{eq}_{\sigma,\gamma}} & A \ar@<.5ex>[rr]^{\sigma(t)\, \triangleright\, A} \ar@<-.5ex>[rr]_{A\,\triangleleft\,\gamma(t) } && A}
\end{equation}
for every $t \in \Z{A}$.

$(ii)$. The equation follows from the fact that $\fk{eq}_{\sigma,\lambda} $ is a monomorphism and the computation:
\begin{eqnarray*}
\fk{eq}_{\sigma,\lambda} \circ   m_{\sigma, \lambda}^{\delta}\circ (_\sigma J_\delta\tensor {} m_{\delta, \lambda}^{\gamma})
&\overset{\eqref{Eq:mu3}}{=} &m \circ  (\fk{eq}_{\sigma,\delta}\tensor{}\fk{eq}_{\delta,\lambda}) \circ (_\sigma J_\delta\tensor {} m_{\delta, \lambda}^{\gamma})
\\ &\overset{\eqref{Eq:mu3}}{=}& m  \circ (A\tensor{}m) \circ  (\fk{eq}_{\sigma,\delta}\tensor{}\fk{eq}_{\delta,\gamma}\tensor{}\fk{eq}_{\gamma,\lambda})
\\ &=& m  \circ (m\tensor{} A) \circ  (\fk{eq}_{\sigma,\delta}\tensor{}\fk{eq}_{\delta,\gamma}\tensor{}\fk{eq}_{\gamma,\lambda})
\\ &\overset{\eqref{Eq:mu3}}{=} &m \circ  (\fk{eq}_{\sigma,\gamma}\tensor{}\fk{eq}_{\gamma,\lambda}) \circ (m_{\sigma,\gamma}^\delta\tensor{}{_\gamma J_\lambda})
\\ &\overset{\eqref{Eq:mu3}}{=} &\fk{eq}_{\sigma,\lambda} \circ   m_{\sigma, \lambda}^{\gamma}\circ (m_{\sigma, \gamma}^{\delta}\tensor{} {_\gamma J_\lambda}).
\end{eqnarray*}

$(iii)$. The associativity follows by \eqref{Eq:mu4}. Denote $m^{\sigma}_{\sigma,\sigma}:=m_{\sigma}$ and ${}_{\sigma}J_{\sigma}:=J$.
 Now, we show that there is a morphism $u_\sigma$ such that  the following diagrams commute
\begin{equation}\label{Eq:triangle}
  \xymatrixcolsep{3cm}
\xymatrix{
  & \I\ar[d]^{u}\ar@{-->}[ld]_-{u_\sigma} && \\
 {}_{\sigma}J_{\sigma}  \ar[r]^{\fk{eq}_{\sigma,\sigma}} & A \ar@<.5ex>[rr]^{\sigma(t)\, \triangleright\, A} \ar@<-.5ex>[rr]_{A\, \triangleleft\, \sigma(t) } && A}
\end{equation}
We compute
\begin{eqnarray*}
m \circ (A\tensor{}\sigma(t)) \circ u&=&m \circ (u\tensor{}A)\circ (\I\tensor{}\sigma(t)) =l_A \circ  (\I\tensor{}\sigma(t))
\\&=&\sigma(t)\circ l_\I  =\sigma(t)\circ r_\I
\\&=&r_A \circ  (\sigma(t)\tensor{}\I)=m \circ (A\tensor{}u)\circ (\sigma(t)\tensor{}\I)
\\&=&m \circ (\sigma(t)\tensor{}A) \circ u.
\end{eqnarray*}
We now prove that the multiplication of $J$ is unitary:
\begin{eqnarray*}
 \fk{eq}_{\sigma,\sigma} \circ   m_{\sigma}  \circ (J\tensor{}u_{\sigma})&\overset{\eqref{Eq:mu3}}{=} &m \circ  (\fk{eq}_{\sigma,\sigma}\tensor{}\fk{eq}_{\sigma,\sigma}) \circ (J\tensor{}u_{\sigma})
 \\&=&m \circ  (A\tensor{}u) \circ (\fk{eq}_{\sigma,\sigma}\tensor{}\I)
  \\&=&r_A \circ (\fk{eq}_{\sigma,\sigma}\tensor{}\I)=\fk{eq}_{\sigma,\sigma}\circ r_J.
\end{eqnarray*}
Since $\fk{eq}_{\sigma,\sigma}$ is a monomorphism we obtain $m_{\sigma}  \circ (J\tensor{}u_{\sigma})=r_J$. In a similar manner one gets $m_{\sigma}  \circ (u_{\sigma}\tensor{}J)=l_J$. We have so proved that $(J,m_{\sigma} ,u_{\sigma})$ is a monoid.

$(iv)$. It follows by \eqref{Eq:mu4}.

$(v)$. From \eqref{Eq:mu4} one gets that $m_{\sigma, \gamma}^{\delta}$ is balanced over $_\delta J_\delta$.
\end{proof}

\section{The bijectivity of Miyashita action}\label{sec:azymaya}

The aim of this section is to seek for conditions under which the Miyashita action, that is the map of Proposition \ref{pro: Phi}, or some of its factors are bijective. We work, as before, over   a Penrose  monoidal abelian, locally small and bicomplete  category $(\cat{M},\tensor{}, \I, l,r)$, where the tensor product is right exact on both factors.

\subsection{The case when $\Phi$ is a monomorphism of groups}\label{ssec:mono}

Take $R=S=\I$ in Proposition \ref{pro: Phi} and  consider an invertible subobject $X \in {\rm Inv}_{\I}(A)$ with two-sided inverse $Y$. The image of any $t \in \Z{A}$  through the map $\Phi_X$ of equation \eqref{Eq:PHI} is given by
\begin{equation}\label{Eq:phiXt}
\Phi_X(t)\,=\,  m \circ \left(i_X\tensor{}(t\triangleright i_Y)\right) \circ \text{coev},
\end{equation}
where, $\text{coev}=m_X^{-1}$ is as in  Proposition \ref{pro:invdual}.

\begin{lemma}\label{lema:inclusion}
Let $(A, m,u)$ be a monoid in $\cat{M}$ and  let $X \in {\rm Inv}_{\I}(A)$ with inverse $Y$. Consider the associated automorphism $\Phi_X \in \Aut{\scriptscriptstyle{\Z{\I}}\text{-alg}}{\Z{A}}$ defined as in Proposition \ref{pro: Phi} with $R=S=\I$.  Then we have monomorphisms
$$
\xymatrix@C=45pt{ X\; \ar@{^(->}^-{\iota_X}[r] & {}_{\Phi_X}J_1, & Y \;\ar@{^(->}^-{\iota_Y}[r] & {}_1J_{\Phi_X} }
$$
satisfying $ \fk{eq}_{\Phi_X,1} \circ \iota_X\,=\, i_X$,  and $  \fk{eq}_{1,\Phi_X} \circ \iota_Y\,=\, i_Y$.
\end{lemma}

\begin{proof}
By applying equation \eqref{form:sigmaWiX} to our situation we obtain
\begin{equation}\label{Eq:muphi}
\Phi_X(t)\triangleright i_X\,=\, i_X\triangleleft t,\quad \text{ for every } t \in \Z{A}.
\end{equation}
Equivalently  $(\Phi_X(t)\triangleright A)  \circ i_X \,=\,  (A\triangleleft t) \circ i_X$ for every $t \in \Z{A}$, which by the universal property of ${}_{\Phi_X}J_1$ gives the desired monomorphism. Replacing $X$ by $Y$  one gets the other monomorphism.
\end{proof}

Notice that for any $\Z{\I}$-algebra automorphism $ \theta$ of $\Z{A}$, we have that for each  $s \in \Z{A}$, there exists a unique $t\in \Z{A}$ such that $\theta(t)\,=\, s$. Hence
\begin{equation}\label{Eq:egos}
{}_{\theta}J_{\theta}\,=\, \bigcap_{t \in \Z{A}} Eq(\theta(t)\triangleright A, A\triangleleft \theta(t)) \,=\, \bigcap_{s \in \Z{A}} Eq(s\triangleright A, A\triangleleft s)\,=\, {}_1J_1.
\end{equation}
This observation is implicitly used in the forgoing.

\begin{proposition}\label{prop:J}
Let $(A, m,u)$ be a monoid in $\cat{M}$ and assume that $u$ is a monomorphism. Let $X \in {\rm Inv}_{\I}(A)$ be an  invertible subobject with associated automorphism $\Phi_X$ as in \eqref{Eq:phiXt}. Then ${}_{\Phi_X}J_1$ is  a two-sided invertible ${}_{1}J_{1}$-subbimodule of $A$, that is an element of $\,{\rm Inv}_{{}_{1}J_{1}}(A)$ with inverse ${}_1J_{\Phi_X}$.
\end{proposition}
\begin{proof}
Along  this proof  we abbreviate $\Phi_X$ to $\phi$. The morphisms $m_{{}_{\phi}J_1}=\bara{m}^1_{\phi,\phi}$ and $m_{{}_1J_{\phi}}=\bara{m}^{\phi}_{1,1}$ defined in  Proposition \ref{pro:invdual}  for the pair $({}_{\phi}J_1, {}_1J_{\phi})$, are given by Proposition \ref{lema:-1}(v), where the compatibility constrains are shown using equation \eqref{Eq:mu4} and  \eqref{Eq:mu5}.  Consider the morphism
$$ \xymatrix@C=40pt{\psi :={}_{\phi }J_{\phi }\tensor{}X \ar@{->}^-{{_{\phi }J_{\phi }}
\otimes \iota _{X}}[r] & {}_{\phi }J_{\phi }\otimes  {}_{\phi}J_{1}  \ar@{->}^-{m _{\phi ,1}^{\phi }}[r] &   {}_{\phi }J_{1} }
 $$
where $\iota_X$  was defined in Lemma \ref{lema:inclusion}.
Now let us check that $m_{_\phi J_1}$ is invertible with inverse
$\nu$ defined as the following composition  of morphisms:
$$
\xymatrix@C=45pt{ {}_{\phi}J_{\phi} \ar@{->}_-{\cong}^-{{}_{\phi}J_{\phi}\tensor{}m_X^{-1}}[r] & {}_{\phi}J_{\phi}\tensor{}X\tensor{}Y \ar@{->}^-{\psi\tensor{}Y}[r] & {}_{\phi}J_{1} \tensor{}Y  \ar@{->}^-{{}_{\phi}J_1\tensor{}\iota_Y}[r] & {}_{\phi}J_{1}\tensor{} {}_{1}J_{\phi}  \ar@{->}^-{\chi_{\phi, \phi}^{1}}[r]  & {}_{\phi}J_{1}\tensor{{}_1J_1}{}_{1}J_{\phi}.   }
$$
We have
\begin{center}
\begin{tikzpicture}[x=9pt,y=9pt,thick]\pgfsetlinewidth{0.5pt}
\node(4) at (0,-10) {$\scriptstyle{{}_{\phi}J_{\phi}}$};
\node[rectangle,draw, inner sep=4pt, text width=1cm, text centered](3) at (-3,-1) {$\scriptstyle{\psi}$};
\node[circle,draw, inner sep=0.7pt](2) at (0.5,-1) {$\scriptstyle{\iota_Y}$};
\node(5) at (-4,3) {$\scriptstyle{{}_{\phi}J_{\phi}}$};
\node[rectangle,draw, inner sep=3pt, text width=1.2cm, text centered](1) at (0,-3.7) {$\scriptstyle{\chi^1_{\phi,\phi}}$};
\node[circle,draw, inner sep=0.7pt](6) at (0,-6.5) {$\scriptstyle{\bara{m}^1}$};

\draw[-] (5) to [out=-90, in=90] (-4,-0.2);
\draw[-] (-2,1.2) to [out=90, in=90] (0.5,1.2);
\draw[-] (0.5,1.2) to [out=-90, in=90] (2);
\draw[-] (-2,1.2) to [out=-90, in=90] (-2,-0.2);
\draw[-] (2) to [out=-90, in=90] (0.5,-2.8);
\draw[-] (-2,-1.8) to [out=-90, in=90] (-1,-2.8);
\draw[-] (1) to [out=-90, in=90] (6);
\draw[-] (6) to [out=-90, in=90] (4);
\end{tikzpicture}
\begin{tikzpicture}[x=9pt,y=9pt,thick]\pgfsetlinewidth{0.5pt}
\node[inner sep=1pt] at (0,-3) {$\,\overset{\eqref{Eq:mu5}}{=}\,$};
\node at (0,-10) {};
\end{tikzpicture}
\begin{tikzpicture}[x=9pt,y=9pt,thick]\pgfsetlinewidth{0.5pt}
\node(1) at (-4,5){$\scriptstyle{{}_{\phi}J_{\phi}}$};
\node(4) at (-2,-6.5) {$\scriptstyle{{}_{\phi}J_{\phi}}$};
\node[circle,draw, inner sep=0.7pt](3) at (-2.5,1) {$\scriptstyle{\iota_X}$};
\node[circle,draw, inner sep=0.7pt](2) at (-0.5,1) {$\scriptstyle{\iota_Y}$};

\draw[-] (-2.5,3) to [out=90, in=90] (-0.5,3);
\draw[-] (-0.5,3) to [out=-90, in=90] (2);
\draw[-] (-2.5,3) to [out=-90, in=90] (3);
\draw[-] (2) to [out=-90, in=90] (-0.5,-2);
\draw[-] (-3.2,-2) to [out=-90, in=-90] (-0.5,-2);
\draw[-] (-3.2,-2) to [out=90, in=-90] (-3.2,-1.5);
\draw[-] (3) to [out=-90, in=90] (-2.5,-1);
\draw[-] (-2,-2.7) to [out=-90, in=90] (4);
\draw[-] (-2.5,-1) to [out=-90, in=-90] (-4,-1);
\draw[-] (1) to [out=-90, in=90] (-4,-1);
\end{tikzpicture}
\begin{tikzpicture}[x=9pt,y=9pt,thick]\pgfsetlinewidth{0.5pt}
\node[inner sep=1pt] at (0,-3.5) {$\,\overset{\eqref{Eq:mu4}}{=}\,$};
\node at (0,-10) {};
\end{tikzpicture}
\begin{tikzpicture}[x=9pt,y=9pt,thick]\pgfsetlinewidth{0.5pt}
\node(1) at (-4,5){$\scriptstyle{{}_{\phi}J_{\phi}}$};
\node(4) at (-2.7,-6.5) {$\scriptstyle{{}_{\phi}J_{\phi}}$};
\node[circle,draw, inner sep=0.7pt](3) at (-2.5,1) {$\scriptstyle{\iota_X}$};
\node[circle,draw, inner sep=0.7pt](2) at (-0.5,1) {$\scriptstyle{\iota_Y}$};

\draw[-] (-2.5,3) to [out=90, in=90] (-0.5,3);
\draw[-] (-0.5,3) to [out=-90, in=90] (2);
\draw[-] (-2.5,3) to [out=-90, in=90] (3);
\draw[-] (2) to [out=-90, in=90] (-0.5,-2);
\draw[-] (-2.5,-2) to [out=-90, in=-90] (-0.5,-2);
\draw[-] (-1.5,-4) to [out=90, in=-90] (-1.5,-2.6);
\draw[-] (3) to [out=-90, in=90] (-2.5,-2);
\draw[-] (-4,-4) to [out=-90, in=-90] (-1.5,-4);
\draw[-] (-2.7,-4.7) to [out=-90, in=90] (4);
\draw[-] (1) to [out=-90, in=90] (-4,-4);
\end{tikzpicture}
\begin{tikzpicture}[x=9pt,y=9pt,thick]\pgfsetlinewidth{0.5pt}
\node[inner sep=1pt] at (0,-3.5) {$\,\overset{(\star)}{=}\,$};
\node at (0,-10) {};
\end{tikzpicture}
\begin{tikzpicture}[x=9pt,y=9pt,thick]\pgfsetlinewidth{0.5pt}
\node(1) at (-4,5){$\scriptstyle{{}_{\phi}J_{\phi}}$};
\node(4) at (-2.7,-6.5) {$\scriptstyle{{}_{\phi}J_{\phi}}$};
\node[circle,draw, inner sep=0.7pt](3) at (-1.5,1) {$\scriptstyle{u_{\phi}}$};

\draw[-] (3) to [out=-90, in=90] (-1.5,-4);
\draw[-] (-4,-4) to [out=-90, in=-90] (-1.5,-4);
\draw[-] (-2.7,-4.7) to [out=-90, in=90] (4);
\draw[-] (1) to [out=-90, in=90] (-4,-4);
\end{tikzpicture}
\begin{tikzpicture}[x=9pt,y=9pt,thick]\pgfsetlinewidth{0.5pt}
\node[inner sep=1pt] at (0,-3.5) {$\,=\,$};
\node at (0,-10) {};
\end{tikzpicture}
\begin{tikzpicture}[x=9pt,y=9pt,thick]\pgfsetlinewidth{0.5pt}
\node(1) at (-4,5){$\scriptstyle{{}_{\phi}J_{\phi}}$};
\node(4) at (-4,-6.5) {$\scriptstyle{{}_{\phi}J_{\phi}}$};

\draw[-] (1) to [out=-90, in=90] (4);
\end{tikzpicture}
\end{center}
where the last equality is given by Proposition \ref{lema:-1}(iii), while equality $(\star)$ is proved as follows:
\begin{gather*}
\fk{eq}_{\phi,\phi} \circ m_{\phi,\phi}^1\circ(\iota_X\otimes\iota_Y)
\overset{\eqref{Eq:mu3}}{=}m\circ(\fk{eq}_{\phi,1}\otimes\fk{eq}_{1,\phi})\circ(\iota_X\otimes\iota_Y) =m\circ(i_X\otimes i_Y) \overset{(1)}{=}  u\circ m_X \\ =\fk{eq}_{\phi,\phi} \circ u_\phi\circ\mathrm{coev}^{-1}.
\end{gather*}
This proves that $\bara{m}^1_{\phi,\phi}\circ\nu$ is the identity. A similar computation,  which  uses composition with  the epimorphism $\chi^1_{\phi,\phi}$, shows that $\nu \circ \bara{m}^1_{\phi,\phi}$ is the identity too so that  $\bara{m}^1_{\phi,\phi}$ is invertible.
Replacing $\phi$ by $\phi^{-1}$ will shows that $m_{_{\phi^{-1}}J_1}$ is an isomorphism which means that  $m_{ {}_{1}J_{\phi}}$ is an isomorphism, since $_{\phi^{-1}}J_1\,=\, {}_{1}J_{\phi}$.
\end{proof}

\begin{lemma}\label{lema:1}
Let $(A, m, u)$ be a monoid in $\cat{M}$. Assume that the unit $u_1:\I\rightarrow {_1J_1}$ is an isomorphism and consider a two-sided invertible object $X \in \Inv{\I}{A}$ with inverse $Y$ and associated automorphism $\Phi_X$ as in \eqref{Eq:phiXt}. Then, we have
$X\,=\, {}_{\Phi_X}J_1$ and $Y={}_1J_{\Phi_X}$.
\end{lemma}
\begin{proof}
Under the assumption made on $u_1$, the unit $u$ is a monomorphism by Proposition \ref{lema:-1}(iii). Hence, by Proposition \ref{prop:J}, we know that  $_{\Phi_X}J_1$ is invertible with inverse ${}_1J_{\Phi_X}$. Therefore, we have that ${}_{\Phi_X}J_1\tensor{}{}_1J_{\Phi_X} \cong \I$ and ${}_1J_{\Phi_X}\tensor{}{}_{\Phi_X}J_{1} \cong \I$. Now,  by Lemma \ref{lema:inclusion}, we know that $X$ is injected in ${}_{\Phi_X}J_1$ while $Y$ is injected in ${}_1J_{\Phi_X}$. Summing up,  we have shown that the elements $X$ and ${}_{\Phi_X}J_1$ satisfy the assumptions of Proposition \ref{prop:monoiso}(2) so that we obtain the equalities  $X={}_{\Phi_X}J_1$ and $Y={}_1J_{\Phi_X}$.
\end{proof}

Given a monoid $(A,m,u)$ in $\cat{M}$, we define the set $\cat{G}_{\scriptscriptstyle{\Z{\I}}}(\Z{A})$ as follows:
\begin{equation}\label{Eq:G}
\cat{G}_{\scriptscriptstyle{\Z{\I}}}(\Z{A})\,:=\, \Big\{\underset{}{} \theta \in \Aut{\scriptscriptstyle{\Z{\I}}\text{-alg}}{\Z{A}}| \,\, {}_{\theta}J_1 \text{ is two-sided invertible subobject }  \Big\}.
\end{equation}
Here by ${}_{\theta}J_1$ invertible we mean that the morphisms $\bara{m}_{\theta,\theta}^1$  and  $\bara{m}^{\theta}_{1,1}$ given in equation \eqref{Eq:mu5} are  isomorphisms,  which, in view of equation \eqref{Eq:egos}, shows that ${}_{\theta}J_1 \in {\rm Inv}_{\scriptscriptstyle{{}_1J_1}}(A)$.  In this case,  the two-sided inverse of ${}_{\theta}J_1$ is ${}_1J_{\theta}={}_{\theta^{-1}}J_1$. As we will see below the set  $\cat{G}_{\scriptscriptstyle{\Z{\I}}}(\Z{A})$ is in fact a subgroup of the automorphisms group $\Aut{\scriptscriptstyle{\Z{\I}}\text{-alg}}{\Z{A}}$.

The following is one of our main results.
\begin{theorem}\label{thm:mono}
Let $(\cat{M},\tensor{}, \I)$ be a Penrose  monoidal abelian, locally small and bicomplete category with  right exact tensor products functors.  Let $(A,m, u)$ be a monoid such that the morphism $u_{1}: \I \to {}_1J_1$
is an isomorphism. Then the map $\Phi$ of Proposition \ref{pro: Phi} induces an isomorphism of groups $\widehat{\Phi}:\Inv{\I}{A} \rightarrow \cat{G}_{\scriptscriptstyle{\Z{\I}}}(\Z{A})$ which makes the following triangle commutative
\begin{small}
$$
\xymatrix@C=40pt{\Inv{\I}{A}  \ar@{->}^-{\Phi}[r] \ar@{->}_-{\cong}^{\widehat{\Phi}}[rd] &  \Aut{\scriptscriptstyle{\Z{\I}}\text{-alg}}{\Z{A}}  \\  & \cat{G}_{\scriptscriptstyle{\Z{\I}}}(\Z{A}),  \ar@{^(->}_-{\B{\varsigma}}[u]  }
$$
\end{small}
where $\B{\varsigma}$ denotes the canonical injection. Moreover, $\widehat{\Phi}^{-1}(\theta)= {}_{\theta}J_1$, for every $\theta \in  \cat{G}_{\scriptscriptstyle{\Z{\I}}}(\Z{A})$.
\end{theorem}

\begin{proof}
In view of Proposition \ref{prop:J},  it is clear that $\Phi$ corestricts to a map $\widehat{\Phi}$ such that $\varsigma\circ\widehat{\Phi}=\Phi$ and, by Lemma \ref{lema:1}, we have that $\widehat{\Phi}$ is injective. Let us check that it is also surjective i.e.~that  $\Phi_{{}_{\theta}J_1} \,=\, \theta$, for every $\theta \in \cat{G}_{\scriptscriptstyle{\Z{\I}}}(\Z{A})$.  Recall that, for a given $t \in \Z{A}$, we have
$$ \Phi_{{}_{\theta}J_1}(t) \,=\, m \circ (A\tensor{}m) \circ (\fk{eq}_{\theta,1}\tensor{}A\tensor{}\fk{eq}_{1,\theta}) \circ ({}_{\theta}J_1\tensor{}t\tensor{} {}_1J_{\theta})  \circ (m_{\theta,\theta}^{1})^{-1}\circ u_\theta.
$$

The desired equality  can be checked by diagrams as follows
\begin{equation}\label{Eq:phitheta}
\begin{tikzpicture}[x=8pt,y=8pt,thick]\pgfsetlinewidth{0.5pt}
\node[inner sep=1pt] at (0,0) {$\,\Phi_{{}_{\theta}J_1}(t)=$};
\node at (0,-6) {};
\end{tikzpicture}
\begin{tikzpicture}[x=8pt,y=8pt,thick]\pgfsetlinewidth{0.5pt}
\node[circle,draw, inner sep=0.5pt](1) at (-2,3) {$\scriptstyle{\fk{eq}_{\theta,1}}$};
\node[circle,draw, inner sep=0.5pt](2) at (2,3) {$\scriptstyle{\fk{eq}_{1,\theta}}$};
\node[circle,draw, inner sep=1.5pt](3) at (0,1) {$\scriptstyle{t}$};
\node(4) at (-0.5,-5) {$\scriptstyle{A}$};

\draw[-] (-2,4.5) to [out=-90,in=90] (1);
\draw[-] (2,4.5) to [out=-90,in=90] (2);
\draw[-] (2,4.5) to [out=90,in=90] (-2,4.5);

\draw[-] (2) to [out=-90,in=90] (2,-1);
\draw[-] (3) to [out=-90,in=90] (0,-1);

\draw[-] (0,-1) to [out=-90,in=-90] (2,-1);
\draw[-] (1) to [out=-90,in=90] (-2,-2.5);
\draw[-] (1,-2.5) to [out=90,in=-90] (1,-1.5);
\draw[-] (1,-2.5) to [out=-90,in=-90] (-2,-2.5);
\draw[-] (4) to [out=90,in=-90] (-0.5,-3.4);
\end{tikzpicture}
\begin{tikzpicture}[x=8pt,y=8pt,thick]\pgfsetlinewidth{0.5pt}
\node[inner sep=1pt] at (0,0) {$\overset{\eqref{Eq:seraloqsera}}{=}$};
\node at (0,-6) {};
\end{tikzpicture}
\begin{tikzpicture}[x=8pt,y=8pt,thick]\pgfsetlinewidth{0.5pt}
\node[circle,draw, inner sep=0.5pt](1) at (-1.5,3) {$\scriptstyle{\fk{eq}_{\theta,1}}$};
\node[circle,draw, inner sep=0.5pt](2) at (1.5,3) {$\scriptstyle{\fk{eq}_{1,\theta}}$};
\node[circle,draw, inner sep=0.5pt](3) at (4,1) {$\scriptstyle{\theta(t)}$};
\node(4) at (0.7,-5) {$\scriptstyle{A}$};

\draw[-] (-1.5,4.5) to [out=-90,in=90] (1);
\draw[-] (1.5,4.5) to [out=-90,in=90] (2);
\draw[-] (1.5,4.5) to [out=90,in=90] (-1.5,4.5);

\draw[-] (2) to [out=-90,in=90] (1.5,-1);
\draw[-] (3) to [out=-90,in=90] (4,-1);

\draw[-] (4,-1) to [out=-90,in=-90] (1.5,-1);

\draw[-] (1) to [out=-90,in=90] (-1.5,-2.5);
\draw[-] (3,-2.5) to [out=90,in=-90] (3,-1.7);
\draw[-] (3,-2.5) to [out=-90,in=-90] (-1.5,-2.5);
\draw[-] (4) to [out=90,in=-90] (0.7,-3.8);
\end{tikzpicture}
\begin{tikzpicture}[x=8pt,y=8pt,thick]\pgfsetlinewidth{0.5pt}
\node[inner sep=1pt] at (0,0) {$\,=\,$};
\node at (0,-6) {};
\end{tikzpicture}
\begin{tikzpicture}[x=8pt,y=8pt,thick]\pgfsetlinewidth{0.5pt}
\node[circle,draw, inner sep=0.5pt](1) at (-1.5,3) {$\scriptstyle{\fk{eq}_{\theta,1}}$};
\node[circle,draw, inner sep=0.5pt](2) at (1.5,3) {$\scriptstyle{\fk{eq}_{1,\theta}}$};
\node[circle,draw, inner sep=0.5pt](3) at (4,1) {$\scriptstyle{\theta(t)}$};
\node(4) at (2,-5) {$\scriptstyle{A}$};

\draw[-] (-1.5,4.5) to [out=-90,in=90] (1);
\draw[-] (1.5,4.5) to [out=-90,in=90] (2);
\draw[-] (1.5,4.5) to [out=90,in=90] (-1.5,4.5);

\draw[-] (2) to [out=-90,in=90] (1.5,0);
\draw[-] (3) to [out=-90,in=90] (4,-2);

\draw[-] (1.5,0) to [out=-90,in=-90] (-1.5,0);

\draw[-] (1) to [out=-90,in=90] (-1.5,0);
\draw[-] (0,-1) to [out=-90,in=90] (0,-2);
\draw[-] (4,-2) to [out=-90,in=-90] (0,-2);
\draw[-] (4) to [out=90,in=-90] (2,-3.2);
\end{tikzpicture}
\begin{tikzpicture}[x=8pt,y=8pt,thick]\pgfsetlinewidth{0.5pt}
\node[inner sep=1pt] at (0,0) {$\,=\,$};
\node at (0,-6) {};
\end{tikzpicture}
\begin{tikzpicture}[x=8pt,y=8pt,thick]\pgfsetlinewidth{0.5pt}
\node[rectangle,draw, inner sep=2pt, text width=1.2cm, text centered](1) at (0,4) {$\scriptstyle{m_{\theta,\theta}^1}$};
\node[circle,draw, inner sep=0.5pt](2) at (0,1) {$\scriptstyle{\fk{eq}_{\theta,\theta}}$};
\node[circle,draw, inner sep=1pt](3) at (2,-1) {$\scriptstyle{\theta(t)}$};
\node(4) at (1,-5) {$\scriptstyle{A}$};

\draw[-] (-1.3,6) to [out=90,in=90] (1.3,6);
\draw[-] (-1.3,6) to [out=-90,in=90] (-1.3,4.8);
\draw[-] (1.3,6) to [out=-90,in=90] (1.3,4.8);
\draw[-] (2) to [out=-90,in=90] (0,-3);
\draw[-] (3) to [out=-90,in=90] (2,-3);

\draw[-] (2,-3) to [out=-90,in=-90] (0,-3);

\draw[-] (1) to [out=-90,in=90] (2);
\draw[-] (4) to [out=90,in=-90] (1,-3.5);
\end{tikzpicture}
\begin{tikzpicture}[x=8pt,y=8pt,thick]\pgfsetlinewidth{0.5pt}
\node[inner sep=1pt] at (0,0) {$\,=\;$};
\node at (0,-6) {};
\end{tikzpicture}
\begin{tikzpicture}[x=8pt,y=8pt,thick]\pgfsetlinewidth{0.5pt}
\node[circle,draw, inner sep=1pt](1) at (0,4) {$\scriptstyle{u_{\theta}}$};
\node[circle,draw, inner sep=0.5pt](2) at (0,1) {$\scriptstyle{\fk{eq}_{\theta,\theta}}$};
\node[circle,draw, inner sep=0.5pt](3) at (2,-1) {$\scriptstyle{\theta(t)}$};
\node(4) at (1,-5) {$\scriptstyle{A}$};

\draw[-] (2) to [out=-90,in=90] (0,-3);
\draw[-] (3) to [out=-90,in=90] (2,-3);

\draw[-] (2,-3) to [out=-90,in=-90] (0,-3);

\draw[-] (1) to [out=-90,in=90] (2);
\draw[-] (4) to [out=90,in=-90] (1,-3.5);
\end{tikzpicture}
\begin{tikzpicture}[x=8pt,y=8pt,thick]\pgfsetlinewidth{0.5pt}
\node[inner sep=1pt] at (0,0) {$=$};
\node at (0,-6) {};
\end{tikzpicture}
\begin{tikzpicture}[x=8pt,y=8pt,thick]\pgfsetlinewidth{0.5pt}
\node[circle,draw, inner sep=1.5pt](1) at (0,2) {$\scriptstyle{u}$};
\node[circle,draw, inner sep=0.5pt](3) at (3,2) {$\scriptstyle{\theta(t)}$};
\node(4) at (1.5,-5) {$\scriptstyle{A}$};

\draw[-] (3) to [out=-90,in=90] (3,-1);
\draw[-] (3,-1) to [out=-90,in=-90] (0,-1);
\draw[-] (1) to [out=-90,in=90] (0,-1);
\draw[-] (4) to [out=90,in=-90] (1.5,-2);
\end{tikzpicture}
\begin{tikzpicture}[x=8pt,y=8pt,thick]\pgfsetlinewidth{0.5pt}
\node[inner sep=1pt] at (0,0) {$=\theta(t).$};
\node at (0,-6) {};
\end{tikzpicture}
\end{equation}

We have so proved that the map $\widehat{\Phi}:\Inv{\I}{A} \rightarrow \cat{G}_{\scriptscriptstyle{\Z{\I}}}(\Z{A})$ is bijective. As a consequence, since $\Phi$ is a group morphism, we have for free that  $\cat{G}_{\scriptscriptstyle{\Z{\I}}}(\Z{A})$ is a subgroup of the automorphisms group $\Aut{\scriptscriptstyle{\Z{\I}}\text{-alg}}{\Z{A}}$ and that both $\widehat{\Phi}$ and $\varsigma$ are group morphisms.
\end{proof}

\subsection{The case when $\Phi$ is an isomorphism of groups}\label{ssec:iso}

Let $M$ and $N$ be two $A$-bimodules. Consider the action $\Z{A}\times \hom{\cat{M}}{M}{N}\rightarrow \hom{\cat{M}}{M}{N}:(z,f)\mapsto z\triangleright f$, defined by \eqref{Eq:triangles} for $S=\I$, and its right-hand version $f\triangleleft z$. These actions  turns in fact the $\Z{\I}$-module $\hom{\cat{M}}{M}{N}$ into a $\Z{A}$-bimodule.
Note that for $f\in \hom{\cat{M}}{A}{A}$ one easily checks that
\begin{equation}\label{Eq:linTriang}
f\in \lhom{A}{A}{A}\quad\text{ if and only if  }\quad f=A \triangleleft (fu).
\end{equation}
Here the notation $fu$ stands for the composition. Given $\sigma,\tau:\Z{A}\rightarrow \Z{A}$  two $\Z{I}$-linear maps, we set $$\cat{M}_{A,\Z{A}}(M_\sigma,N_\tau):=\LR{f\in \hom{A\text{-}}{M}{N} \mid f\circ (M\triangleleft \sigma(z))=f\triangleleft \tau(z),\forall z\in \Z{A}}.$$

For every $f\in\cat{M}_{A,\Z{A}}(A_\sigma,A_\tau),z\in\Z{A}$, we have
 \begin{gather}\label{form:lf}
   z\triangleright f\overset{\eqref{Eq:linTriang}}{=}z\triangleright(A \triangleleft (fu))=(z\triangleright A) \triangleleft (fu)=(A\triangleleft (fu))\circ(z\triangleright A)\overset{\eqref{Eq:linTriang}}{=}f \circ(z\triangleright A).
 \end{gather}
Let $\rho$ be another $\Z{\I}$-linear endomorphism of $\Z{A}$. We then compute
\begin{multline*}
 (\rho(t)\triangleright A)\circ f\circ \fk{eq}_{\rho,\sigma}=(\rho(t)\triangleright f ) \circ\fk{eq}_{\rho,\sigma}\overset{\eqref{form:lf}}{=}f\circ (\rho(t)\triangleright A)\circ\fk{eq}_{\rho,\sigma}
  \overset{\eqref{Eq:J}}{=}f\circ (A\triangleleft\sigma(t))\circ\fk{eq}_{\rho,\sigma}=(f\triangleleft\tau(t))\circ\fk{eq}_{\rho,\sigma} \\ =(A\triangleleft\tau(t))\circ f\circ\fk{eq}_{\rho,\sigma},
\end{multline*}
for any  $f\in\cat{M}_{A,\Z{A}}(A_\sigma,A_\tau)$ and $ t\in\Z{A}$. Using the universal property of $\fk{eq}_{\rho,\tau}$, we so get a unique morphism ${_\rho f}:{_\rho J_\sigma}\rightarrow {_\rho J_\tau}$ such that $\fk{eq}_{\rho,\tau}\circ {_\rho f}\,=\,f\circ \fk{eq}_{\rho,\sigma}$ which in diagrammatic form is expressed by
\begin{equation}\label{Eq:rhof}
\begin{tikzpicture}[x=8pt,y=8pt,thick]\pgfsetlinewidth{0.5pt}
\node(1) at (0,5) {$\scriptstyle{{}_{\rho}J_{\sigma}}$};
\node[circle,draw, inner sep=0.5pt](2) at (0,2) {$\scriptstyle{\fk{eq}_{\rho, \sigma}}$};
\node[circle,draw, inner sep=0.7pt](3) at (0,-2) {$\scriptstyle{f}$};
\node(4) at (0,-5) {$\scriptstyle{A}$};

\draw[-] (1) to [out=-90, in=90] (2);
\draw[-] (2) to [out=-90, in=90] (3);
\draw[-] (3) to [out=-90, in=90] (4);
\end{tikzpicture}
\begin{tikzpicture}[x=8pt,y=8pt,thick]\pgfsetlinewidth{0.5pt}
\node[inner sep=1pt] at (0,0) {$\,=\,$};
\node at (0,-5) {};
\end{tikzpicture}
\begin{tikzpicture}[x=8pt,y=8pt,thick]\pgfsetlinewidth{0.5pt}
\node(1) at (0,5) {$\scriptstyle{{}_{\rho}J_{\sigma}}$};
\node[circle,draw, inner sep=0.5pt](2) at (0,2) {$\scriptstyle{{}_{\rho}f}$};
\node[circle,draw, inner sep=0.5pt](3) at (0,-2) {$\scriptstyle{\fk{eq}_{\rho, \tau}}$};
\node(4) at (0,-5) {$\scriptstyle{A}$};

\draw[-] (1) to [out=-90, in=90] (2);
\draw[-] (2) to [out=-90, in=90] (3);
\draw[-] (3) to [out=-90, in=90] (4);
\end{tikzpicture}
\end{equation}
In this way we have defined a map
\begin{equation}\label{Eq:Rf}
\cat{M}_{A,\Z{A}}(A_\sigma,A_\tau)\longrightarrow \hom{\cat{M}}{_\rho J_\sigma}{_\rho J_\tau}:f\longmapsto {_\rho f}.
\end{equation}

\begin{definition}\label{def:divide}
Let $M$ and $N$ be two $A$-bimodules.
\begin{enumerate}[(1)]
\item For $\sigma,\tau:\Z{A}\rightarrow \Z{A}$, we will write $M_\sigma|N_\tau$ whenever there is $k\geq 1 $ and morphisms $f_{i}\in\cat{M}_{A,\Z{A}}(M_\sigma,N_\tau)$, $g_{i}\in\cat{M}_{A,\Z{A}}(N_\tau,M_\sigma)$, for $i=1,\cdots,k$,
such that $\sum_{i}g_{i}\circ f_{i}={M}$.
We will write $M_\sigma\sim N_\tau$ whenever both $M_\sigma|N_\tau$ and $N_\tau|M_\sigma$.

\item If there is $k\geq 1 $ and morphisms $f_{i}\in\hom{A,A}{M}{N}$, $g_{i}\in \hom{A,A}{N}{M}$, for $i=1,\cdots,k$,
such that $\sum_{i}g_{i}\circ f_{i}={M}$, then we will write $M|N$.  Since  our base category $\cat{M}$ is an abelian category, this is equivalent to say that $M$,  as an $A$-bimodule, is direct summand   of a finite direct sum of copies of $N$.  We will write $M\sim N$ whenever both $M|N$ and $N|M$.  The same notation was used for bimodules over rings by  Miyashita in \cite{Mi} which in fact  goes back to a Hirata \cite{Hirata:68}, although with different symbols.
\end{enumerate}
\end{definition}

The possible relation between the items $(1)$ and $(2)$ in the previous Definition will be explored in Proposition \ref{pro:div} below.
 Part of the following Theorem can be compared with \cite[Theorem 5.3]{PareigisIII}.
\begin{theorem}\label{teo:Miy1.2}
Let $(A,m,u)$ be a monoid in $\cat{M}$. Let $\sigma,\delta,\gamma:\Z{A}\rightarrow \Z{A}$ be $\Z{\I}$-linear maps and assume that $A_\gamma | A_\delta$.
Then the morphism $m_{\sigma,\gamma}^\delta$ is a split epimorphism. If we further assume that  $A$ is  right flat and
that $u_\delta:\I\rightarrow {_\delta J_\delta}$ is an isomorphism, then $m_{\sigma,\gamma}^\delta$ is an isomorphism.
\end{theorem}
\begin{proof}
Let  $g_i$ denote the resulting maps from the assumption  $A_\gamma|A_\delta$. Since each $g_i$ is left $A$-linear, we have
\begin{equation}\label{form:figi}
\sum_{i}f_i\triangleleft (g_i\circ u)=\sum_{i}g_i\circ (f_i\triangleleft u)=\sum_{i}g_i\circ f_i=A.
\end{equation}
We set $\xi_{\sigma,\gamma}^\delta:=(\sum_{i}{_\sigma f_{i}}\otimes {_\delta g_{i}})\circ(_\sigma J_\gamma\otimes u_\delta)\circ r^{-1}_{_\sigma J_\gamma}$,
where ${}_{\sigma}f_i$ and ${}_{\delta}g_j$ are defined as in equation \eqref{Eq:rhof}.
By its own definition the morphism $\xi^{\delta}_{\sigma,\gamma}$ satisfies
\begin{equation}\label{Eq:xidsg}
\begin{tikzpicture}[x=8pt,y=8pt,thick]\pgfsetlinewidth{0.5pt}
\node(1) at (0,4) {$\scriptstyle{{}_{\sigma}J_{\gamma}}$};
\node[rectangle,draw, text width=1.5cm, text centered](2) at (0,1) {$\scriptstyle{\xi_{\sigma,\gamma}^{\delta}}$};
\node(3) at (-2,-4) {$\scriptstyle{{}_{\sigma}J_{\delta}}$};
\node(4) at (2,-4) {$\scriptstyle{{}_{\delta}J_{\gamma}}$};

\draw[-] (1) to [out=-90, in=90] (2);
\draw[-] (2,0) to [out=-90, in=90] (4);
\draw[-] (-2,0) to [out=-90, in=90] (3);
\end{tikzpicture}
\begin{tikzpicture}[x=8pt,y=8pt,thick]\pgfsetlinewidth{0.5pt}
\node[inner sep=1pt] at (0,0) {$\,=\,$};
\node at (0,-4) {};
\end{tikzpicture}
\begin{tikzpicture}[x=8pt,y=8pt,thick]\pgfsetlinewidth{0.5pt}
\node(1) at (0,4) {$\scriptstyle{{}_{\sigma}J_{\gamma}}$};
\node(3) at (0,-4) {$\scriptstyle{{}_{\sigma}J_{\delta}}$};
\node(4) at (4,-4) {$\scriptstyle{{}_{\delta}J_{\gamma}}$};
\node[circle,draw, inner sep=0.5pt](6) at (4,-0.5) {$\scriptstyle{{}_{\delta}g_i}$};
\node[circle,draw, inner sep=1pt](5) at (4,2) {$\scriptstyle{u_{\delta}}$};
\node[circle,draw, inner sep=0.5pt](2) at (0,-0.5) {$\scriptstyle{{}_{\sigma}f_i}$};

\draw[-] (1) to [out=-90, in=90] (2);
\draw[-] (2) to [out=-90, in=90] (3);
\draw[-] (5) to [out=-90, in=90] (6);
\draw[-] (6) to [out=-90, in=90] (4);
\end{tikzpicture}
\end{equation}
where the summation is understood. So we compute
\begin{center}
\begin{tikzpicture}[x=8pt,y=8pt,thick]\pgfsetlinewidth{0.5pt}
\node(1) at (0,8) {$\scriptstyle{{}_{\sigma}J_{\gamma}}$};
\node[rectangle,draw, text width=1.5cm, text centered](2) at (0,5) {$\scriptstyle{\xi_{\sigma,\gamma}^{\delta}}$};
\node[rectangle,draw, text width=1.5cm, text centered](3) at (0,1) {$\scriptstyle{m_{\sigma,\gamma}^{\delta}}$};
\node[circle,draw, inner sep=0.5pt](4) at (0,-3) {$\scriptstyle{\fk{eq}_{\sigma,\gamma}}$};
\node(5) at (0,-8) {$\scriptstyle{A}$};

\draw[-] (1) to [out=-90, in=90] (2);
\draw[-] (-1.5,4) to [out=-90, in=90] (-1.5,2);
\draw[-] (1.5,4) to [out=-90, in=90] (1.5,2);
\draw[-] (3) to [out=-90, in=90] (4);
\draw[-] (4) to [out=-90, in=90] (5);
\end{tikzpicture}
\begin{tikzpicture}[x=8pt,y=8pt,thick]\pgfsetlinewidth{0.5pt}
\node[inner sep=1pt] at (0,1) {$\,\overset{\eqref{Eq:mu3}, \,\eqref{Eq:xidsg}}{=}\,$};
\node at (0,-8) {};
\end{tikzpicture}
\begin{tikzpicture}[x=8pt,y=8pt,thick]\pgfsetlinewidth{0.5pt}
\node(1) at (0,8) {$\scriptstyle{{}_{\sigma}J_{\gamma}}$};
\node[circle,draw, inner sep=0.5pt](2) at (0,2) {$\scriptstyle{{}_{\sigma}f_i}$};
\node[circle,draw, inner sep=0.5pt](3) at (0,-3) {$\scriptstyle{\fk{eq}_{\sigma,\delta}}$};
\node[circle,draw, inner sep=0.5pt](4) at (4,-3) {$\scriptstyle{\fk{eq}_{\delta,\gamma}}$};
\node(5) at (2,-8) {$\scriptstyle{A}$};
\node[circle,draw, inner sep=0.5pt](6) at (4,1) {$\scriptstyle{{}_{\delta}g_i}$};
\node[circle,draw, inner sep=0.5pt](7) at (4,5) {$\scriptstyle{u_{\delta}}$};

\draw[-] (1) to [out=-90, in=90] (2);
\draw[-] (3) to [out=-90, in=90] (0,-4.5);
\draw[-] (4) to [out=-90, in=90] (4,-4.5);
\draw[-] (0,-4.5) to [out=-90, in=-90] (4,-4.5);
\draw[-] (2) to [out=-90, in=90] (3);
\draw[-] (7) to [out=-90, in=90] (6);
\draw[-] (6) to [out=-90, in=90] (4);
\draw[-] (5) to [out=90, in=-90] (2,-5.6);
\end{tikzpicture}
\begin{tikzpicture}[x=8pt,y=8pt,thick]\pgfsetlinewidth{0.5pt}
\node[inner sep=1pt] at (0,1) {$\,\overset{\eqref{Eq:rhof}}{=}\,$};
\node at (0,-8) {};
\end{tikzpicture}
\begin{tikzpicture}[x=8pt,y=8pt,thick]\pgfsetlinewidth{0.5pt}
\node(1) at (0,8) {$\scriptstyle{{}_{\sigma}J_{\gamma}}$};
\node[circle,draw, inner sep=0.5pt](2) at (0,2) {$\scriptstyle{\fk{eq}_{\sigma,\gamma}}$};
\node[circle,draw, inner sep=0.5pt](3) at (0,-3) {$\scriptstyle{f_i}$};
\node[circle,draw, inner sep=1pt](4) at (4,-3) {$\scriptstyle{g_i}$};
\node(5) at (2,-8) {$\scriptstyle{A}$};
\node[circle,draw, inner sep=0.5pt](6) at (4,1) {$\scriptstyle{\fk{eq}_{\delta,\delta}}$};
\node[circle,draw, inner sep=0.5pt](7) at (4,5) {$\scriptstyle{u_{\delta}}$};

\draw[-] (1) to [out=-90, in=90] (2);
\draw[-] (3) to [out=-90, in=90] (0,-4.5);
\draw[-] (4) to [out=-90, in=90] (4,-4.5);
\draw[-] (0,-4.5) to [out=-90, in=-90] (4,-4.5);
\draw[-] (2) to [out=-90, in=90] (3);
\draw[-] (7) to [out=-90, in=90] (6);
\draw[-] (6) to [out=-90, in=90] (4);
\draw[-] (5) to [out=90, in=-90] (2,-5.6);
\end{tikzpicture}
\begin{tikzpicture}[x=8pt,y=8pt,thick]\pgfsetlinewidth{0.5pt}
\node[inner sep=1pt] at (0,1) {$\,=\,$};
\node at (0,-8) {};
\end{tikzpicture}
\begin{tikzpicture}[x=8pt,y=8pt,thick]\pgfsetlinewidth{0.5pt}
\node(1) at (0,8) {$\scriptstyle{{}_{\sigma}J_{\gamma}}$};
\node[circle,draw, inner sep=0.5pt](2) at (0,2) {$\scriptstyle{\fk{eq}_{\sigma,\gamma}}$};
\node[circle,draw, inner sep=0.5pt](3) at (0,-3) {$\scriptstyle{f_i}$};
\node[circle,draw, inner sep=1pt](4) at (4,-3) {$\scriptstyle{g_i}$};
\node(5) at (2,-8) {$\scriptstyle{A}$};
\node[circle,draw, inner sep=1.5pt](7) at (4,1) {$\scriptstyle{u}$};

\draw[-] (1) to [out=-90, in=90] (2);
\draw[-] (3) to [out=-90, in=90] (0,-4.5);
\draw[-] (4) to [out=-90, in=90] (4,-4.5);
\draw[-] (0,-4.5) to [out=-90, in=-90] (4,-4.5);
\draw[-] (2) to [out=-90, in=90] (3);
\draw[-] (7) to [out=-90, in=90] (4);
\draw[-] (5) to [out=90, in=-90] (2,-5.6);
\end{tikzpicture}
\begin{tikzpicture}[x=8pt,y=8pt,thick]\pgfsetlinewidth{0.5pt}
\node[inner sep=1pt] at (0,1) {$\,\overset{\eqref{form:figi}}{=}\,$};
\node at (0,-8) {};
\end{tikzpicture}
\begin{tikzpicture}[x=8pt,y=8pt,thick]\pgfsetlinewidth{0.5pt}
\node(1) at (0,8) {$\scriptstyle{{}_{\sigma}J_{\gamma}}$};
\node[circle,draw, inner sep=0.5pt](2) at (0,2) {$\scriptstyle{\fk{eq}_{\sigma,\gamma}}$};
\node(5) at (0,-8) {$\scriptstyle{A}$};

\draw[-] (1) to [out=-90, in=90] (2);
\draw[-] (2) to [out=-90, in=90] (5);
\end{tikzpicture}
\end{center}

Since $\fk{eq}_{\sigma,\gamma}$ is a monomorphism, we conclude that
$m_{\sigma,\gamma}^\delta\circ\xi_{\sigma,\gamma}^\delta={_\sigma J_\gamma}$ and hence $m_{\sigma,\gamma}^\delta$ is a split epimorphism.
Now let us check, under the assumptions $u_\delta:\I\rightarrow {_\delta J_\delta}$ is an isomorphism and $A$ is right flat, that $m_{\sigma,\gamma}^{\delta}$ is indeed an isomorphism. Using diagrams, with summation understood again,  we have

\begin{center}
\begin{tikzpicture}[x=8pt,y=8pt,thick]\pgfsetlinewidth{0.5pt}
\node(1) at (-2,3.5) {$\scriptstyle{{}_{\sigma}J_{\delta}}$};
\node(2) at (2,3.5) {$\scriptstyle{{}_{\delta}J_{\gamma}}$};
\node[rectangle,draw, text width=1.5cm, text centered](3) at (0,0) {$\scriptstyle{m_{\sigma,\gamma}^{\delta}}$};
\node[rectangle,draw, text width=1.5cm, text centered](4) at (0,-3.5) {$\scriptstyle{\xi_{\sigma,\gamma}^{\delta}}$};
\node[circle,draw, inner sep=0.5pt](5) at (-2,-7.5) {$\scriptstyle{\fk{eq}_{\sigma, \delta}}$};
\node[circle,draw, inner sep=0.5pt](6) at (2,-7.5) {$\scriptstyle{\fk{eq}_{\delta, \gamma}}$};
\node(7) at (-2,-11) {$\scriptstyle{A}$};
\node(8) at (2,-11) {$\scriptstyle{A}$};

\draw[-] (1) to [out=-90, in=90] (-2,1);
\draw[-] (2) to [out=-90, in=90] (2,1);
\draw[-] (3) to [out=-90, in=90] (4);
\draw[-] (-2,-4.5) to [out=-90, in=90] (5);
\draw[-] (2,-4.5) to [out=-90, in=90] (6);
\draw[-] (5) to [out=-90, in=90] (7);
\draw[-] (6) to [out=-90, in=90] (8);
\end{tikzpicture}
\begin{tikzpicture}[x=8pt,y=8pt,thick]\pgfsetlinewidth{0.5pt}
\node[inner sep=1pt] at (0,-3) {$\,=\,$};
\node at (0,-10) {};
\end{tikzpicture}
\begin{tikzpicture}[x=8pt,y=8pt,thick]\pgfsetlinewidth{0.5pt}
\node(1) at (-2,3.5) {$\scriptstyle{{}_{\sigma}J_{\delta}}$};
\node(2) at (2,3.5) {$\scriptstyle{{}_{\delta}J_{\gamma}}$};
\node[rectangle,draw, text width=1.5cm, text centered](3) at (0,0) {$\scriptstyle{m_{\sigma,\gamma}^{\delta}}$};
\node[circle,draw, inner sep=0.5pt](5) at (0,-3.5) {$\scriptstyle{\fk{eq}_{\sigma, \gamma}}$};
\node[circle,draw, inner sep=0.5pt](6) at (4,-7.5) {$\scriptstyle{g_i}$};
\node[circle,draw, inner sep=1.5pt](4) at (4,-5) {$\scriptstyle{u}$};
\node[circle,draw, inner sep=0.5pt](9) at (0,-7.5) {$\scriptstyle{f_i}$};
\node(7) at (0,-11) {$\scriptstyle{A}$};
\node(8) at (4,-11) {$\scriptstyle{A}$};

\draw[-] (1) to [out=-90, in=90] (-2,1);
\draw[-] (2) to [out=-90, in=90] (2,1);
\draw[-] (3) to [out=-90, in=90] (5);
\draw[-] (4) to [out=-90, in=90] (6);
\draw[-] (5) to [out=-90, in=90] (9);
\draw[-] (6) to [out=-90, in=90] (8);
\draw[-] (9) to [out=-90, in=90] (7);
\end{tikzpicture}
\begin{tikzpicture}[x=8pt,y=8pt,thick]\pgfsetlinewidth{0.5pt}
\node[inner sep=1pt] at (0,-3) {$\,=\,$};
\node at (0,-10) {};
\end{tikzpicture}
\begin{tikzpicture}[x=8pt,y=8pt,thick]\pgfsetlinewidth{0.5pt}
\node(1) at (-2,3.5) {$\scriptstyle{{}_{\sigma}J_{\delta}}$};
\node(2) at (2,3.5) {$\scriptstyle{{}_{\delta}J_{\gamma}}$};
\node[circle,draw, inner sep=0.5pt](3) at (4,-7.5) {$\scriptstyle{g_i}$};
\node[circle,draw, inner sep=1.5pt](4) at (4,-5) {$\scriptstyle{u}$};
\node[circle,draw, inner sep=0.5pt](9) at (0,-7.5) {$\scriptstyle{f_i}$};
\node[circle,draw, inner sep=0.5pt](5) at (-2,0) {$\scriptstyle{\fk{eq}_{\sigma, \delta}}$};
\node[circle,draw, inner sep=0.5pt](6) at (2,0) {$\scriptstyle{\fk{eq}_{\delta, \gamma}}$};
\node(7) at (0,-11) {$\scriptstyle{A}$};
\node(8) at (4,-11) {$\scriptstyle{A}$};

\draw[-] (1) to [out=-90, in=90] (5);
\draw[-] (2) to [out=-90, in=90] (6);
\draw[-] (-2,-2) to [out=90, in=-90] (5);
\draw[-] (2,-2) to [out=90, in=-90] (6);
\draw[-] (2,-2) to [out=-90, in=-90] (-2,-2);
\draw[-] (9) to [out=90, in=-90] (0,-3.2);
\draw[-] (3) to [out=-90, in=90] (8);
\draw[-] (4) to [out=-90, in=90] (3);
\draw[-] (9) to [out=-90, in=90] (7);
\end{tikzpicture}
\begin{tikzpicture}[x=8pt,y=8pt,thick]\pgfsetlinewidth{0.5pt}
\node[inner sep=1pt] at (0,-3) {$\,=\,$};
\node at (0,-10) {};
\end{tikzpicture}
\begin{tikzpicture}[x=8pt,y=8pt,thick]\pgfsetlinewidth{0.5pt}
\node(1) at (-2,3.5) {$\scriptstyle{{}_{\sigma}J_{\delta}}$};
\node(2) at (2,3.5) {$\scriptstyle{{}_{\delta}J_{\gamma}}$};
\node[circle,draw, inner sep=0.5pt](3) at (5,-7.5) {$\scriptstyle{g_i}$};
\node[circle,draw, inner sep=1.5pt](4) at (5,-5) {$\scriptstyle{u}$};
\node[circle,draw, inner sep=0.5pt](9) at (2,-4) {$\scriptstyle{f_i}$};
\node[circle,draw, inner sep=0.5pt](5) at (-2,0) {$\scriptstyle{\fk{eq}_{\sigma, \delta}}$};
\node[circle,draw, inner sep=0.5pt](6) at (2,0) {$\scriptstyle{\fk{eq}_{\delta, \gamma}}$};
\node(7) at (0,-11) {$\scriptstyle{A}$};
\node(8) at (5,-11) {$\scriptstyle{A}$};

\draw[-] (1) to [out=-90, in=90] (5);
\draw[-] (2) to [out=-90, in=90] (6);
\draw[-] (5) to [out=-90, in=90] (-2,-6);
\draw[-] (7) to [out=90, in=-90] (0,-7.2);
\draw[-] (9) to [out=-90, in=90] (2,-6);
\draw[-] (-2,-6) to [out=-90, in=-90] (2,-6);
\draw[-] (3) to [out=-90, in=90] (8);
\draw[-] (4) to [out=-90, in=90] (3);
\draw[-] (9) to [out=90, in=-90] (6);
\end{tikzpicture}
\begin{tikzpicture}[x=8pt,y=8pt,thick]\pgfsetlinewidth{0.5pt}
\node[inner sep=1pt] at (0,-3) {$\,=\,$};
\node at (0,-10) {};
\end{tikzpicture}
\begin{tikzpicture}[x=8pt,y=8pt,thick]\pgfsetlinewidth{0.5pt}
\node(1) at (-2,3.5) {$\scriptstyle{{}_{\sigma}J_{\delta}}$};
\node(2) at (2,3.5) {$\scriptstyle{{}_{\delta}J_{\gamma}}$};
\node[circle,draw, inner sep=0.5pt](3) at (5,-7.5) {$\scriptstyle{g_i}$};
\node[circle,draw, inner sep=1.5pt](4) at (5,-5) {$\scriptstyle{u}$};
\node[circle,draw, inner sep=0.5pt](9) at (2,-4) {$\scriptstyle{\fk{eq}_{\delta, \delta}}$};
\node[circle,draw, inner sep=0.5pt](5) at (-2,0) {$\scriptstyle{\fk{eq}_{\sigma, \delta}}$};
\node[circle,draw, inner sep=0.5pt](6) at (2,0) {$\scriptstyle{{}_{\delta }f_i}$};
\node(7) at (0,-11) {$\scriptstyle{A}$};
\node(8) at (5,-11) {$\scriptstyle{A}$};

\draw[-] (1) to [out=-90, in=90] (5);
\draw[-] (2) to [out=-90, in=90] (6);
\draw[-] (5) to [out=-90, in=90] (-2,-6);
\draw[-] (7) to [out=90, in=-90] (0,-7.2);
\draw[-] (9) to [out=-90, in=90] (2,-6);
\draw[-] (-2,-6) to [out=-90, in=-90] (2,-6);
\draw[-] (3) to [out=-90, in=90] (8);
\draw[-] (4) to [out=-90, in=90] (3);
\draw[-] (9) to [out=90, in=-90] (6);
\end{tikzpicture}
\end{center}
where in the first equality we have used simultaneously equation \eqref{Eq:xidsg}, \eqref{Eq:rhof} and the fact that the unit $u_{\delta}$ of ${}_{\delta}J_{\delta}$ satisfies $\fk{eq}_{\delta,\delta} \circ u_{\delta}\,=\, u$ the unit of $A$. We continue then our computation
\begin{center}
\begin{tikzpicture}[x=8pt,y=8pt,thick]\pgfsetlinewidth{0.5pt}
\node[inner sep=1pt] at (0,-3) {$\,=\,$};
\node at (0,-10) {};
\end{tikzpicture}
\begin{tikzpicture}[x=8pt,y=8pt,thick]\pgfsetlinewidth{0.5pt}
\node(1) at (-2,3.5) {$\scriptstyle{{}_{\sigma}J_{\delta}}$};
\node(2) at (2,3.5) {$\scriptstyle{{}_{\delta}J_{\gamma}}$};
\node[circle,draw, inner sep=0.5pt](3) at (5,-7.5) {$\scriptstyle{g_i}$};
\node[circle,draw, inner sep=1.5pt](4) at (5,-5) {$\scriptstyle{u}$};
\node[circle,draw, inner sep=0.5pt](9) at (2,-3.5) {$\scriptstyle{u_{\delta}^{-1}}$};
\node[circle,draw, inner sep=0.5pt](5) at (-2,0) {$\scriptstyle{\fk{eq}_{\sigma, \delta}}$};
\node[circle,draw, inner sep=0.5pt](6) at (2,0) {$\scriptstyle{{}_{\delta }f_i}$};
\node[circle,draw, inner sep=1.5pt](10) at (0,-5) {$\scriptstyle{u}$};
\node(7) at (-1,-11) {$\scriptstyle{A}$};
\node(8) at (5,-11) {$\scriptstyle{A}$};

\draw[-] (1) to [out=-90, in=90] (5);
\draw[-] (2) to [out=-90, in=90] (6);
\draw[-] (5) to [out=-90, in=90] (-2,-6);
\draw[-] (7) to [out=90, in=-90] (-1,-6.5);
\draw[-] (10) to [out=-90, in=90] (0,-6);
\draw[-] (-2,-6) to [out=-90, in=-90] (0,-6);
\draw[-] (3) to [out=-90, in=90] (8);
\draw[-] (4) to [out=-90, in=90] (3);
\draw[-] (9) to [out=90, in=-90] (6);
\end{tikzpicture}
\begin{tikzpicture}[x=8pt,y=8pt,thick]\pgfsetlinewidth{0.5pt}
\node[inner sep=1pt] at (0,-3) {$\,=\,$};
\node at (0,-10) {};
\end{tikzpicture}
\begin{tikzpicture}[x=8pt,y=8pt,thick]\pgfsetlinewidth{0.5pt}
\node(1) at (-2,3.5) {$\scriptstyle{{}_{\sigma}J_{\delta}}$};
\node(2) at (2,3.5) {$\scriptstyle{{}_{\delta}J_{\gamma}}$};
\node[circle,draw, inner sep=0.5pt](3) at (2,-7.5) {$\scriptstyle{g_i}$};
\node[circle,draw, inner sep=1.5pt](4) at (2,-5.5) {$\scriptstyle{u}$};
\node[circle,draw, inner sep=0.5pt](9) at (2,-3) {$\scriptstyle{u_{\delta}^{-1}}$};
\node[circle,draw, inner sep=0.5pt](5) at (-2,0) {$\scriptstyle{\fk{eq}_{\sigma, \delta}}$};
\node[circle,draw, inner sep=0.5pt](6) at (2,0.5) {$\scriptstyle{{}_{\delta }f_i}$};
\node(7) at (-2,-11) {$\scriptstyle{A}$};
\node(8) at (2,-11) {$\scriptstyle{A}$};

\draw[-] (1) to [out=-90, in=90] (5);
\draw[-] (2) to [out=-90, in=90] (6);
\draw[-] (5) to [out=-90, in=90] (7);
\draw[-] (3) to [out=-90, in=90] (8);
\draw[-] (4) to [out=-90, in=90] (3);
\draw[-] (9) to [out=90, in=-90] (6);
\end{tikzpicture}
\begin{tikzpicture}[x=8pt,y=8pt,thick]\pgfsetlinewidth{0.5pt}
\node[inner sep=1pt] at (0,-3) {$\,=\,$};
\node at (0,-10) {};
\end{tikzpicture}
\begin{tikzpicture}[x=8pt,y=8pt,thick]\pgfsetlinewidth{0.5pt}
\node(1) at (-2,3.5) {$\scriptstyle{{}_{\sigma}J_{\delta}}$};
\node(2) at (2,3.5) {$\scriptstyle{{}_{\delta}J_{\gamma}}$};
\node[circle,draw, inner sep=0.5pt](3) at (2,-7.5) {$\scriptstyle{g_i}$};
\node[circle,draw, inner sep=0.5pt](9) at (2,-3) {$\scriptstyle{\fk{eq}_{\delta, \delta}}$};
\node[circle,draw, inner sep=0.5pt](5) at (-2,0) {$\scriptstyle{\fk{eq}_{\sigma, \delta}}$};
\node[circle,draw, inner sep=0.5pt](6) at (2,0.5) {$\scriptstyle{{}_{\delta }f_i}$};
\node(7) at (-2,-11) {$\scriptstyle{A}$};
\node(8) at (2,-11) {$\scriptstyle{A}$};

\draw[-] (1) to [out=-90, in=90] (5);
\draw[-] (2) to [out=-90, in=90] (6);
\draw[-] (5) to [out=-90, in=90] (7);
\draw[-] (3) to [out=-90, in=90] (8);
\draw[-] (3) to [out=90, in=-90] (9);
\draw[-] (9) to [out=90, in=-90] (6);
\end{tikzpicture}
\begin{tikzpicture}[x=8pt,y=8pt,thick]\pgfsetlinewidth{0.5pt}
\node[inner sep=1pt] at (0,-3) {$\,=\,$};
\node at (0,-10) {};
\end{tikzpicture}
\begin{tikzpicture}[x=8pt,y=8pt,thick]\pgfsetlinewidth{0.5pt}
\node(1) at (-2,3.5) {$\scriptstyle{{}_{\sigma}J_{\delta}}$};
\node(2) at (2,3.5) {$\scriptstyle{{}_{\delta}J_{\gamma}}$};
\node[circle,draw, inner sep=0.5pt](3) at (2,-7.5) {$\scriptstyle{g_i}$};
\node[circle,draw, inner sep=0.5pt](9) at (2,-4.5) {$\scriptstyle{f_i}$};
\node[circle,draw, inner sep=0.5pt](5) at (-2,0) {$\scriptstyle{\fk{eq}_{\sigma, \delta}}$};
\node[circle,draw, inner sep=0.5pt](6) at (2,0) {$\scriptstyle{\fk{eq}_{\delta, \gamma}}$};
\node(7) at (-2,-11) {$\scriptstyle{A}$};
\node(8) at (2,-11) {$\scriptstyle{A}$};

\draw[-] (1) to [out=-90, in=90] (5);
\draw[-] (2) to [out=-90, in=90] (6);
\draw[-] (5) to [out=-90, in=90] (7);
\draw[-] (3) to [out=-90, in=90] (8);
\draw[-] (3) to [out=90, in=-90] (9);
\draw[-] (9) to [out=90, in=-90] (6);
\end{tikzpicture}
\begin{tikzpicture}[x=8pt,y=8pt,thick]\pgfsetlinewidth{0.5pt}
\node[inner sep=1pt] at (0,-3) {$\,=\,$};
\node at (0,-10) {};
\end{tikzpicture}
\begin{tikzpicture}[x=8pt,y=8pt,thick]\pgfsetlinewidth{0.5pt}
\node(1) at (-2,3.5) {$\scriptstyle{{}_{\sigma}J_{\delta}}$};
\node(2) at (2,3.5) {$\scriptstyle{{}_{\delta}J_{\gamma}}$};
\node[circle,draw, inner sep=0.5pt](5) at (-2,0) {$\scriptstyle{\fk{eq}_{\sigma, \delta}}$};
\node[circle,draw, inner sep=0.5pt](6) at (2,0) {$\scriptstyle{\fk{eq}_{\delta, \gamma}}$};
\node(7) at (-2,-11) {$\scriptstyle{A}$};
\node(8) at (2,-11) {$\scriptstyle{A}$};

\draw[-] (1) to [out=-90, in=90] (5);
\draw[-] (2) to [out=-90, in=90] (6);
\draw[-] (5) to [out=-90, in=90] (7);
\draw[-] (8) to [out=90, in=-90] (6);
\end{tikzpicture}
\end{center}
If $\fk{eq}_{\sigma,\delta}\tensor{}\fk{eq}_{\delta,\gamma}$ is a monomorphism, then we  get  $\xi_{\sigma,\gamma}^\delta\circ m_{\sigma,\gamma}^\delta={}_{\sigma}J_{\delta}\tensor{}{}_{\delta}J_{\gamma}$, and so $m_{\sigma,\gamma}^\delta$ is an isomorphism.

The fact that $\fk{eq}_{\sigma,\delta}\tensor{}\fk{eq}_{\delta,\gamma}$ is a monomorphism is proved as follows. First, using Proposition \ref{lema:-1}  and  that $A_{\gamma}| A_{\delta} $, one shows that  $({}_\gamma J_\delta, {}_\delta J_\gamma, \bara{m}_{\delta,\, \delta}^{\gamma}, \bara{\xi}_{\gamma,\gamma}^{\delta})$ is  a right dualizable datum between the monoids ${}_{\gamma}J_{\gamma}$ and ${}_{\delta}J_{\delta}$, where $\bara{\xi}_{\gamma,\gamma}^{\delta}= \chi_{\gamma,\gamma}^{\delta} \circ \xi_{\gamma,\gamma}^{\delta}$ is a  splitting of $\bara{m}_{\gamma,\gamma}^{\delta}$. Secondly, by Proposition \ref{prop:adjunction}, we obtain that  ${}_{\delta}J_{\gamma}$ is left flat as $\I \cong {}_{\delta}J_{\delta}$.
In this way, we deduce that $ \fk{eq}_{\sigma,\delta}\tensor{}\fk{eq}_{\delta,\gamma} = (A\tensor{}\fk{eq}_{\delta,\gamma})  \circ (\fk{eq}_{\sigma,\delta}\tensor{} {}_{\delta}J_{\gamma}) $ is a monomorphism since $A$ is right flat.
\end{proof}

As a consequence of  Theorem \ref{teo:Miy1.2}, we have
\begin{corollary}\label{coro:tau 1}
Let $(A,m,u)$ be a monoid such  that $A$ is right flat
and $u_1:\I\rightarrow {_1 J_1}$  is an isomorphism.
Take $\theta \in \mathrm{Aut}_{\scriptscriptstyle{\Z{\I}}\text{-}alg}(\Z{A})$ such that
$A_\theta \sim A_1$. Then $\theta \in \cat{G}_{\scriptscriptstyle{\Z{\I}}}(\Z{A})$ the subgroup  defined in \eqref{Eq:G}.
\end{corollary}

Under the assumptions of Corollary \ref{coro:tau 1}, we do not know if the invertibility of ${}_{\theta}J_1$, for some $\theta \in \mathrm{Aut}_{\scriptscriptstyle{\Z{\I}}\text{-}alg}(\Z{A})$, is a sufficient condition to have  that $A_{\theta} \sim A_1$. This is why at this level of generality it is convenient   to distinguish between the subset of elements $\theta$ for which ${}_{\theta}J_1$ is invertible and those for which $A_{\theta} \sim A_1$. However, as we will see below, for Azumaya  monoids in $\cat{M}$ where $\cat{Z}$ is faithful, these two subsets coincide as subgroups of $\mathrm{Aut}_{\scriptscriptstyle{\Z{\I}}\text{-}alg}(\Z{A})$. Thus, under these stronger assumptions one obtains the converse of the previous corollary,  as in the classical case (see e.g. \cite{Mi}).

The following result collects useful  properties of Definition \ref{def:divide}.

\begin{proposition}\label{pro:div}
 Let $L,M,N$ be $A$-bimodules and let $\rho,\sigma,\tau:\Z{A}\rightarrow \Z{A}$ be  $\Z{I}$-linear maps. Then we have the following properties:
 \begin{enumerate}
  \item $\hom{A,A}{M}{N}\subseteq\cat{M}_{A,\Z{A}}(M_\sigma,N_\sigma)$.
  \item $M|N$ implies $M_\sigma|N_\sigma$.
  \item $M\cong N$ implies $M_\sigma\sim N_\sigma$.
  \item $M|N$ implies $M\otimes_A L|N\otimes_A L$.
  \item {$\cat{M}_{A,\Z{A}}(M_\sigma,N_\tau)\circ \cat{M}_{A,\Z{A}}(L_\rho,M_\sigma)\subseteq\cat{M}_{A,\Z{A}}(L_\rho,N_\tau)$}, where $\circ$ is the composition in ${}_A\cat{M}$.
  \item $L_\rho|M_\sigma$ and $M_\sigma|N_\tau$ implies $L_\rho|N_\tau$.
 \item If $M_\sigma|N_\tau$, then $M_{\sigma \rho}|N_{\tau \rho}$.
 \end{enumerate}
\end{proposition}
\begin{proof}
 $(1)$ Let  $f\in \hom{A,A}{M}{N}$. Then $f\in\lhom{A}{M}{N}$ and, for every $z\in \Z{A}$, we have
$$f\circ(M\triangleleft\sigma(z))=f\circ\varrho_N\circ(M\otimes\sigma(z))=\varrho_M\circ(f\otimes A)\circ(M\otimes\sigma(z))=f\triangleleft \sigma(z).$$

$(2)$ It follows by $(1)$.

$(3)$ Let $f:M\rightarrow N $ be an isomorphism of $A$-bimodules. Set $k:=1,f_1:=f$ and $g_1:=f^{-1}$. Clearly $\sum_{i=1}^k g_{i}\circ f_{i}={M}$ so that, by $(2)$, we get  $M_\sigma|N_\sigma$. If we interchange the roles of $M$ and $N$ we get also $N_\sigma|M_\sigma$.

$(4)$ Let $k\geq 1 $ and let $f_{i}\in\hom{A,A}{M}{N}$, $g_{i}\in\hom{A,A}{N}{M}$, for $i=1,\cdots,k$,
such that $\sum_{i}g_{i}\circ f_{i}={M}$. Then $f_{i}':=f_{i}\otimes_A L\in\hom{A,A}{M\tensor{A} L}{N\tensor{A}L}$, $g_{i}':=g_{i}\tensor{A} L\in\hom{A,A}{N\tensor{A} L}{M\tensor{A} L}$. Moreover $\sum_{i}g_{i}'\circ f_{i}'= {M\tensor{A} L}$

$(5)$ Let {$f\in\cat{M}_{A,\Z{A}}(M_\sigma,N_\tau)$} and $g\in \cat{M}_{A,\Z{A}}(L_\rho,M_\sigma)$. Then { $f\circ g\in \lhom{A}{L}{N}$}. Moreover $$f\circ g\circ(L\triangleleft \rho(z))=f\circ (g\triangleleft \sigma(z))=f\circ (M\triangleleft \sigma(z))\circ g= (f\triangleleft \tau(z))\circ g= (f\circ g)\triangleleft \tau(z).$$

$(6)$ By assumption there is $k\geq 1 $ and morphisms $f_{i}\in\cat{M}_{A,\Z{A}}(M_\sigma,N_\tau)$,
$g_{i}\in\cat{M}_{A,\Z{A}}(N_\tau,M_\sigma)$, for $i=1,\cdots,k$,
such that $\sum_{i}g_{i}\circ f_{i}={M}$. Moreover there  is $k'\geq 1 $ and morphisms $f'_{i}\in\cat{M}_{A,\Z{A}}(L_\rho,M_\sigma)$, $g'_{i}\in\cat{M}_{A,\Z{A}}(M_\sigma,L_\rho)$, for $i=1,\cdots,k'$,
such that $\sum_{i}g'_{i}\circ f'_{i}={L}$. Set $f_{i,j}:=f_i\circ f'_j$ and $g_{i,j}:=g'_i\circ g_j$. Then $\sum_{i,j}g_{i,j}\circ f_{j,i}=\sum_{i,j}g'_i\circ g_j\circ f_j\circ f'_i=L$ and, by $(5)$, we have that $f_{i,j}\in \cat{M}_{A,\Z{A}}(L_\rho,N_\tau)$ and $g_{i,j}\in \cat{M}_{A,\Z{A}}(N_\tau,L_\rho)$.

$(7)$ It follows from the obvious inclusion $\cat{M}_{A,\Z{A}}(M_\sigma,N_\tau) \subseteq \cat{M}_{A,\Z{A}}(M_{\sigma\rho},N_{\tau\rho})$.
\end{proof}

Take $M, N,  L$ to be three $A$-bimodules, and let $\sigma , \tau: \Z{A} \to \Z{A}$ be $\Z{\I}$-linear maps. We define
\begin{multline*}
\cat{M}_{A,\Z{A}} ( L\tensor{}M_{\sigma}, L\tensor{}N_{\tau})
:=\big\{ f \in {\rm Hom}_{A\text{-}}(L\tensor{}M, L\tensor{}N)| \, f \circ (L\tensor{}(M\triangleleft\sigma(z))) =f\triangleleft\tau(z), \forall z \in \Z{A}  \big\}
\end{multline*}
where $f\triangleleft\tau(z)= (L\tensor{}(N\triangleleft\tau(z))) \circ f$ as  the right $A$-module structure of $L\tensor{}N$ is the one coming from $N$. Using this set we define as in Definition \ref{def:divide}, the relations $L\tensor{}M_{\sigma} | L\tensor{}N_{\tau}$ and  $L\tensor{}M_{\sigma} \sim L\tensor{}N_{\tau}$.

\begin{proposition}\label{pro:side}
Let $M$ be an $A$-bimodule and let $\alpha,\beta$ be two monoids automorphisms of $A$. Set $\sigma:=\Z{\alpha}$ and $\tau:=\Z{\beta}$.
Then $M\otimes A_\sigma\sim M\otimes A_{\tau}$.
\end{proposition}
\begin{proof}
 Set $k:=1,f_1:=M\otimes \beta\alpha^{-1}$ and $g_1:=f_1^{-1}$. Clearly $f_1\in\cat{M}_{A,-}(M\otimes A,M\otimes A)$. Moreover
 \begin{multline*}
  f_1\circ ((M\otimes A)\triangleleft\sigma(z))=(M\otimes\beta)\circ(M\otimes\alpha^{-1})\circ(M\otimes({ A}\triangleleft\sigma(z)))
   =M\otimes[\beta\circ\alpha^{-1}\circ({ A}\triangleleft\sigma(z))]
\\  =M\otimes[(\beta\circ\alpha^{-1})\triangleleft(\beta\circ\alpha^{-1}\circ\sigma(z))]
 =M\otimes[(\beta\circ\alpha^{-1})\triangleleft(\beta\circ\alpha^{-1}\circ\alpha\circ z)]
 =M\otimes[(\beta\circ\alpha^{-1})\triangleleft (\beta\circ z)]
 \\=M\otimes[(\beta\circ\alpha^{-1})\triangleleft \tau(z)]
 =f_1\triangleleft \tau(z)
 \end{multline*}
 so that $f_1\in\cat{M}_{A,\Z{A}}(M\otimes A_\sigma,M\otimes A_\tau)$. A similar argument shows that $g_1\in\cat{M}_{A,\Z{A}}(M\otimes A_\tau,M\otimes A_\sigma)$. Since $\sum_{i}g_{i}\circ f_{i}={M\otimes A}$, we conclude that $M\otimes A_\sigma|M\otimes A_{\tau}$. If we interchange the roles of $M$ and $N$ and the ones of $\sigma$ and $\alpha$ we get also $M\otimes A_\tau|M\otimes A_{\sigma}$.
\end{proof}

\begin{theorem} \label{teo:azum}
Let $(A,m,u)$ be a monoid in $\cat{M}$ and assume that $A\otimes A\sim A$. Let $\alpha,\beta$ be two monoids automorphisms of $A$ and set $\sigma:=\Z{\alpha}$ and $\tau:=\Z{\beta}$. Then $A_\sigma\sim A_{\tau}$.
\end{theorem}

\begin{proof}
We split the proof in three steps.

1) Since $A\cong A\tensor{A} A$, by Proposition \ref{pro:div}(3), we get $A_\sigma\sim(A\tensor{A}A)_\sigma=A\tensor{A} A_\sigma$.

2) Since $A\sim A\otimes A$, in view of Proposition \ref{pro:div}(4), we have $A\tensor{A} A\sim A\otimes A\tensor{A} A$ and hence $A\tensor{A} A_\sigma\sim A\otimes A\tensor{A} A_\sigma$, by Proposition \ref{pro:div}(2).

3) Since $A\otimes A\tensor{A} A\cong A\otimes A$, by Proposition \ref{pro:div}(3), we get $A\otimes A\tensor{A} A_\sigma\sim A\otimes A_\sigma$.

Now, if we glue together 1), 2) and 3) using Proposition \ref{pro:div}(6), we obtain $A_\sigma\sim A\otimes A_\sigma$. Similarly one gets $A_\tau \sim A\otimes A_\tau.$ The conclusion follows by Propositions \ref{pro:side} and   \ref{pro:div}(6).
\end{proof}

\begin{corollary}\label{coro:Jinv}
Let $(A,m, u)$ be a monoid in $\cat{M}$ with a right flat underlying object $A$. Assume that $A\otimes A\sim A$ and $u_1:\I\rightarrow {_1 J_1}$ is an isomorphism. Let $\alpha$ be a  monoid automorphism of $A$ and set $\sigma:=\Z{\alpha}$.
Then ${_\sigma J_1}\in\mathrm{Inv}_{\I}(A)$ with inverse ${_1 J_\sigma}$. Moreover $\sigma=\Phi_{_\sigma J_1}\in \cat{G}_{\scriptscriptstyle{\Z{\I}}}(\Z{A})$, where $\Phi$ is the morphism of groups $\Phi^{\I}$ given in Proposition \ref{pro: Phi}. In particular we have a group homomorphism $$ \omega: \mathrm{Aut}_{alg}(A)\rightarrow \cat{G}_{\scriptscriptstyle{\Z{\I}}}(\Z{A}), \quad \big\{\gamma \mapsto \Z{\gamma}\big\},$$ where $\mathrm{Aut}_{alg}(A)$ stands for the automorphisms group of the monoid $A$ inside $\cat{M}$.
\end{corollary}
\begin{proof}
By  Theorem \ref{teo:azum}, we have that $A_{\sigma} \sim A_1$. Henceforth, we  apply  Corollary \ref{coro:tau 1} to obtain that $\sigma \in \cat{G}_{\scriptscriptstyle{\Z{\I}}}(\Z{A})$, which  means that ${_\sigma J_1}\in\mathrm{Inv}_{\I}(A)$  with inverse ${_1 J_\sigma}$. By applying now Theorem \ref {thm:mono}, it is clear that $\sigma=\widehat{\Phi} \widehat{\Phi}^{-1} (\sigma) = \Phi_{{}_{\sigma}J_1}$. The particular statement is immediate.
\end{proof}

The following  main result gives conditions under which the Miyashita action is bijective.

\begin{theorem}\label{thm:iso}
Let $(\cat{M},\tensor{}, \I)$ be a Penrose  monoidal abelian, locally small and bicomplete category with  right exact tensor products functors.
Let $(A,m,u)$ be a monoid in $\cat{M}$ with a right flat underlying object $A$.
Assume that $A\otimes A\sim A$ and $u_{1}:\I\rightarrow {_1 J_1}$ is an isomorphism. Consider the map $\omega$ as in Corollary \ref{coro:Jinv} and $\B{\varsigma}, \Phi$ as in Theorem \ref{thm:mono}.
\begin{enumerate}
 \item If the map $\Omega:\mathrm{Aut}_{alg}(A)\rightarrow \mathrm{Aut}_{\scriptscriptstyle{\Z{\I}}\text{-}alg}(\Z{A}), \,\gamma \mapsto \Z{\gamma}$ is surjective, then $\B{\varsigma}$ and $\Phi$ are bijective.
 \item If $\Omega$ is surjective and $\cat{Z}$ is faithful, then both $\Omega$ and $\omega$ are bijective. In this case,  we obtain a chain of isomorphisms of groups:
$$
\mathrm{Inv}_{\I}(A)\cong \mathrm{Aut}_{\scriptscriptstyle{\Z{\I}}\text{-}alg}(\Z{A})\cong  \mathrm{Aut}_{alg}(A) \cong \cat{G}_{\scriptscriptstyle{\Z{\I}}}(\Z{A}).
$$
\end{enumerate}
\end{theorem}
\begin{proof}
(1). Since $\Omega=\B{\varsigma} \circ \omega$ is surjective, it is clear that $\B{\varsigma}$ is surjective, and so bijective as it is by construction injective. Therefore, by Theorem \ref{thm:mono}, this implies that $\Phi=\B{\varsigma} \circ \widehat{\Phi}$ is bijective as well. (2). If $\cat{Z}$ is faithful, then clearly $\Omega$ is injective, hence it is also bijective. Thus $\omega= \B{\varsigma}^{-1} \circ \Omega$ is also bijective. The stated chain of isomorphisms  of groups is now immediate.
\end{proof}

\begin{remark} Under the assumptions made on $A$ in Theorem  \ref{thm:iso},  if we further  assume that
$\cat{Z}$ is full and faithful, then clearly we obtain the chain of isomorphisms stated there. As we will see below,  part of those isomorphisms are also obtained in the special case of certain Azumaya monoids.
\end{remark}

\section{Miyashita action: The Azumaya case}\label{sec:azymayacentral}
The main aim in  this section is to apply the foregoing results to  a certain class of Azumaya monoids (see Definition \ref{def:azumaya} below) when the base category in symmetric. The novelty here is the construction of a homomorphism of groups $\B{\Gamma}$  from $\mathrm{Inv}_{\I}(A)$ to $\mathrm{Aut}_{alg}(A)$, which is shown to be bijective when $A$ is an Azumaya monoid (with a certain property) and $\cat{Z}$ is faithful, see  Corollary \ref{thm:isoAzumaya}.

In all this section $(\cat{M}, \tensor{}, \I)$ is  symmetric  monoidal abelian, locally small and bicomplete category with  right exact tensor products functors. We denote by $\B{\tau}_{\scriptscriptstyle{M,\,N}}: M\tensor{}N \to N\tensor{}M$ the natural isomorphism defining the symmetry of $\cat{M}$.

\vspace{-0.3cm}
\subsection{The map $\B{\Gamma}$}\label{ssec:gamma}
Let $(A,m,u)$ be a monoid in $\cat{M}$ with $u$ a monomorphsm. Consider  $X\in \mathrm{Inv}_{\I}(A)$ with inverse $Y$. Define the morphism
$$
\xymatrix@C=45pt{\B{\Gamma}_{(X,Y)}:  A \ar@{->}^-{\cong}[r] & A\tensor{}X\tensor{}Y \ar@{->}^-{\B{\tau}\tensor{}Y}[r]  & X\tensor{}A\tensor{}Y \ar@{>}^-{i_X\tensor{}A\tensor{}i_Y}[r] &  A\tensor{}A\tensor{}A \ar@{->}^-{m \circ (A\tensor{}m)}[r]  & A.}
$$
An argument analogue to the one used in Proposition \ref{pro: right Miyashita} for $\B{\sigma}_{X,Y}$, shows that this morphism does not depend on the choice of $Y$ (nor on a representing  object of the equivalence class of $(X,i_X)$)  so that we can also denote it by $\B{\Gamma}_{X}$. This, in fact, defines a map to the endomorphisms ring of $A$ with image in the automorphisms group as the following result shows.
First we give the following diagrammatic expression of $\B{\Gamma}_{X}$
\begin{center}
\begin{tikzpicture}[x=7pt,y=7pt,thick]\pgfsetlinewidth{0.5pt}
\node[inner sep=1pt] at (0,0) {$\B{\Gamma}_{X} \,:=\,$};
\node at (0,-7) {};
\end{tikzpicture}
\begin{tikzpicture}[x=7pt,y=7pt,thick]\pgfsetlinewidth{0.5pt}
\node(1) at (-1.5,5) {$\scriptstyle{A}$};
\node(a) at (0,-1.5) {};
\node[circle,draw, inner sep=0.7pt](2) at (1.5,1) {$\scriptstyle{i_X}$};
\node[circle,draw, inner sep=0.7pt](3) at (4.5,1) {$\scriptstyle{i_Y}$};
\node(4) at (0.7,-8) {$\scriptstyle{A}$};

\draw[-] (2) to [out=90,in=-90] (1.5,3);
\draw[-] (3) to [out=90,in=-90] (4.5,3);
\draw[-] (4.5,3) to [out=90,in=90] (1.5,3);

\draw[-] (-1.5,-3) to [out=90,in=200] (a);
\draw[-] (a) to [out=10,in=-90] (1.5,0);
\draw[-] (-1.5,0) to [out=-90,in=90] (1.5,-3);

\draw[-] (3) to [out=-90,in=90] (4.5,-3);
\draw[-] (1.5,-3) to [out=-90,in=-90] (4.5,-3);

\draw[-] (1) to [out=-90,in=90] (-1.5,0);

\draw[-] (3,-4) to [out=-90,in=90] (3,-5);
\draw[-] (-1.5,-3) to [out=-90,in=90] (-1.5,-5);
\draw[-] (-1.5,-5) to [out=-90,in=-90] (3,-5);

\draw[-] (4) to [out=90,in=-90] (0.7,-6.3);
\end{tikzpicture}
\end{center}

\begin{proposition}\label{prop:Gamma}
The map
$$\B{\Gamma}:\mathrm{Inv}_{\I}(A)\longrightarrow \mathrm{Aut}_{alg}(A),\quad \big(X\longmapsto \Gamma_{X}\big)$$
is a homomorphism of groups.  In particular we have
$\Phi\,=\, \Omega \circ \B{\Gamma}$, where $\Omega$ is as in  Theorem \ref{thm:iso} and $\Phi:=\Phi^{\I}$ as in Proposition \ref{pro: Phi}.
\end{proposition}
\begin{proof}
In view of Corollary \ref{coro:InvGrp} (here the assumption $u$ is monomorphism is used), it suffices to prove that we have a morphism of monoids
$$\B{\Gamma}:(\mathrm{Inv}_{\I}^r(A),\otimes,\I)\longrightarrow (\mathrm{End}_{alg}(A),\circ,\id_A),\quad \Big(X\longmapsto \Gamma_{X}\Big).$$
First we have to check that this map is well-defined, that is $\B{\Gamma}_{X}$ is a monoid endomorphism of $A$.
Let us check that  $\B{\Gamma}_{X}$ is unitary
\begin{center}
\begin{tikzpicture}[x=8pt,y=8pt,thick]\pgfsetlinewidth{0.5pt}
\node[inner sep=1pt] at (0,0) {$\B{\Gamma}_{X} \circ u \,\,=\, $};
\node at (0,-8) {};
\end{tikzpicture}
\begin{tikzpicture}[x=8pt,y=8pt,thick]\pgfsetlinewidth{0.5pt}
\node[circle,draw, inner sep=1.5pt](u) at (0,5) {$\scriptstyle{u}$};
\node(a) at (1.5,-2.5) {};
\node[circle,draw, inner sep=0.7pt](2) at (5,2) {$\scriptstyle{i_X}$};
\node[circle,draw, inner sep=0.7pt](33) at (8,2) {$\scriptstyle{i_Y}$};
\node(4) at (4,-8) {$\scriptstyle{A}$};

\draw[-] (2) to [out=90,in=-90] (5,4);
\draw[-] (33) to [out=90,in=-90] (8,4);
\draw[-] (8,4) to [out=90,in=90] (5,4);

\draw[-] (a) to [out=90,in=200] (2.5,-0.7);
\draw[-] (1.5,-2) to [out=-90,in=90] (1.5,-4);

\draw[-] (u) to [out=-90,in=90] (0,2);

\draw[-] (0,2) to [out=-90,in=90] (5,-3);
\draw[-] (33) to [out=-90,in=90] (8,-3);
\draw[-] (8,-3) to [out=-90,in=-90] (5,-3);

\draw[-] (2) to [out=-90,in=20] (3.2,-0.5);

\draw[-] (1.5,-4) to [out=-90,in=-90] (6.5,-4);

\draw[-] (4) to [out=90,in=-90] (4,-5.5);
\end{tikzpicture}
\begin{tikzpicture}[x=8pt,y=8pt,thick]\pgfsetlinewidth{0.5pt}
\node[inner sep=1pt] at (0,0) {$\,=\,$};
\node at (0,-8) {};
\end{tikzpicture}
\begin{tikzpicture}[x=8pt,y=8pt,thick]\pgfsetlinewidth{0.5pt}
\node[circle,draw, inner sep=1.5pt](u) at (0,0) {$\scriptstyle{u}$};

\node[circle,draw, inner sep=0.7pt](2) at (-2,2) {$\scriptstyle{i_X}$};
\node[circle,draw, inner sep=0.7pt](33) at (2,2) {$\scriptstyle{i_Y}$};
\node(4) at (-0.5,-8) {$\scriptstyle{A}$};

\draw[-] (2) to [out=90,in=-90] (-2,4);
\draw[-] (33) to [out=90,in=-90] (2,4);
\draw[-] (-2,4) to [out=90,in=90] (2,4);

\draw[-] (u) to [out=-90,in=90] (0,-2);
\draw[-] (33) to [out=-90,in=90] (2,-2);
\draw[-] (2,-2) to [out=-90,in=-90] (0,-2);

\draw[-] (1,-2.6) to [out=-90,in=-90] (-2,-2.6);

\draw[-] (2) to [out=-90,in=90] (-2,-2.6);

\draw[-] (4) to [out=90,in=-90] (-0.5,-3.5);
\end{tikzpicture}
\begin{tikzpicture}[x=8pt,y=8pt,thick]\pgfsetlinewidth{0.5pt}
\node[inner sep=1pt] at (0,0) {$\,=\,$};
\node at (0,-8) {};
\end{tikzpicture}
\begin{tikzpicture}[x=8pt,y=8pt,thick]\pgfsetlinewidth{0.5pt}
\node[circle,draw, inner sep=0.7pt](2) at (-2,2) {$\scriptstyle{i_X}$};
\node[circle,draw, inner sep=0.7pt](33) at (2,2) {$\scriptstyle{i_Y}$};
\node(4) at (0,-8) {$\scriptstyle{A}$};

\draw[-] (2) to [out=90,in=-90] (-2,4);
\draw[-] (33) to [out=90,in=-90] (2,4);
\draw[-] (-2,4) to [out=90,in=90] (2,4);
\draw[-] (33) to [out=-90,in=90] (2,-2);
\draw[-] (2,-2) to [out=-90,in=-90] (-2,-2);
\draw[-] (2) to [out=-90,in=90] (-2,-2);
\draw[-] (4) to [out=90,in=-90] (0,-3.2);
\end{tikzpicture}
\begin{tikzpicture}[x=8pt,y=8pt,thick]\pgfsetlinewidth{0.5pt}
\node[inner sep=1pt] at (0,0) {$\,=\, u$};
\node at (0,-8) {};
\end{tikzpicture}
\end{center}

Let us check that $\B{\Gamma}_{X}$ is multiplicative. We start by computing $m \circ (\B{\Gamma}_{X}\tensor{} \B{\Gamma}_{X})\,=\,$
\begin{center}
\begin{tikzpicture}[x=7pt,y=7pt,thick]\pgfsetlinewidth{0.5pt}
\node(1) at (-1.5,5) {$\scriptstyle{A}$};
\node(a) at (0,-1.5) {};
\node[circle,draw, inner sep=0.7pt](2) at (1.5,1) {$\scriptstyle{i_X}$};
\node[circle,draw, inner sep=0.7pt](3) at (4.5,1) {$\scriptstyle{i_Y}$};
\node(4) at (5,-11) {$\scriptstyle{A}$};
\node(5) at (6,5) {$\scriptstyle{A}$};

\draw[-] (2) to [out=90,in=-90] (1.5,3);
\draw[-] (3) to [out=90,in=-90] (4.5,3);
\draw[-] (4.5,3) to [out=90,in=90] (1.5,3);

\draw[-] (-1.5,-3) to [out=90,in=200] (a);
\draw[-] (a) to [out=10,in=-90] (1.5,0);
\draw[-] (-1.5,0) to [out=-90,in=90] (1.5,-3);

\draw[-] (3) to [out=-90,in=90] (4.5,-3);
\draw[-] (1.5,-3) to [out=-90,in=-90] (4.5,-3);

\draw[-] (1) to [out=-90,in=90] (-1.5,0);

\draw[-] (3,-3.9) to [out=-90,in=90] (3,-5);
\draw[-] (-1.5,-3) to [out=-90,in=90] (-1.5,-5);
\draw[-] (-1.5,-5) to [out=-90,in=-90] (3,-5);

\node(b) at (8,-1.5) {};
\node[circle,draw, inner sep=0.7pt](22) at (9.5,1) {$\scriptstyle{i_X}$};
\node[circle,draw, inner sep=0.7pt](33) at (12.5,1) {$\scriptstyle{i_Y}$};

\draw[-] (22) to [out=90,in=-90] (9.5,3);
\draw[-] (33) to [out=90,in=-90] (12.5,3);
\draw[-] (9.5,3) to [out=90,in=90] (12.5,3);

\draw[-] (6.5,-3) to [out=90,in=200] (b);
\draw[-] (b) to [out=10,in=-90] (9.5,0);
\draw[-] (6,0) to [out=-90,in=90] (9.5,-3);
\draw[-] (6,0) to [out=90,in=-90] (5);

\draw[-] (33) to [out=-90,in=90] (12.5,-3);
\draw[-] (9.5,-3) to [out=-90,in=-90] (12.5,-3);

\draw[-] (6.5,-3) to [out=-90,in=90] (6.5,-5);
\draw[-] (11,-3.9) to [out=-90,in=90] (11,-5);
\draw[-] (11,-5) to [out=-90,in=-90] (6.5,-5);

\draw[-] (8.75,-6.3) to [out=-90,in=-90] (0.7,-6.3);

\draw[-] (4) to [out=90,in=-90] (5,-8.6);
\end{tikzpicture}
\begin{tikzpicture}[x=7pt,y=7pt,thick]\pgfsetlinewidth{0.5pt}
\node[inner sep=1pt] at (0,-3) {\,=\,};
\node at (0,-11) {};
\end{tikzpicture}
\begin{tikzpicture}[x=7pt,y=7pt,thick]\pgfsetlinewidth{0.5pt}
\node(1) at (-1.5,5) {$\scriptstyle{A}$};
\node(a) at (0,-1.5) {};
\node[circle,draw, inner sep=0.7pt](2) at (1.5,1) {$\scriptstyle{i_X}$};
\node[circle,draw, inner sep=0.7pt](3) at (4.5,-1) {$\scriptstyle{i_Y}$};
\node(4) at (7,-11) {$\scriptstyle{A}$};
\node(5) at (7.5,5) {$\scriptstyle{A}$};

\draw[-] (2) to [out=90,in=-90] (1.5,3);
\draw[-] (3) to [out=90,in=-90] (4.5,3);
\draw[-] (4.5,3) to [out=90,in=90] (1.5,3);

\draw[-] (-1.5,-3) to [out=90,in=200] (a);
\draw[-] (a) to [out=10,in=-90] (1.5,0);
\draw[-] (-1.5,0) to [out=-90,in=90] (1.5,-3);

\draw[-] (1) to [out=-90,in=90] (-1.5,0);

\draw[-] (1.5,-3) to [out=-90,in=90] (1.5,-4);
\draw[-] (-1.5,-3) to [out=-90,in=90] (-1.5,-4);
\draw[-] (-1.5,-4) to [out=-90,in=-90] (1.5,-4);

\node[circle,draw, inner sep=0.7pt](22) at (7.5,-1) {$\scriptstyle{i_X}$};
\node[circle,draw, inner sep=0.7pt](33) at (12.5,1) {$\scriptstyle{i_Y}$};

\draw[-] (22) to [out=-90,in=90] (7.5,-4);
\draw[-] (3) to [out=-90,in=90] (4.5,-4);
\draw[-] (7.5,-4) to [out=-90,in=-90] (4.5,-4);
\draw[-] (6,-4.8) to [out=-90,in=-90] (0,-4.8);

\draw[-] (5) to [out=-90,in=90] (9.5,-1);
\draw[-] (22) to [out=90,in=200] (8,1.8);
\draw[-] (33) to [out=90,in=-90] (12.5,3);
\draw[-] (12.5,3) to [out=90,in=90] (9.5,3);
\draw[-] (9.5,3) to [out=-90,in=20] (8.5,2);

\draw[-] (33) to [out=-90,in=90] (12.5,-3);
\draw[-] (9.5,-3) to [out=-90,in=-90] (12.5,-3);
\draw[-] (9.5,-1) to [out=90,in=-90] (9.5,-3);

\draw[-] (11,-3.9) to [out=-90,in=90] (11,-6.6);

\draw[-] (3,-6.6) to [out=-90,in=-90] (11,-6.6);

\draw[-] (4) to [out=90,in=-90] (7,-9);
\end{tikzpicture}
\begin{tikzpicture}[x=7pt,y=7pt,thick]\pgfsetlinewidth{0.5pt}
\node[inner sep=1pt] at (0,-1) {\,=\,};
\node at (0,-11) {};
\end{tikzpicture}
\begin{tikzpicture}[x=7pt,y=7pt,thick]\pgfsetlinewidth{0.5pt}
\node(1) at (-1.5,5) {$\scriptstyle{A}$};
\node(a) at (0,-1.5) {};
\node[circle,draw, inner sep=0.7pt](2) at (1.5,1) {$\scriptstyle{i_X}$};

\node(4) at (7,-11) {$\scriptstyle{A}$};
\node(5) at (7.5,5) {$\scriptstyle{A}$};

\draw[-] (2) to [out=90,in=-90] (1.5,3);
\draw[-] (3) to [out=90,in=-90] (4.5,3);
\draw[-] (4.5,3) to [out=90,in=90] (1.5,3);

\draw[-] (-1.5,-3) to [out=90,in=200] (a);
\draw[-] (a) to [out=10,in=-90] (1.5,0);
\draw[-] (-1.5,0) to [out=-90,in=90] (1.5,-3);

\draw[-] (1) to [out=-90,in=90] (-1.5,0);

\draw[-] (1.5,-3) to [out=-90,in=90] (1.5,-4);
\draw[-] (-1.5,-3) to [out=-90,in=90] (-1.5,-4);
\draw[-] (-1.5,-4) to [out=-90,in=-90] (1.5,-4);

\node[circle,draw, inner sep=1.5pt](u) at (6,-3) {$\scriptstyle{u}$};
\node[circle,draw, inner sep=0.7pt](33) at (12.5,1) {$\scriptstyle{i_Y}$};

\draw[-] (4.5,0) to [out=-90,in=-90] (7.5,0);

\draw[-] (u) to [out=-90,in=90] (6,-4.8);
\draw[-] (6,-4.8) to [out=-90,in=-90] (0,-4.8);

\draw[-] (5) to [out=-90,in=90] (9.5,-1);
\draw[-] (22) to [out=90,in=200] (8,1.8);
\draw[-] (33) to [out=90,in=-90] (12.5,3);
\draw[-] (12.5,3) to [out=90,in=90] (9.5,3);
\draw[-] (9.5,3) to [out=-90,in=20] (8.5,2);

\draw[-] (33) to [out=-90,in=90] (12.5,-3);
\draw[-] (9.5,-3) to [out=-90,in=-90] (12.5,-3);
\draw[-] (9.5,-1) to [out=90,in=-90] (9.5,-3);

\draw[-] (11,-3.9) to [out=-90,in=90] (11,-6.6);

\draw[-] (3,-6.6) to [out=-90,in=-90] (11,-6.6);

\draw[-] (4) to [out=90,in=-90] (7,-9);
\end{tikzpicture}

\begin{tikzpicture}[x=7pt,y=7pt,thick]\pgfsetlinewidth{0.5pt}
\node[inner sep=1pt] at (0,-1) {$\,\overset{\eqref{Eq:R}}{=}\,$};
\node at (0,-11) {};
\end{tikzpicture}
\begin{tikzpicture}[x=7pt,y=7pt,thick]\pgfsetlinewidth{0.5pt}
\node(1) at (-1.5,5) {$\scriptstyle{A}$};
\node(a) at (0,-2.5) {};
\node[circle,draw, inner sep=0.7pt](2) at (5,2) {$\scriptstyle{i_X}$};
\node[circle,draw, inner sep=0.7pt](33) at (8,2) {$\scriptstyle{i_Y}$};
\node(4) at (3.25,-11) {$\scriptstyle{A}$};
\node(5) at (1.5,5) {$\scriptstyle{A}$};

\draw[-] (2) to [out=90,in=-90] (5,4);
\draw[-] (33) to [out=90,in=-90] (8,4);
\draw[-] (8,4) to [out=90,in=90] (5,4);

\draw[-] (-1.5,-4) to [out=90,in=200] (a);
\draw[-] (a) to [out=10,in=200] (2.5,-0.7);
\draw[-] (-1.5,-1) to [out=-90,in=90] (1.5,-4);

\draw[-] (1) to [out=-90,in=90] (-1.5,-1);

\draw[-] (1.5,-4) to [out=-90,in=90] (1.5,-5);
\draw[-] (-1.5,-4) to [out=-90,in=90] (-1.5,-5);
\draw[-] (-1.5,-5) to [out=-90,in=-90] (1.5,-5);

\draw[-] (5) to [out=-90,in=90] (1.5,1.3);

\draw[-] (1.5,1.3) to [out=-90,in=90] (5,-3);
\draw[-] (33) to [out=-90,in=90] (8,-3);
\draw[-] (5,-3) to [out=-90,in=-90] (8,-3);

\draw[-] (2) to [out=-90,in=20] (3.2,-0.5);

\draw[-] (6.5,-3.9) to [out=-90,in=90] (6.5,-6);

\draw[-] (0,-6) to [out=-90,in=-90] (6.5,-6);

\draw[-] (4) to [out=90,in=-90] (3.25,-7.9);
\end{tikzpicture}
\begin{tikzpicture}[x=7pt,y=7pt,thick]\pgfsetlinewidth{0.5pt}
\node[inner sep=1pt] at (0,-1) {\,=\,};
\node at (0,-11) {};
\end{tikzpicture}
\begin{tikzpicture}[x=7pt,y=7pt,thick]\pgfsetlinewidth{0.5pt}
\node(1) at (-1.5,5) {$\scriptstyle{A}$};
\node(a) at (0,-2.5) {};
\node[circle,draw, inner sep=0.7pt](2) at (5,2) {$\scriptstyle{i_X}$};
\node[circle,draw, inner sep=0.7pt](33) at (8,2) {$\scriptstyle{i_Y}$};
\node(4) at (2.25,-11) {$\scriptstyle{A}$};
\node(5) at (1.5,5) {$\scriptstyle{A}$};

\draw[-] (2) to [out=90,in=-90] (5,4);
\draw[-] (33) to [out=90,in=-90] (8,4);
\draw[-] (8,4) to [out=90,in=90] (5,4);

\draw[-] (-1.5,-4) to [out=90,in=200] (a);
\draw[-] (a) to [out=10,in=200] (2.5,-0.7);
\draw[-] (-1.5,-1) to [out=-90,in=90] (1.5,-4);

\draw[-] (1) to [out=-90,in=90] (-1.5,-1);

\draw[-] (1.5,-4) to [out=-90,in=90] (1.5,-4);
\draw[-] (-1.5,-4) to [out=-90,in=90] (-1.5,-6.2);

\draw[-] (-1.5,-6.2) to [out=-90,in=-90] (5.7,-6.4);

\draw[-] (5) to [out=-90,in=90] (1.5,1.3);

\draw[-] (1.5,1.3) to [out=-90,in=90] (5,-3);
\draw[-] (33) to [out=-90,in=90] (8,-5);
\draw[-] (5,-3) to [out=-90,in=90] (5,-4);

\draw[-] (5,-4) to [out=-90,in=-90] (1.5,-4);
\draw[-] (2) to [out=-90,in=20] (3.2,-0.5);

\draw[-] (8,-5) to [out=-90,in=-90] (3.25,-5);

\draw[-] (4) to [out=90,in=-90] (2.25,-8.5);
\end{tikzpicture}
\begin{tikzpicture}[x=7pt,y=7pt,thick]\pgfsetlinewidth{0.5pt}
\node[inner sep=1pt] at (0,0) {\,=\,};
\node at (0,-8) {};
\end{tikzpicture}
\begin{tikzpicture}[x=7pt,y=7pt,thick]\pgfsetlinewidth{0.5pt}
\node(1) at (-1.5,5) {$\scriptstyle{A}$};
\node(a) at (1.5,-2.5) {};
\node[circle,draw, inner sep=0.7pt](2) at (5,2) {$\scriptstyle{i_X}$};
\node[circle,draw, inner sep=0.7pt](33) at (8,2) {$\scriptstyle{i_Y}$};
\node(4) at (4,-8) {$\scriptstyle{A}$};
\node(5) at (1.5,5) {$\scriptstyle{A}$};

\draw[-] (2) to [out=90,in=-90] (5,4);
\draw[-] (33) to [out=90,in=-90] (8,4);
\draw[-] (8,4) to [out=90,in=90] (5,4);

\draw[-] (a) to [out=90,in=200] (2.5,-0.7);
\draw[-] (1.5,-2) to [out=-90,in=90] (1.5,-4);

\draw[-] (1) to [out=-90,in=90] (-1.5,3);
\draw[-] (5) to [out=-90,in=90] (1.5,3);
\draw[-] (-1.5,3) to [out=-90,in=-90] (1.5,3);

\draw[-] (0,2) to [out=-90,in=90] (5,-3);
\draw[-] (33) to [out=-90,in=90] (8,-3);
\draw[-] (8,-3) to [out=-90,in=-90] (5,-3);

\draw[-] (2) to [out=-90,in=20] (3.2,-0.5);

\draw[-] (1.5,-4) to [out=-90,in=-90] (6.5,-4);

\draw[-] (4) to [out=90,in=-90] (4,-5.5);
\end{tikzpicture}
\begin{tikzpicture}[x=7pt,y=7pt,thick]\pgfsetlinewidth{0.5pt}
\node[inner sep=1pt] at (0,0) {$\,=\, \B{\Gamma}_{X} \circ m.$};
\node at (0,-8) {};
\end{tikzpicture}
\end{center}
It is clear that $\Gamma_{\I}=\id_A$. Let now $X,X'\in\mathrm{Inv}_{\I}^r(A)$ have right inverses $Y,Y'$ respectively. Then, we compute $\B{\Gamma}_{X\tensor{}X' } \,=\,$
\begin{center}
\begin{tikzpicture}[x=8pt,y=8pt,thick]\pgfsetlinewidth{0.5pt}
\node(1) at (-3,6) {$\scriptstyle{A}$};
\node[circle,draw, inner sep=0.7pt](2) at (0,1) {$\scriptstyle{i_X}$};
\node[circle,draw, inner sep=0.7pt](3) at (3,1) {$\scriptstyle{i_{X'}}$};
\node[circle,draw, inner sep=0.7pt](4) at (7,1) {$\scriptstyle{i_{Y'}}$};
\node[circle,draw, inner sep=0.7pt](5) at (10,1) {$\scriptstyle{i_{Y}}$};
\node(6) at (0.5,-10) {$\scriptstyle{A}$};

\draw[-] (2) to [out=-90,in=-90] (3);
\draw[-] (3) to [out=90,in=90] (4);
\draw[-] (4) to [out=-90,in=-90] (5);
\draw[-] (2) to [out=90,in=90] (5);

\draw[-] (1.5,-0.6) to [out=-90, in=10] (0,-3.1);
\draw[-] (1) to [out=-90,in=90] (-3,-1);
\draw[-] (-3,-1) to [out=-90,in=90] (1.5,-5);
\draw[-] (-2,-5) to [out=90,in=200] (-0.5,-3.3);
\draw[-] (-2,-5) to [out=-90,in=90] (-2,-6);

\draw[-] (8.5,-0.6) to [out=-90,in=90] (5,-5);
\draw[-] (5,-5) to [out=-90,in=-90] (1.5,-5);
\draw[-] (3.5,-6) to [out=-90,in=-90] (-2,-6);

\draw[-] (6) to [out=90,in=-90] (0.5,-7.6);

\end{tikzpicture}
\begin{tikzpicture}[x=8pt,y=8pt,thick]\pgfsetlinewidth{0.5pt}
\node[inner sep=1pt] at (0,0) {$\,=\,$};
\node at (0,-9) {};
\end{tikzpicture}
\begin{tikzpicture}[x=8pt,y=8pt,thick]\pgfsetlinewidth{0.5pt}
\node(1) at (-3,6) {$\scriptstyle{A}$};
\node[circle,draw, inner sep=0.7pt](2) at (0,1) {$\scriptstyle{i_X}$};
\node[circle,draw, inner sep=0.7pt](3) at (3,1) {$\scriptstyle{i_{X'}}$};
\node[circle,draw, inner sep=0.7pt](4) at (7,1) {$\scriptstyle{i_{Y'}}$};
\node[circle,draw, inner sep=0.7pt](5) at (10,1) {$\scriptstyle{i_{Y}}$};
\node(6) at (1,-10) {$\scriptstyle{A}$};

\draw[-] (3) to [out=90,in=90] (4);
\draw[-] (2) to [out=90,in=90] (5);
\draw[-] (4) to [out=-90,in=90] (7,-2);
\draw[-] (5) to [out=-90,in=90] (10,-2);
\draw[-] (7,-2) to [out=-90,in=-90] (10,-2);
\draw[-] (8.5,-2.9) to [out=-90,in=90] (5,-5);

\draw[-] (3) to [out=-90, in=10] (0.6,-3);
\draw[-] (2) to [out=-90,in=10] (-1.5,-2.1);

\draw[-] (1) to [out=-90,in=90] (-3,-0.5);
\draw[-] (-3,-0.5) to [out=-90,in=90] (2,-5);
\draw[-] (-3,-5) to [out=90,in=200] (-2,-2.3);
\draw[-] (-0.5,-5) to [out=90,in=200] (0.3,-3.3);

\draw[-] (-0.5,-5) to [out=-90,in=-90] (-3,-5);
\draw[-] (5,-5) to [out=-90,in=-90] (2,-5);

\draw[-] (3.5,-5.9) to [out=-90,in=90] (2,-7.5);
\draw[-] (-2,-5.7) to [out=-90,in=90] (0,-7.5);
\draw[-] (2,-7.5) to [out=-90,in=-90] (0,-7.5);
\draw[-] (6) to [out=90,in=-90] (1,-8.1);
\end{tikzpicture}
\begin{tikzpicture}[x=8pt,y=8pt,thick]\pgfsetlinewidth{0.5pt}
\node[inner sep=1pt] at (0,0) {$\,=\,$};
\node at (0,-9) {};
\end{tikzpicture}
\begin{tikzpicture}[x=8pt,y=8pt,thick]\pgfsetlinewidth{0.5pt}
\node(1) at (-3,6) {$\scriptstyle{A}$};
\node[circle,draw, inner sep=0.7pt](2) at (0,1) {$\scriptstyle{i_X}$};
\node[circle,draw, inner sep=0.7pt](3) at (3,1) {$\scriptstyle{i_{X'}}$};
\node[circle,draw, inner sep=0.7pt](4) at (7,1) {$\scriptstyle{i_{Y'}}$};
\node[circle,draw, inner sep=0.7pt](5) at (10,1) {$\scriptstyle{i_{Y}}$};
\node(6) at (1,-12) {$\scriptstyle{A}$};

\draw[-] (3) to [out=90,in=90] (4);
\draw[-] (2) to [out=90,in=90] (5);
\draw[-] (4) to [out=-90,in=90] (5,-5);
\draw[-] (5,-5) to [out=-90,in=-90] (2,-5);

\draw[-] (3) to [out=-90, in=10] (0.6,-3);
\draw[-] (2) to [out=-90,in=10] (-1.5,-2.1);

\draw[-] (1) to [out=-90,in=90] (-3,-0.5);
\draw[-] (-3,-0.5) to [out=-90,in=90] (2,-5);
\draw[-] (-3,-5) to [out=90,in=200] (-2,-2.3);
\draw[-] (-0.5,-6) to [out=90,in=200] (0.3,-3.3);

\draw[-] (3.5,-6) to [out=-90,in=-90] (-0.5,-6);
\draw[-] (1.5,-7.1) to [out=-90,in=-90] (7,-7);
\draw[-] (5) to [out=-90,in=90] (7,-7);
\draw[-] (-3,-5) to [out=-90,in=90] (-3,-8);
\draw[-] (1,-10) to [out=0,in=-90] (4,-8.6);
\draw[-] (1,-10) to [out=180,in=-90] (-3,-8);

\draw[-] (6) to [out=90,in=-90] (1,-10);
\end{tikzpicture}

\begin{tikzpicture}[x=8pt,y=8pt,thick]\pgfsetlinewidth{0.5pt}
\node[inner sep=1pt] at (0,0) {$\,=\,$};
\node at (0,-9) {};
\end{tikzpicture}
\begin{tikzpicture}[x=8pt,y=8pt,thick]\pgfsetlinewidth{0.5pt}
\node(1) at (-3,4) {$\scriptstyle{A}$};
\node[circle,draw, inner sep=0.7pt](2) at (0,1) {$\scriptstyle{i_X}$};
\node[circle,draw, inner sep=0.7pt](3) at (6,1) {$\scriptstyle{i_{Y}}$};
\node[ellipse,draw, inner sep=1pt](4) at (0.35,-4.8) {$\scriptstyle{\B{\Gamma}_{X'}}$};
\node(6) at (1,-12) {$\scriptstyle{A}$};

\draw[-] (3) to [out=90,in=90] (2);
\draw[-] (3) to [out=-90,in=90] (6,-7.5);
\draw[-] (4) to [out=-90,in=90] (3,-7.5);
\draw[-] (6,-7.5) to [out=-90,in=-90] (3,-7.5);
\draw[-] (2) to [out=-90,in=10] (-1.5,-1.9);
\draw[-] (1) to [out=-90,in=90] (-3,-0.5);
\draw[-] (-3,-0.5) to [out=-90,in=90] (0,-4);
\draw[-] (-3,-5) to [out=90,in=200] (-2,-2.2);
\draw[-] (4.5,-8.4) to [out=-90,in=0] (1,-10);
\draw[-] (1,-10) to [out=180,in=-90] (-3,-8);
\draw[-] (-3,-8) to [out=90,in=-90] (-3,-5);
\draw[-] (6) to [out=90,in=-90] (1,-10);
\end{tikzpicture}
\begin{tikzpicture}[x=8pt,y=8pt,thick]\pgfsetlinewidth{0.5pt}
\node[inner sep=1pt] at (0,0) {$\,=\,$};
\node at (0,-9) {};
\end{tikzpicture}
\begin{tikzpicture}[x=8pt,y=8pt,thick]\pgfsetlinewidth{0.5pt}
\node(1) at (-3,4) {$\scriptstyle{A}$};
\node[circle,draw, inner sep=0.7pt](2) at (2,-2) {$\scriptstyle{i_X}$};
\node[circle,draw, inner sep=0.7pt](3) at (5,-2) {$\scriptstyle{i_{Y}}$};
\node[ellipse,draw, inner sep=1pt](4) at (-3,1) {$\scriptstyle{\B{\Gamma}_{X'}}$};
\node(6) at (1,-12) {$\scriptstyle{A}$};

\draw[-] (3) to [out=90,in=-90] (5,0);
\draw[-] (2) to [out=90,in=-90] (2,0);
\draw[-] (2,0) to [out=90,in=90] (5,0);

\draw[-] (3) to [out=-90,in=90] (5,-6);
\draw[-] (4) to [out=-90,in=90] (-3,-3);

\draw[-] (5,-6) to [out=-90,in=-90] (2,-6);
\draw[-] (-3,-3) to [out=-90,in=90] (2,-6);
\draw[-] (1) to [out=-90,in=90] (4);

\draw[-] (2) to [out=-90,in=10] (0.3,-4.4);
\draw[-] (-2,-7) to [out=90,in=200] (-0.5,-4.7);
\draw[-] (3.5,-7) to [out=-90,in=0] (1,-9);

\draw[-] (1,-9) to [out=180,in=-90] (-2,-7);
\draw[-] (6) to [out=90,in=-90] (1,-9);
\end{tikzpicture}
\begin{tikzpicture}[x=8pt,y=8pt,thick]\pgfsetlinewidth{0.5pt}
\node[inner sep=1pt] at (0,0) {$\,=\,$};
\node at (0,-9) {};
\end{tikzpicture}
\begin{tikzpicture}[x=8pt,y=8pt,thick]\pgfsetlinewidth{0.5pt}
\node(1) at (0,4) {$A$};
\node[ellipse,draw, inner sep=1pt](2) at (0,0) {$\scriptstyle{\B{\Gamma}_{X'}}$};
\node[ellipse,draw, inner sep=1.2pt](3) at (0,-6) {$\scriptstyle{\B{\Gamma}_{X}}$};
\node(4) at (0,-12) {$\scriptstyle{A}$};

\draw[-] (1) to [out=-90,in=90] (2);
\draw[-] (2) to [out=-90,in=90] (3);
\draw[-] (3) to [out=-90,in=90] (4);
\end{tikzpicture}
\begin{tikzpicture}[x=8pt,y=8pt,thick]\pgfsetlinewidth{0.5pt}
\node[inner sep=1pt] at (0,0) {$\,=\, \B{\Gamma}_{X} \circ \B{\Gamma}_{X'}$};
\node at (0,-9) {};
\end{tikzpicture}
\end{center}
It is now clear that, for $X\in \mathrm{Inv}_{\I}(A)$ with inverse $Y$, one has $\B{\Gamma}_{X}^{-1} = \B{\Gamma}_{Y}$. The particular statement is clear, and this finishes the proof.
\end{proof}

\subsection{Central monoid with a left internal hom functor}\label{ssec:central}
Let $(A, m , u)$ be a monoid in $\cat{M}$ we denote by $A^e$ the monoid $A\tensor{}A^o$ where $A^o$ is the opposite monoid (i.e., $A$ with the multiplication morphism twisted by the symmetry $\B{\tau}$). Since our base category $\cat{M}$ is symmetric, we can as in the classical case, identify the category of $A$-bimodules with the category of left (or right) $A^{e}$-modules.
Assume that the functor $A\tensor{}-:\cat{M} \to \cat{M}$ has a right adjoint functor, which we denote by $[A,-]: \cat{M} \to \cat{M}$. In this case, we say that \emph{$A$ has a left internal hom functor}.  Consider as in  Appendix \ref{ssec:appendix2} the functor  ${}_{A^e}[A,-]: {}_{A^e}\cat{M} \to \cat{M}$ which is the right adjoint of  $A\tensor{}-: \cat{M} \to {}_{A^e}\cat{M}$ with the canonical natural monomorphism $ {}_{A^e}[A,-] \hookrightarrow [A,\mathscr{O}(-)]$, where $\mathscr{O}: {}_{A^e}\cat{M} \to \cat{M}$ is the forgetful functor. On the other hand, notice that $\mathscr{O}$ is faithful and exact.

The monoid $A$   is said to be  \emph{central} provided that the canonical map $\I \to {}_{A^e}[A,A]$ is an isomorphism (this is the counit at $\I$ of the previous adjunction).

\begin{proposition}\label{prop:centers}
Let $(A, m, u)$ be a monoid  in $\cat{M}$  and consider the submonoid ${}_1J_1$ of  equation \eqref{Eq:J} with structure given by Proposition \ref {lema:-1}(iii).  Assume that $A$ has a left internal hom functor. Then
\begin{enumerate}[(i)]
\item there is a commutative diagram of monoids:
$$
\xymatrix@C=40pt{  {}_{A^e}[A,A] \ar@{^(->}[r]  \ar@{->}_-{\lambda}[rd] & A \\ & {}_1J_1 \ar@{_(->}[u]  }
$$
\item Assume that $\I \cong {}_1J_1$ via the unit given in \eqref{Eq:triangle} (thus $u$ is a monomorphism). Then $ {}_{A^e}[A,A] \cong \I$ and so $A$ is central.

\item If $A$ is central and the functor $\cat{Z}$ reflects isomorphisms, then $\I \cong {}_1J_1$ via its unit.
\item When $\cat{Z}$ is faithful (i.e.~ $\I$ is a generator), then $A$ is central if and only if $\,\I \cong {}_1J_1$.
\end{enumerate}
\end{proposition}
\begin{proof}
$(i)$. By \cite[Proposition 3.3.]{Femic}, we can identify ${}_{A^e}[A,A]$ with the equalizer
\begin{equation}\label{Eq:center}
\xymatrix@C=40pt{ 0 \ar@{->}[r] & {}_{A^e}[A,A] \ar@{->}[r] &  A  \ar@<.5ex>[rr]^{[A,m] \circ \boldsymbol{\zeta}^A_A} \ar@<-.5ex>[rr]_{[A,m\circ \B{\tau}] \circ \boldsymbol{\zeta}^A_A } &  & [A,A],  }
\end{equation}
where  ${\boldsymbol{\zeta}^A}$ is the unit of the adjunction $A\tensor{}- \dashv [A,-]$.
Now, using the universal property of ${}_1J_1$ given by equation \eqref{eq:equalizer}, one shows the existence of the triangle. We leave to the reader to check that the stated diagram is in fact  a triangle  of monoids. \\
$(ii)$. It follows from the fact that the diagram in $(i)$ commutes and that the morphisms involved are unitary.\\
$(iii)$.  The definition of the functor ${}_{A^e}[A,-]$ obviously  implies the isomorphism $\Z{{}_{A^e}[A,A]} \cong \lhom{A^e}{A}{A}$ of $\Z{\I}$-algebras. On the other hand, under the assumption made on the functor $\cat{Z}$  a tedious computation gives a chase to an isomorphism $\Z{{}_1J_1} \cong  \lhom{A^e}{A}{A}$ of $\Z{\I}$-algebras. Combining these two  isomorphisms  shows that $\cat{Z}(\lambda)$ is an isomorphism. Therefore, $ {}_1J_1\cong {}_{A^e}[A,A]  \cong \I$, since $A$ is central.\\
$(iv)$. We only need to check the direct implication since the converse follows by the second item. This implication clearly follows from part $(iii)$, since a  faithful functor, whose domain is an abelian category, reflects isomorphisms.
\end{proof}

\subsection{The Azumaya case}\label{ssec:azumaya}
The notion of Azumaya monoid  in a monoidal category of modules over a commutative ring (see  \cite[Th\'eor\`eme 5.1]{Knus-Ojanguren} for equivalent definitions), can be extended to any  closed symmetric  monoidal additive category, see for example  \cite{Vitale,Fisher-Palmquist}. However, one can droop the additivity and the closeness conditions on the base monoidal category, as was done in \cite{PareigisIV}, and the definition of Azumaya monoid still makes sense in this context. In our setting $\cat{M}$ is a symmetric abelian category which is possibly not closed, so we can follow  the ideas of \cite{Vitale,PareigisIV}.

\begin{definition}\label{def:azumaya}{\cite{Vitale,PareigisIV,Fisher-Palmquist}}
A monoid $(A,m,u)$ in the monoidal category $\cat{M}$ is called an \emph{Azumaya monoid} if the functor $A\tensor{}-:\cat{M} \to  {}_{A^e}\cat{M}$ establishes an equivalence of categories and the multiplication $m$ splits as a morphism of $A$-bimodules (in other words $A$ is a \emph{separable monoid}). This,  in fact, is the notion of $2$-Azumaya in the sense of \cite{PareigisIV}.
\end{definition}

It is convenient to make some comments and remarks on the conditions of the previous definition.
\begin{remarks}\label{rem:azumaya} Consider  an Azumaya monoid $(A,m,u)$ in $\cat{M}$ as in Definition \ref{def:azumaya}.
\begin{enumerate}[(i)]
\item  It is noteworthy to mention that one can not expect to have for free that $A$ is a dualizable object in  $\cat{M}$, that is dualizable as in Definition \ref{def:dualizable} with the trivial structure of $(\I,\I)$-bimodule. This perhaps happens when $\I$ is a small generator in $\cat{M}$. However, as was show in \cite[Proposition  1.2]{Vitale}, $A$ is in fact a dualizable object in a certain monoidal category of monoids whose dual object is exactly  the opposite monoid $A^o$. On the other hand, one can show as follows that $A$ forms  part  of  a two-sided dualizable datum in the sense of  Definition \ref{def:dualizable} by taking $\I$ and $A^e$ as base monoids.  Following the proof of \cite[Proposition 1.2]{Vitale} since $A$ can be considered either as an $(\I,A^e)$-bimodule or as an $(A^e,\I)$-bimodule,  $A\tensor{}-:\cat{M} \to  {}_{A^e}\cat{M}$ is an equivalence of categories if and only if $A\tensor{A^e}-:  {}_{A^e}\cat{M} \to \cat{M}$ is so. In this way, by applying twice \cite[Proposition 5.1]{
PareigisIII} or \cite[Proposition 1.3]{Vitale}, we obtain the existence of two objects: $[A,\I]$ and $[A,A]$ together with the following  natural isomorphisms $\hom{\cat{M}}{A\tensor{}-}{\I} \cong \hom{\cat{M}}{-}{[A,\I]}$ and $ \hom{\cat{M}}{A\tensor{}-}{A} \cong \hom{\cat{M}}{-}{[A,A]}$. Furthermore, there are isomorphisms: $[A,\I]\tensor{
}A \cong A^e$ of $A^e$-bimodules, $\I\cong A\tensor{A^e}[A,\I]$ of objects in $\cat{M}$ and  $A^e=A\tensor{}A^o \cong [A,A]$ of monoids.
Thus,  the pair $(A, [A,\I])$, within these two first isomorphisms, is a two-sided dualizable datum relating  the monoids $\I$ and $A^e$ in the sense of Definition \ref{def:dualizable}, as claimed above.

\item Since $A\tensor{}-: \cat{M} \to  {}_{A^e}\cat{M}$ is an equivalence of categories, its has a right adjoint functor which, as in Appendix \ref{ssec:appendix2},  is denoted by ${}_{A^e}[A,-]: {}_{A^e}\cat{M} \to \cat{M}$. In contrast with  Proposition  \ref{prop:appII}, here we only know that ${}_{A^e}[A,-]$ exists and there are no indications about its construction. Indeed, at this level of generality, it is not clear whether $A$  has a
left internal hom functor in $\cat{M}$. Thus a very interesting class of Azumaya monoids consists of those for which the underlying object has this property. This is, of course,  the case of the usual class of  Azumaya algebras over commutative rings.
\end{enumerate}
\end{remarks}

\begin{remark}\label{rem:separable}
Let $(A,m,u)$ be a monoid in $\cat{M}$. In Definition \ref{def:azumaya} we can not drop the separability condition  in general as it happens in the classical case. This is due to the fact that  $\I$ is not always projective in $\cat{M}$. However, if we assume that $\I$ is projective in $\cat{M}$   and $A\tensor{}-: \cat{M} \to {}_{A^e}\cat{M}$ is an equivalence, then $A$ must be an Azumaya monoid. Indeed,  since ${}_{A^e}[A,m]:{}_{A^e}[A,A^e]\to{}_{A^e}[A,A]\cong \I$  is an epimorphism being the image of $m$ by the equivalence ${}_{A^e}[A,-]$, then ${}_{A^e}[A,m]$ must split as $\I$ is projective. Therefore,  $m$ itself must split in ${}_{A^e}\cat{M}$, that is,  $A$ should be separable.

On the other hand,  observe that  $A$ is separable  if and only if  it is projective relatively to all morphisms in $_{A^e}\cat{M}$ which split as morphisms in $\cat{M}$, see e.g. \cite[Theorem 1.30 and equality (5)]{AMS-Hoch}. Since this does not mean that $A$ is projective in $_{A^e}\cat{M}$, then $\I$ is not projective in $\cat{M}$ provided that $A\tensor{}-: \cat{M} \to {}_{A^e}\cat{M}$ is an equivalence. Hence if $A$ is Azumaya,  we cannot conclude that $\I$ is projective in $\cat{M}$.

\end{remark}

Before giving more consequences of Definition \ref{def:azumaya}, we give here  conditions under which $\Z{A}$ is an Azumaya $\Z{\I}$-algebra.  Recall from \cite[Proposition 1.2]{May et all:01} that an object $X$ in $\cat{M}$ is said to be a \emph{K\"unneth object} if it is a direct summand of a finite product of copies of $\I$. In the notation of subsection \ref{ssec:iso} this means that $X | \I$ in $\cat{M}$.  In case $\cat{M}$ is closed, any K\"unneth object is obviously a dualizable object cf. \cite{May et all:01}.

\begin{proposition}\label{prop:Kunneth}
Let $(A, m,u)$ be an Azumaya monoid in $\cat{M}$ with  underlying dualizable K\"unneth object $A$.  Then $\Z{A}$ is an Azumaya $\Z{\I}$-algebra in the classical sense.
\end{proposition}
\begin{proof}
It is clear that $\Z{A}$ is a finitely generated and projective $\Z{\I}$-module. Using the characterization of Azumaya algebras (see for example  \cite{DeMeyer-Ingraham, Kadison-NewExamples}), we need to check that ${\rm End}_{\scriptscriptstyle{\Z{\I}}}(\Z{A})$ is isomorphic as a $\Z{\I}$-algebra to $\Z{A}\tensor{\Z{\I}}\Z{A}^o$, which is proved as follows. Following the ideas of \cite[Proposition 1.2]{May et all:01}, for any object $X$ in $\cat{M}$, we know  that $\Z{X}\tensor{\Z{\I}}\Z{\I^n} \,\cong\,\Z{X\tensor{}\I^n}$.  Thus, the same isomorphism is inherited  by any direct summands of some $\I^n$.  Therefore, we have an isomorphism  $\Z{X}\tensor{\Z{\I}}\Z{A} \,\cong\,\Z{X\tensor{}A}$ and by symmetry $\Z{A}\tensor{\Z{\I}}\Z{X} \,\cong\,\Z{A\tensor{}X}$, for any object $X$ in $\cat{M}$.
In particular, we have
$$ \Z{A}\tensor{\Z{\I}}\Z{A} \,\cong\,\Z{A\tensor{}A}, \,\text{ and }\, \Z{A}\tensor{\Z{\I}}\Z{[A,\I]} \,\cong\,\Z{A\tensor{}[A,\I]}.$$
Therefore, $\Z{[A,\I]}\cong \Z{A}^*$ as $\Z{\I}$-modules, where $\Z{A}^*$ is the $\Z{\I}$-linear dual of $\Z{A}$, since $[A,\I]$ is a dual object of $A$ in $\cat{M}$. By Remark \ref{rem:azumaya}(i), we have a chain of isomorphisms $A^e \cong [A,A] \cong A\tensor{}[A,\I]$, from which  we deduce the following isomorphisms
$$\Z{A}\tensor{\Z{\I}}\Z{A}^o\,\cong\, \Z{A\tensor{}A^o} \,\cong \,\Z{[A,A]} \,\cong\, \Z{A}\tensor{\Z{\I}}\Z{A}^* \cong {\rm End}_{\scriptscriptstyle{\Z{\I}}}(\Z{A}),$$
whose composition leads to the desired $\Z{\I}$-algebra isomorphism.
\end{proof}

In view of Remark \ref{rem:azumaya}(ii), \emph{in what follows we only  consider  an Azumaya monoid $(A, m,u)$ for which the  underlying object $A$ has a left internal hom functor in $\cat{M}$}.  This is the case, for instance,  when $A$ is a left dualizable object in $\cat{M}$ (in the sense of Definition \ref{def:dualizable} with $R=S=\I$).  Obviously, the strong assumption of $\cat{M}$ being left closed automatically  guaranties the existence of left internal hom functor for any object; however, this is not the case of our interest.

\begin{corollary}\label{coro:azumaya}
Let  $(A,m,u)$ be an Azumaya monoid in $\cat{M}$. Then $A$ is flat, central and the unit $u: \I \to A$ is a section in $\cat{M}$.
\end{corollary}
\begin{proof}
The functor $A\tensor{}-: \cat{M} \to \cat{M}$ is clearly the composition of the functor $A\tensor{}-: \cat{M} \to {}_{A^e}\cat{M}$ with the forgetful functor $\mathscr{O}: {}_{A^e}\cat{M} \to \cat{M}$, so that it is left exact. Thus $A$ is a flat object.

We know that the  unit and the counit
$$ \B{\varepsilon}_M: A\tensor{}{}_{A^e}[A,M] \longrightarrow M, \quad
\B{\zeta}_X: X \longrightarrow {}_{A^e}[A,A\tensor{}X],$$
of the adjunction  $\xymatrix{ A\tensor{}-: \cat{M} \ar@<.5ex>[r]  &  \ar@<.5ex>[l] {}_{A^e}\cat{M}: {}_{A^e}[A,-]}$
are natural isomorphisms, for every pair of objects $(X,M)$ in $\cat{M} \times {}_{A^e}\cat{M}$. Thus, $\B{\zeta}_{\I}: \I \to {}_{A^e}[A,A]$ is an isomorphism, and so $A$ is central.  The retraction of $u$ is given by the following dashed  arrow
$$\xymatrix{  {}_{A^e}[A,A\tensor{}A]  \ar@{->}^-{{}_{A^e}[A,m]}[rr] & & {}_{A^e}[A,A] \cong \I  \\ A \ar@{->}^-{\B{\zeta}_A}[u] \ar@{-->}_-{}[urr]& &  }  $$
\end{proof}

Next we want to apply the results of section \ref{sec:azymaya}, specially Theorem \ref{thm:iso}, to the case of Azumaya monoid. Most of the assumptions in that Theorem are in fact fulfilled for an  Azumaya monoid $(A,m,u)$. Indeed, from the fact that $m$ splits in ${}_{A^e}\cat{M}$,  we have that  $A|(A\tensor{}A)$ as in Definition \ref{def:divide}(2). By Proposition \ref{prop:centers} and Corollary \ref{coro:azumaya}, we know that $u_1:\I \to {}_1J_1$ is an isomorphism and that $A$ is flat. Henceforth, the only condition on $A$, which  one needs to check is $(A\tensor{}A)|A$.

However, even under the assumption that $\cat{Z}$ is faithful, one can not expect to have for free this last condition, as  was given  in  the classical case of modules over a commutative ring. In our setting, it seems that this condition depends heavily on the fact that  $A$ and $A^e$ should be ``progenerators'' in the category of  $A^e$-bimodules.   To be  more precise,  as was argued in Remark \ref{rem:separable}, the  projectivity of $A$ in ${}_{A^e}\cat{M}$ is, for instance, linked to that of $\I$ in $\cat{M}$ as the following natural  isomorphism  shows: $\lhom{A^e}{A}{-}\cong \cat{Z} \circ {}_{A^e}[A,-]$, see Appendix \ref{ssec:appendix2}.

Recall from Proposition \ref{pro:fX} that, for any element $\theta  \in \cat{G}_{\scriptscriptstyle{\Z{\I}}}(\Z{A})$ the group defined in \eqref{Eq:G}, there is, by Propositions \ref{pro:invdual} and \ref{pro:fX},  an isomorphism of left $A$-modules $f_{{}_{\theta}J_1}:A\tensor{}{}_{\theta}J_1 \to A$.  Considering $A\tensor{}{}_{\theta}J_1$ as left $A^{e}$-module, we also  have
${}_{A^e}[A, A\tensor{}{}_{\theta}J_1] \cong {}_{\theta}J_1$ via the  unit $\B{\zeta}$.

For an element $\theta$ as above, we denote by $\td{\theta}\,:=\,  \B{\Gamma}_{{}_{\theta}J_1}$ the image of ${}_{\theta}J_1$ by  the morphism of groups $\B{\Gamma}$  stated in Proposition \ref{prop:Gamma} (recall here that $i_{{}_{\theta}J_1}=\fk{eq}_{\theta,1}$, see \eqref{Eq:strmaps}). In this way, to each $A$-bimodule $M$, we associate the $A$-bimodule $M_{\td{\theta}}$ whose underlying object is $M$  where the left action is unchanged while the right action is twisted by $\td{\theta}$. Precisely,  we have $\rho_{M_{\td{\theta}}}: =\rho_M \circ (M\tensor{}\td{\theta})$, where $\rho_M: M\tensor{}A \to M$ is the right structure morphism of $M$.

Now, given another element $\sigma \in \cat{G}_{\scriptscriptstyle{\Z{\I}}}(\Z{A})$ and another $A$-bimodule $N$, we have two $\Z{\I}$-modules under consideration. Namely, the first one is $\cat{M}_{A,\Z{A}} (M_{\theta}, N_{\sigma})$ defined in the same way as in Subsection \ref{ssec:iso}, and the other is the module of $A$-bimodules morphisms  $\hom{A,A}{M_{\td{\theta}}}{N_{\td{\sigma}}}$.

\begin{proposition}\label{prop:Azumaya1}
Let $(A,m,u)$ be an Azumaya monoid in $\cat{M}$. Consider elements $\theta, \sigma \in \cat{G}_{\scriptscriptstyle{\Z{\I}}}(\Z{A})$ and their respective associated  images $\td{\theta}, \td{\sigma} \in \mathrm{Aut}_{alg}(A)$. Assume that the functor $\cat{Z}$ is faithful. Then
\begin{enumerate}[(i)]
\item For every $t \in \Z{A}$, we have $ \Z{\td{\theta} } (t)\,=\, {\Phi}_{{}_{\theta}J_1}(t)\,=\, \theta(t)$, that is,
 $\Z{\td{\theta}}\,=\, \theta$.
\item The homomorphism of groups $\omega$ stated in Corollary \ref{coro:Jinv} is surjective.
\item  There is an equality $\hom{A,A}{A_{\td{\theta}}}{A_{\td{\sigma}}}  \,=\,   \cat{M}_{A,\Z{A}} (A_{\theta}, A_{\sigma})$.
\end{enumerate}
\end{proposition}
\begin{proof} By Corollary \ref{coro:azumaya}, we know that $A$ is central, hence $\I \cong {}_1J_1$ by Proposition \ref{prop:centers}(iv). \\
$(i)$.  It is a direct consequence of Proposition \ref{prop:Gamma} and Theorem \ref{thm:mono} (see equation \eqref{Eq:phitheta}). \\ $(ii)$.  It follows by item $(i)$.
\\$(iii)$. The direct inclusion follows as in Proposition \ref{pro:div}(1).
Conversely, take an element $f \in  \cat{M}_{A,\Z{A}} (A_{\theta}, A_{\sigma})$. This element belongs to $\hom{A,A}{A_{\td{\theta}}}{A_{\td{\sigma}}}$ if and only if
\begin{equation}\label{Eq:seraloqsera2}
m \circ (f\tensor{}\td{\sigma}) \,= \, f \circ m \circ (A\tensor{}\td{\theta}).
\end{equation}
By the equalities $m \circ (f\tensor{}\sigma(t)) \,= \, f \circ m \circ (A\tensor{}\theta(t))$, $t \in \Z{A}$ derived from the definition of $f$,  we have, using part $(i)$, that
$$
m \circ (f\tensor{}\td{\sigma}) \circ (A\tensor{}t) \,= \, f \circ m \circ (A\tensor{}\td{\theta}) \circ (A\tensor{}t),
$$
for every $t \in \Z{A}$. Using these equalities and the fact that $\I$ is a generator we are able to show that $m \circ (f\tensor{}\td{\sigma}) \circ (A\tensor{}  \pi_A)\,=\, f \circ m \circ (A\tensor{}\td{\theta}) \circ (A\tensor{}  \pi_A)$, where $\pi_A: \I^{(\Lambda)} \to A$ is the canonical epimorphism. Now, equality \eqref{Eq:seraloqsera2} follows since $A\tensor{}\pi_A$ is an epimorphism, which completes the proof.
\end{proof}

Now, with notations as in Definition  \ref{def:divide},  we set
$$
\cat{H}_{\scriptscriptstyle{\Z{\I}}}(\Z{A}) \,:=\, \left\{\underset{}{}  \theta \in \mathrm{Aut}_{\scriptscriptstyle{\Z{\I}}\text{-}alg}(\Z{A})| \,\, A_{\theta} \sim A_1 \right\}.
$$
Using Proposition \ref{pro:div}(7), we easily check that this is a subgroup of $\mathrm{Aut}_{\scriptscriptstyle{\Z{\I}}\text{-}alg}(\Z{A})$. Now, it is clear from Corollary  \ref{coro:tau 1}, that under the assumptions of $A$ being Azumaya and $u_1: \I \to {}_1J_1$ is an isomorphism, we   have an inclusion $\cat{H}_{\scriptscriptstyle{\Z{\I}}}(\Z{A}) \subseteq \cat{G}_{\scriptscriptstyle{\Z{\I}}}(\Z{A})$ of groups.

\begin{corollary}\label{thm:isoAzumaya}
Let $(A,m,u)$ be an Azumaya monoid in $\cat{M}$ such that $A\tensor{}A | A$.  Assume that the functor  $\cat{Z}$ is faithful. Then the maps  $\omega$ and $\B{\Gamma}$ are bijective so that we have the following commutative diagram
\begin{small}
$$
\xymatrix@C=40pt{ & \mathrm{Aut}_{alg}(A)\ar@/_1,5pc/_<<<<<<{\omega}^<<<<<<{\cong}[dd]  \ar@{^{(}->}^-{\Omega}[dr]  & \\ \Inv{\I}{A} \ar@{->}^-{\B{\Gamma}}_-{\cong}[ru]  \ar@{->}^>>>>>>>>>>>>>>>>{\Phi}[rr]|(.385)\hole \ar@{->}_-{\cong}^{\widehat{\Phi}}[rd] &  & \Aut{\scriptscriptstyle{\Z{\I}}\text{-alg}}{\Z{A}}  \\  & \cat{G}_{\scriptscriptstyle{\Z{\I}}}(\Z{A}) = \cat{H}_{\scriptscriptstyle{\Z{\I}}}(\Z{A})  \ar@{^(->}_-{\B{\varsigma}}[ur] &  }
$$
\end{small}
of homomorphisms of groups.
\end{corollary}
\begin{proof}
The map $\omega$ is bijective since by Proposition \ref{prop:Azumaya1}(ii) it is surjective and it is injective as $\cat{Z}$ is faithful. We know from Proposition \ref{prop:Gamma}, that $\Omega \circ \B{\Gamma} = {\Phi}$. Therefore,  Theorem \ref{thm:mono} implies that $\B{\varsigma} \circ \omega \circ \B{\Gamma} = \B{\varsigma}\circ \widehat{\Phi}$, and so $\omega \circ \B{\Gamma} = \widehat{\Phi}$ which implies that  $\B{\Gamma} $ is  bijective as well. Lastly, the  inclusion $\cat{G}_{\scriptscriptstyle{\Z{\I}}}(\Z{A}) \subseteq \cat{H}_{\scriptscriptstyle{\Z{\I}}}(\Z{A})$  is deduced from  Proposition \ref{prop:Azumaya1}(i) in combination with Theorem \ref{teo:azum}.
\end{proof}

\section{Application to the category of comodules over a flat Hopf algebroid}\label{sec:application}

All Hopf algebroids which will be considered here are commutative and flat over the base ring, for the axiomatic  definitions and basic properties, we refer the reader to \cite[Appendix 1]{Ravenel:1986}.

Let $(R, H)$ be a commutative Hopf  algebroid with base ring $R$ and  structure maps  $ s,t:R\to H$, $\varepsilon:H\to R$,  $\Delta:H\to H\tensor{R} H$ and $\mathscr{S}:H\to H$. An  $H$-comodule stands for  right $H$-comodule, we denote the category of $H$-comodules by $\rcomod{H}$. It is well-known, see for instance \cite{Bruguieres:1994} or \cite{Hovey:2004},   that any  $H$-comodule $P$ whose underlying  $R$-module  is finitely generated and projective is a dualizable object in the monoidal category of comodules with dual the $R$-module $P^*=\hom{R}{P}{R}$. The comodule structure of $P^*$ is given by
$$
P^* \to P^*\tensor{R}H, \quad \LR{ \varphi \longmapsto e_i^*\tensor{R}t(\varphi(e_{i,0}))\mathscr{S}(e_{i,1}) },
$$
where $\{e_i,e_i^*\}$ is a dual basis for $P_R$ and $\varrho_P(p)=p_0\tensor{R}p_1$ is the $H$-coaction of $P$ (the summation is understood). Notice that the converse also holds true which means that any dualizable object in $\rcomod{H}$ is a finitely generated and projective $R$-module. This is due to the fact  that the forgetful functor $\rcomod{H} \to\rmod{R}$ is a strict monoidal  functor and the unit object in $\rcomod{H}$ is $R[1]$, that is $R$ with structure of comodule  given by the grouplike element $1_H$, via the target map $\mathsf{t}: R \to H \cong R\tensor{R}H$.

Next we want to apply the results of Subsection \ref{ssec:azumaya} to the category of comodules $\rcomod{H}$. Observe that this category is a symmetric monoidal Grothendieck category with respect to the canonical flip over $R$, where the tensor product is right exact on both factors. Moreover, the forgetful functor $\rcomod{H} \to \rmod{R}$ is faithful and exact. So the category ${\rcomod{H}}$ fits in the context of that subsection.

In the previous notations, the functor $\cat{Z}$ is identified with $\Z{M}=M^{\mathrm{co}(H)}$,  for an  $H$-comodule $M$, where $$M^{\mathrm{co}(H)}:=\{ m \in M|\, \varrho_M(m)=m\tensor{R}1\} $$ is the submodule of coinvariant elements. Therefore, the condition that $R[1]$ is a generator in $\rcomod{H}$, means that the sudmodule of coinvariant elements $M^{\mathrm{co}(H)}$ is not zero, for every right $H$-comodule $M$.  For simplicity we denote by $R^{\mathrm{co}(H)}:=(R[1])^{\mathrm{co}(H)}$ the subalgebra of $R$ of coinvariant elements, which is explicitly given by $R^{\mathrm{co}(H)}=\{ r \in R| s(r)=t(r)\}$.

\begin{lemma}
Let $A$ be an $H$-comodule with coaction $\varrho_A: A \to A\tensor{R}H$ . Then $A$ admits a structure of monoid in $\rcomod{H}$ if and only if  $\varrho_A$ is a morphism of $R$-algebras, where $H$ is considered as an $R$-algebra via its source map $s$.
\end{lemma}
\begin{proof}
Straightforward.
\end{proof}
We refer to such an object as an \emph{$H$-comodule $R$-algebra}.

\begin{corollary}\label{coro:fgp}
Let $A$ be an  Azumaya $H$-comodule $R$-algebra.
\begin{enumerate}[(i)]
\item If $A$ is finitely generated and projective $R$-module, then $A$ is an Azumaya  $R$-algebra.
\item If the underlying comodule of $A$ is a direct summand of finite products of copies of $R[1]$, then  $A^{\mathrm{co}(H)}$ is an Azumaya $R^{\mathrm{co}(H)}$-algebra.
\end{enumerate}
\end{corollary}
\begin{proof}
$(i)$. Since $A$ is a dualizable $H$-comodule, the right adjoint of the functor $A\tensor{}-:\rcomod{H} \to \rcomod{H}$ is given by the functor $[A,-]\cong -\tensor{R}A^*$ defined using the tensor product of two $H$-comodules. Using this adjunction and equation \eqref{Eq:center}, we can show that the underlying  $R$-algebra of the $H$-comodule algebra ${}_{A^e}[A,A]$  coincides with the centre of the underlying $R$-algebra of $A$. Therefore,  the centre of $A$ coincides with $R \cong {}_{A^e}[A,A]$  as $A$ is a central $H$-comodule $R$-algebra.  From this we conclude that $A$ is a central separable $R$-algebra, that is, an Azumaya $R$-algebra.
\\ $(ii)$. It follows directly from Proposition \ref{prop:Kunneth}, since $A$ is a dualizable comodule.
\end{proof}

\begin{example}\label{exam:H}
Assume that $R[1]$ is a projective $H$-comodule and take $P$ an $H$-comodule such that $P_R$ is finitely generated and projective module. Consider in a canonical way $\mathrm{End}_R(P)\cong [P,P]$ as an $H$-comodule $R$-algebra. Assume that the evaluation map $P\tensor{\mathrm{End}_R(P)}P^* \to R[1]$ is an isomorphism of $H$-comodules. Since $R[1]$ is projective this isomorphism implies that $P$ is a progenerator in $\rcomod{H}$ in the sense of \cite[page 113]{PareigisIV} whence, by \cite[Theorem 14]{PareigisIV}, $\mathrm{End}_R(P)$ is an Azumaya $H$-comodule $R$-algebra.
\end{example}

\begin{example}
Assume that $(R,H)$ is a split Hopf algebroid, that is  $H=R\tensor{\mathbb{K}}B$, where $B$ is a flat commutative Hopf algebra over a ground commutative ring $\mathbb{K}$ and $R$ is a (right) $B$-comodule commutative  $\mathbb{K}$-algebra.  Let $A$ be an Azumaya $H$-comodule $R$-algebra which finitely generated and projective as an $R$-module. Then, by Corollary \ref{coro:fgp}(i), $A$ is an Azumaya $R$-algebra. If $R=\mathbb{K}$, then $A$ is in particular an $B$-comodule Azumaya algebra in the sense of \cite[page 328]{Caenepeel:book}.
\end{example}

The fact that $A$ is an $H$-comodule $R$-algebra leads to different groups so far treated here.  On the one hand, we have $\mathrm{Inv}_{R}(A)$ and $\mathrm{Aut}_{R\text{-}alg}(A)$, where $A$ is considered as an $R$-algebra i.e. a monoid in the category of $R$-modules. On the other hand, we have $\mathrm{Inv}_{R}^H(A)$ and $\mathrm{Aut}_{R\text{-}alg}^H(A)$, where $A$ is considered as a monoid in the category of $H$-comodules. Obviously we have the following  inclusions of groups
$$
\mathrm{Inv}_{R}^H(A) \subseteq \mathrm{Inv}_{R}(A), \quad  \mathrm{Aut}_{R\text{-}alg}^H(A) \subseteq \mathrm{Aut}_{R\text{-}alg}(A).
$$

By a direct application of Corollary \ref{thm:isoAzumaya}, we obtain the following.
\begin{corollary}\label{coro:Halgd}
Let $A$ be an  Azumaya $H$-comodule $R$-algebra such that $A\tensor{R}A|A$ simultaneously in the category of $H$-comodules and of $A$-bimodules.  Assume that $R[1]$ is a generator in $\rcomod{H}$. Then  there is an isomorphism of groups
$$\mathrm{Inv}_{R}^H(A) \, \cong  \, \mathrm{Aut}_{R\text{-}alg}^H(A).$$
\end{corollary}

\begin{remark}\label{rem:final}
Consider an Azumaya $H$-comodule $R$-algebra $A$ which is finitely generated and projective  as $R$-module. One can expect to deduce Corollary \ref{coro:Halgd} directly by using Corollary \ref{coro:fgp}(i) in conjunction with \cite[Corollary of Theorem 1.4]{Mi}. This could be  so simple if  one succeeds to show, for instance, that the following diagram is commutative
$$
\xymatrix{  \mathrm{Inv}_{R}^H(A) \ar@{_{(}->}[d]  \ar@{->}^-{\B{\psi}}[r] &  \mathrm{Aut}_{R\text{-}alg}^H(A) \ar@{^{(}->}[d] \\
\mathrm{Inv}_{R}(A) \ar@{->}^-{\B{\varphi}}[r]  &  \mathrm{Aut}_{R\text{-}alg}(A) }
$$
where the map $\B{\psi}$ is the isomorphism of Corollary \ref{coro:Halgd} while $\B{\varphi}$ is the isomorphism of \cite[page 100]{Mi} which for any element $X \in \mathrm{Inv}_{R}(A)$ with inverse $Y$ and decomposition of unit  $1_R=\sum_ix_iy_i$ ($x_i \in X, y_i \in Y$), the associated automorphism  is given by $\B{\varphi}_X(a)=\sum_i x_iay_i$, for any $a \in A$.   That is, to show that the map $\B{\psi}$ is the restriction of $\B{\varphi}$. However this is not  clear at all.  Or perhaps by showing that the map $\B{\varphi}_X$ is $H$-colinear whenever $X$ belongs to the subgroup
$\mathrm{Inv}_{R}^H(A)$. This is also not clear at all.  In any case, both ways  will only lead to the injectivity and one has to check the surjectivity which is perhaps much more complicated by using  elementary methods.
\end{remark}

\appendix

\section{More results on invertible and dualizable bimodules}\label{ssec:appendix1}
In all this appendix  $(\cat{M},\tensor{}, \I,l,r)$ will be a Penrose  monoidal abelian, locally small and bicomplete category, where tensor products are right exact on both factors.  Let $(A,m_A,u_A),  (R,m_R,u_R), (S,m_S,u_S)$ be monoids in $\cat{M}$ with two morphisms of monoids $\alpha: R \rightarrow A \leftarrow S: \beta$. In the sequel, we will use the notation of Section  \ref{sec:biobjects}.

\subsection{From right dualizable datum to right inverse}\label{ssec:converse}
This subsection is devoted to discuss the converse of Proposition \ref{pro:invdual}, that is, trying to give conditions under which a right invertible sub-bimodule can be extracted from a right dualizable datum in the sense of Definition \ref{def:dualizable}.

The following corollary is complementary to Proposition \ref{pro:fX}.
\begin{corollary}
\label{coro:fX} If in Proposition \ref{pro:fX}, we assume $\mathrm{ev}$ and $%
\mathrm{coev}$ are isomorphisms, then so are the following morphisms%
\begin{equation*}
f_{Y} :=\mar\circ \left( i_{Y}\tensor{R}A\right) :Y\tensor{R}A\rightarrow A, \quad
g_{X} :=\mar\circ \left( A\tensor{R}i_{X}\right) :A\tensor{R}X\rightarrow A.
\end{equation*}
\end{corollary}

\begin{proof}
Since $\left( X,Y,\mathrm{ev},\mathrm{coev}\right) $ is a right
dualizable datum then $\big( Y,X,( \mathrm{coev}) ^{-1},(
\mathrm{ev}) ^{-1}\big) $ is a right dualizable datum. By
Proposition \ref{pro:fX}, we get that $f_{Y}$ and $g_{X}$ are isomorphisms
too.
\end{proof}

\begin{proposition}
Consider a right dualizable datum $\left( X,Y,\mathrm{ev},\mathrm{coev}\right) $ which satisfies equations \eqref{form:ev3} and \eqref{form:ev4}.  Assume that $g_X$ is an isomorphism and  the functor  $A\tensor{R}(-): {}_R\cat{M} \to {}_A\cat{M}$ is faithful. Then $X$ is a two-sided invertible sub-bimodule.
\end{proposition}
\begin{proof}
By Proposition \ref{pro:fX}, $g_Y$ is an isomorphism with inverse
\begin{equation*}
g_{Y}^{-1} =\left( \mar\tensor{S}Y\right) \circ \left( A\tensor{R}i_{X}\tensor{S}Y\right) \circ \left( A\tensor{R}\mathrm{coev}\right)
\circ \left( \rar\right) ^{-1}
=\left( g_{X}\tensor{S}Y\right) \circ \left( A\tensor{R}\mathrm{coev}%
\right) \circ \left( \rar\right) ^{-1},
\end{equation*}
Therefore, we get that
\begin{equation*}
A\tensor{R}\mathrm{coev}=\left( \left( g_{X}\right) ^{-1}\tensor{S}Y\right) \circ g_{Y}^{-1}\circ \rar
\end{equation*}%
is an isomorphism. Hence, since the functor $A\tensor{R}\left( -\right) $ is
faithful, then $\mathrm{coev}$ is an epimorphism as well as a monomorphism, so that it is an isomorphism.
By (\ref{form:ev1}), we have
\begin{equation*}
\rxs\circ \left( X\tensor{S}\mathrm{ev}\right) \circ \left( \mathrm{%
coev}\tensor{R}X\right) \circ \left( \lxr\right) ^{-1}=\mathrm{Id}%
_{X}.
\end{equation*}%
Since $\mathrm{coev}$ is an isomorphism, we get that $X\tensor{S}\mathrm{%
ev}$ is an isomorphism too. Now, from the equality
\begin{equation*}
\left( g_{X}\tensor{S}S\right) \circ \left( A\tensor{R}X\tensor{S}%
\mathrm{ev}\right) =\left( A\tensor{S}\mathrm{ev}\right) \circ \left(
g_{X}\tensor{S}Y\tensor{R}X\right)
\end{equation*}%
we deduce that $A\tensor{S}\mathrm{ev}$ is an isomorphism. Therefore,   we conclude as above  that $\mathrm{ev}$ is  an isomorphism.
\end{proof}

\subsection{Base change and dualizable datum}\label{ssec:A1}
Keep the notation of Section \ref{sec:biobjects}, and assume that $\alpha, \beta$ are  monomorphisms in $\cat{M}$.
The main aim here is to prove that an $(R,S)$-bimodule $X$ such that $X\tensor{S}A\cong A$ and which fits into a right dualizable datum admits a right inverse if we change the ring $R$ by a suitable extension $R'$ of it. First we need to prove the following technical result.

\begin{proposition} \label{pro:tau}
Let $(X,Y, {\rm ev}, {\rm coev})$ be a right dualizable datum (Definition \ref{def:dualizable}) such that $Y \in \mathscr{P}\left({}_{S}A_{R}\right) $ and $X \in \mathscr{P}\left( _{R}A_{S}\right)$. Assume that the
morphism $f_{X}=\mas\circ \left( i_{X}\tensor{S}A\right) :X\tensor{S}A\rightarrow A$ is an isomorphism. Define
\begin{equation*}
\gamma :=\Big( Y\overset{\left( r_{Y}^{R}\right) ^{-1}}{\longrightarrow }%
Y\tensor{R}R\overset{Y\tensor{R}\alpha }{\longrightarrow }Y\tensor{R}A%
\overset{Y\tensor{R}f_{X}^{-1}}{\longrightarrow }Y\tensor{R}X\otimes
_{S}A\overset{\mathrm{ev}\tensor{S}A}{\longrightarrow }S\tensor{S}A%
\overset{l_{A}^{S}}{\longrightarrow }A\Big) .
\end{equation*}%
Then
\begin{eqnarray}
f_{X}\circ \left( X\tensor{S}\gamma \right) \circ \mathrm{coev} &=&\alpha
\qquad \text{(i.e. }\mas\circ \left( i_{X}\tensor{S}\gamma \right)
\circ \mathrm{coev}=\alpha \text{),}  \label{Eq:a} \\
\mar\circ \left( \gamma \tensor{R}i_{X}\right) &=&\beta \circ \mathrm{ev%
}.  \label{Eq:b}
\end{eqnarray}
\end{proposition}
\begin{proof}
Using the diagrammatic convention of Section \ref{sec:biobjects}, we have
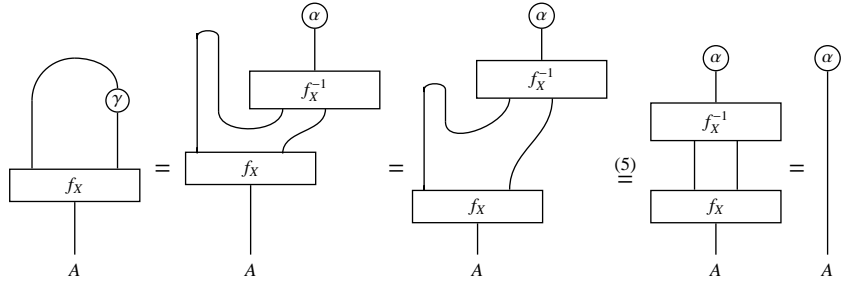
\begin{center}
\begin{tikzpicture}[x=8pt,y=8pt,thick]\pgfsetlinewidth{0.5pt}
\node[circle,draw, inner sep=1pt](1) at (0,1) {$\scriptstyle{\gamma}$}; \node[rectangle,draw,inner sep=3pt, text width=1.5cm, text centered](2) at (-2,-3) {$\scriptstyle{f_X}$};
\node(3) at (-2,-7) {$\scriptstyle{A}$};

\draw[-](1) to [out=90,in=0] (-2,3);
\draw[-] (1) to [out=-90,in=90] (0,-2.2);
\draw[-] (-4,1) to [out=90,in=180] (-2,3);
\draw[-] (-4,1) to [out=-90,in=90] (-4,-2.2);
\draw[-] (3) to [out=90,in=-90] (2);
\end{tikzpicture}
\begin{tikzpicture}[x=8pt,y=8pt,thick]\pgfsetlinewidth{0.5pt}
 \node[inner sep=1pt] at (0,0) {$\,=\,$};
 \node at(0,-5) {};
\end{tikzpicture}
\begin{tikzpicture}[x=8pt,y=8pt,thick]\pgfsetlinewidth{0.5pt} \node[circle,draw, inner sep=1.5pt](1) at (-0.5,5) {$\scriptstyle{\alpha}$};
\node[rectangle,draw, inner sep=3pt, text width=1.5cm, text centered](2) at (-3.5,-2.2) {$\scriptstyle{f_X}$};
\node[rectangle,draw, inner sep=3pt, text width=1.5cm, text centered](3) at (-0.5,1.5) {$\scriptstyle{f_X^{-1}}$};
\node(4) at (-3.5,-7) {$\scriptstyle{A}$};

\draw[-] (1) to [out=-90,in=90] (3);
\draw[-] (-6,-1.2) to [out=-90,in=90] (-6,3.9);
\draw[-] (-6,4) to [out=90,in=90] (-5,4);
\draw[-] (-5,4) to [out=-90,in=90] (-5,0.5);
\draw[-] (-5,0.5) to [out=-90,in=-90](-2,0.6);

\draw[-] (0,0.6) to [out=-90,in=90] (-2,-1.5);
\draw[-] (4) to [out=90,in=-90] (2);
\end{tikzpicture}%
\begin{tikzpicture}[x=8pt,y=8pt,thick]\pgfsetlinewidth{0.5pt}
 \node[inner sep=1pt] at (0,0) {$\,=\,$}; \node at (0,-5) {};
\end{tikzpicture}%
\begin{tikzpicture}[x=8pt,y=8pt,thick]\pgfsetlinewidth{0.5pt} \node[circle,draw, inner sep=1.5pt](1) at (-0.5,5) {$\scriptstyle{\alpha}$};
\node[rectangle,draw, inner sep=3pt, text width=1.5cm, text centered](2) at (-3.5,-4) {$\scriptstyle{f_X}$};
\node[rectangle,draw, inner sep=3pt, text width=1.5cm, text centered](3) at (-0.5,2) {$\scriptstyle{f_X^{-1}}$};
\node(4) at (-3.5,-7) {$\scriptstyle{A}$};

\draw[-] (1) to [out=-90,in=90] (3);
\draw[-] (-6,-3) to [out=-90,in=90] (-6,1.4);
\draw[-] (-6,1.5) to [out=90,in=90] (-5,1.5);
\draw[-] (-5,1.5) to [out=-90,in=90] (-5,0);
\draw[-] (-5,0) to [out=-90,in=-90](-2,1.1);

\draw[-] (0,1.1) to [out=-90,in=90] (-2,-3.2);
\draw[-] (4) to [out=90,in=-90] (2);
\end{tikzpicture}%
\begin{tikzpicture}[x=8pt,y=8pt,thick]\pgfsetlinewidth{0.5pt}
\node[inner sep=1pt] at (0,0) {$\,\overset{\eqref{form:ev1}}{=}\,$};
\node at (0,-5) {};
\end{tikzpicture}
\begin{tikzpicture}[x=8pt,y=8pt,thick]\pgfsetlinewidth{0.5pt}
\node[circle,draw, inner sep=1.5pt](1) at (0,5) {$\scriptstyle{\alpha}$};
\node[rectangle,draw, inner sep=3pt, text width=1.5cm, text centered](2) at (0,-2) {$\scriptstyle{f_X}$};
\node[rectangle,draw, inner sep=3pt, text width=1.5cm, text centered](3) at (0,2) {$\scriptstyle{f_X^{-1}}$};
\node(4) at (0,-5) {$\scriptstyle{A}$};

\draw[-] (1) to [out=-90,in=90] (3);
\draw[-] (-1,1.1) to[out=-90,in=90] (-1,-1.2);
\draw[-] (1,1.1) to[out=-90,in=90] (1,-1.2);
\draw[-] (4) to [out=90,in=-90] (2);
\end{tikzpicture}%
\begin{tikzpicture}[x=8pt,y=8pt,thick]\pgfsetlinewidth{0.5pt}
 \node[inner sep=1pt] at (0,0) {$\,=\,$};
\node at (0,-5) {};
\end{tikzpicture}%
\begin{tikzpicture}[x=8pt,y=8pt,thick]\pgfsetlinewidth{0.5pt}
\node[circle,draw, inner sep=1.5pt](1) at (0,5) {$\scriptstyle{\alpha}$};
\node(4) at (0,-5) {$\scriptstyle{A}$};
\draw[-] (4) to [out=90,in=-90] (1);
\end{tikzpicture}
\end{center}
which gives equality \eqref{Eq:a}. On the other hand, we have
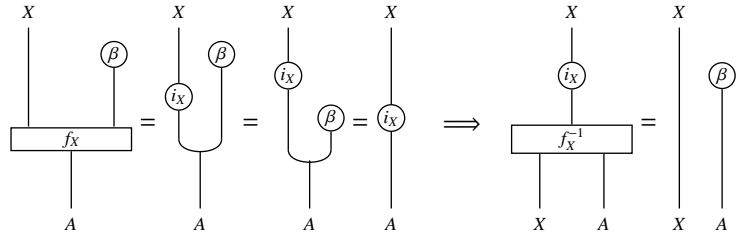
\begin{center}
\begin{tikzpicture}[x=8pt,y=8pt,thick]\pgfsetlinewidth{0.5pt}
\node(1) at(0,5) {$\scriptstyle{X}$};
\node[rectangle,draw, inner sep=1.2pt, text width=1.5cm, text centered](2) at (2,-1) {$\scriptstyle{f_X}$};
\node[circle,draw, inner sep=1pt](3) at (4,3) {$\scriptstyle{\beta}$};
 \node(4) at (2,-5) {$\scriptstyle{A}$};

\draw[-] (1) to [out=-90,in=90] (0,-0.4);
\draw[-] (3) to [out=-90,in=90] (4,-0.4);
\draw[-] (4) to [out=90,in=-90] (2);
\end{tikzpicture}%
\begin{tikzpicture}[x=8pt,y=8pt,thick]\pgfsetlinewidth{0.5pt}
\node[inner sep=1pt] at (0,0) {$\,=\,$}; \node at (0,-5) {};
\end{tikzpicture}%
\begin{tikzpicture}[x=8pt,y=8pt,thick]\pgfsetlinewidth{0.5pt}
\node(1) at (0,5) {$\scriptstyle{X}$};
\node[circle,draw, inner sep=0.7pt](2) at (0,1) {$\scriptstyle{i_X}$};
\node[circle,draw, inner sep=1pt](3) at (2,3) {$\scriptstyle{\beta}$};
\node(4) at (1,-5) {$\scriptstyle{A}$};

\draw[-] (1) to [out=-90,in=90] (2);
\draw[-] (2) to [out=-90,in=90] (0,-1);
\draw[-] (0,-1) to [out=-90,in=-90] (2,-1);
\draw[-] (3) to [out=-90,in=90] (2,-1);
\draw[-] (4) to [out=90,in=-90] (1,-1.6);
\end{tikzpicture}\begin{tikzpicture}[x=8pt,y=8pt,thick]\pgfsetlinewidth{0.5pt} \node[inner sep=1pt] at (0,0) {$\,=\,$};
\node at (0,-5) {};
\end{tikzpicture}
\begin{tikzpicture}[x=8pt,y=8pt,thick]\pgfsetlinewidth{0.5pt}
\node(1) at (0,5) {$\scriptstyle{X}$};
\node[circle,draw, inner sep=0.7pt](2) at (0,2) {$\scriptstyle{i_X}$};
\node[circle,draw, inner sep=1pt](3) at (2,0) {$\scriptstyle{\beta}$};
\node(4) at (1,-5) {$\scriptstyle{A}$};

\draw[-] (1) to [out=-90,in=90] (2);
\draw[-] (3) to [out=-90,in=90] (2,-1.5);
\draw[-] (2) to [out=-90,in=90] (0,-1.5);
\draw[-] (2,-1.5) to[out=-90,in=-90] (0,-1.5);
 \draw[-] (4) to [out=90,in=-90] (1,-2);
\end{tikzpicture}%
\begin{tikzpicture}[x=8pt,y=8pt,thick]\pgfsetlinewidth{0.5pt}
\node[inner sep=1pt] at (0,0) {$\,=\,$};
\node at (0,-5) {};
\end{tikzpicture}%
\begin{tikzpicture}[x=8pt,y=8pt,thick]\pgfsetlinewidth{0.5pt}
\node(1) at (0,5) {$\scriptstyle{X}$};
\node[circle,draw, inner sep=0.7pt](3) at (0,0) {$\scriptstyle{i_X}$};
\node(4) at (0,-5) {$\scriptstyle{A}$};

\draw[-] (3) to [out=90,in=-90] (1);
\draw[-] (4) to [out=90,in=-90] (3);
\end{tikzpicture}
\begin{tikzpicture}[x=8pt,y=8pt,thick]\pgfsetlinewidth{0.5pt}
\node[inner sep=1pt] at (0,0) {$\quad \Longrightarrow\quad$};
\node at (0,-5) {};
\end{tikzpicture}%
\begin{tikzpicture}[x=8pt,y=8pt,thick]\pgfsetlinewidth{0.5pt}
\node(1) at (0,5) {$\scriptstyle{X}$};
\node[rectangle,draw, inner sep=1.2pt, text width=1.5cm, text centered](2) at (0,-1) {$\scriptstyle{f_X^{-1}}$};
\node[circle,draw, inner sep=0.7pt](3) at (0,2) {$\scriptstyle{i_X}$};
\node(4) at (1.5,-5){$\scriptstyle{A}$};
\node(5) at (-1.5,-5) {$\scriptstyle{X}$};

 \draw[-] (1) to [out=-90,in=90] (3);
\draw[-] (3) to [out=-90,in=90] (2);
\draw[-] (4) to [out=90,in=-90] (1.5,-1.7);
\draw[-] (5) to [out=90,in=-90] (-1.5,-1.7);
\end{tikzpicture}%
\begin{tikzpicture}[x=8pt,y=8pt,thick]\pgfsetlinewidth{0.5pt}
\node[inner sep=1pt] at (0,0) {$\,=\,$};
\node at (0,-5) {};
\end{tikzpicture}%
\begin{tikzpicture}[x=8pt,y=8pt,thick]\pgfsetlinewidth{0.5pt}
\node(1) at (0,5) {$\scriptstyle{X}$};
\node[circle,draw, inner sep=1pt](3) at (2,2) {$\scriptstyle{\beta}$};
\node(4) at (2,-5) {$\scriptstyle{A}$};
\node(5) at (0,-5) {$\scriptstyle{X}$};

\draw[-] (1) to[out=-90,in=90] (5);
\draw[-] (4) to [out=90,in=-90] (3);
\end{tikzpicture}
\end{center}

Using this equality and the fact that $f_{X}^{-1}$ is right $A$-linear
(being the inverse of a right $A$-linear morphism) we compute
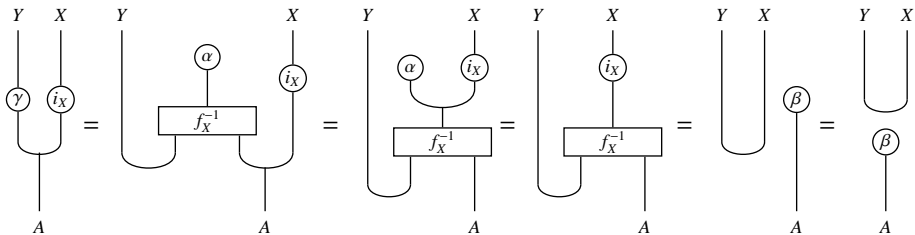
\begin{center}
\begin{tikzpicture}[x=8pt,y=8pt,thick]\pgfsetlinewidth{0.5pt}
 \node(1) at (0,5) {$\scriptstyle{Y}$};
\node[circle,draw, inner sep=1pt](2) at (0,1) {$\scriptstyle{\gamma}$};
\node(3) at (2,5) {$\scriptstyle{X}$};
\node[circle,draw, inner sep=0.7pt](4) at (2,1){$\scriptstyle{i_X}$};
\node(5) at (1,-5) {$\scriptstyle{A}$};

\draw[-] (1) to [out=-90,in=90] (2);
\draw[-] (3) to [out=-90,in=90] (4);
\draw[-] (2) to [out=-90,in=90](0,-1);
\draw[-] (4) to [out=-90,in=90](2,-1);
\draw[-] (0,-1) to [out=-90,in=-90] (2,-1);
\draw[-] (1,-1.5) to[out=-90,in=90] (5);
\end{tikzpicture}%
\begin{tikzpicture}[x=8pt,y=8pt,thick]\pgfsetlinewidth{0.5pt}
\node[inner sep=1pt] at (0,0) {$\,=\,$};
\node at (0,-5) {};
 \end{tikzpicture}%
\begin{tikzpicture}[x=8pt,y=8pt,thick]\pgfsetlinewidth{0.5pt}
\node(1) at (-1,5) {$\scriptstyle{Y}$};
\node[circle,draw, inner sep=1.5pt](2) at (3,3) {$\scriptstyle{\alpha}$};
\node(3) at (7,5) {$\scriptstyle{X}$};
\node[rectangle,draw, inner sep=1.2pt, text width=1.2cm, text centered](6) at (3,0) {$\scriptstyle{f_X^{-1}}$};
\node[circle,draw, inner sep=0.7pt](4) at (7,2) {$\scriptstyle{i_X}$};
\node(5) at (5.7,-5) {$\scriptstyle{A}$};

\draw[-] (1) to [out=-90,in=90] (-1,-1.5);
\draw[-] (-1,-1.5) to [out=-90,in=-90] (1.5,-1.5);
\draw[-] (1.5,-1.5) to [out=90,in=-90] (1.5,-0.7);
\draw[-] (4.5,-1.5) to [out=-90,in=-90] (7,-1.5);
\draw[-] (4.5,-0.7) to [out=-90,in=90] (4.5,-1.5);
\draw[-] (7,-1.5) to [out=90,in=-90] (4);
\draw[-] (2) to [out=-90,in=90] (6);
\draw[-] (4) to [out=90,in=-90] (3);
\draw[-] (5) to [out=90,in=-90] (5.7,-2.2);
\end{tikzpicture}
\begin{tikzpicture}[x=8pt,y=8pt,thick]\pgfsetlinewidth{0.5pt}
\node[inner sep=1pt] at (0,0) {$\,=\,$};
\node at (0,-5) {};
\end{tikzpicture}
\begin{tikzpicture}[x=8pt,y=8pt,thick]\pgfsetlinewidth{0.5pt}
\node(1) at (1,5) {$\scriptstyle{Y}$};

\node[circle,draw, inner sep=1.5pt](2) at (3,2.5) {$\scriptstyle{\alpha}$};
\node(3) at (6,5) {$\scriptstyle{X}$};
\node[rectangle,draw, inner sep=1.2pt, text width=1.2cm, text centered](6) at (4.5,-1) {$\scriptstyle{f_X^{-1}}$};
\node[circle,draw, inner sep=0.7pt](4) at (6,2.5) {$\scriptstyle{i_X}$};
\node(5) at (6,-5) {$\scriptstyle{A}$};

\draw[-] (1) to [out=-90,in=90] (1,-3);
\draw[-] (1,-3) to [out=-90,in=-90] (3,-3);
\draw[-] (3,-3) to [out=90,in=-90] (3,-1.7);
\draw[-] (2) to [out=-90,in=90] (3,1.5);
\draw[-] (4) to [out=-90,in=90] (6,1.5);
\draw[-] (3,1.5) to [out=-90,in=-90] (6,1.5);
\draw[-] (6) to [out=90,in=-90] (4.5,0.6);
\draw[-] (3) to [out=-90,in=90] (4);
\draw[-] (5) to [out=90,in=-90] (6,-1.7);
\end{tikzpicture}%
\begin{tikzpicture}[x=8pt,y=8pt,thick]\pgfsetlinewidth{0.5pt}
\node[inner sep=1pt] at (0,0) {$\,=\,$};
\node at (0,-5) {};
\end{tikzpicture}%
\begin{tikzpicture}[x=8pt,y=8pt,thick]\pgfsetlinewidth{0.5pt}
\node(1) at (1,5) {$\scriptstyle{Y}$};
\node(3) at (4.5,5) {$\scriptstyle{X}$};
\node[rectangle,draw, inner sep=1.2pt, text width=1.2cm, text centered](6) at (4.5,-1) {$\scriptstyle{f_X^{-1}}$};
\node[circle,draw, inner sep=0.7pt](4) at (4.5,2.5) {$\scriptstyle{i_X}$};
\node(5) at (6,-5) {$\scriptstyle{A}$};

\draw[-] (1) to [out=-90,in=90] (1,-3);
\draw[-] (1,-3) to [out=-90,in=-90] (3,-3);
\draw[-] (3,-3) to [out=90,in=-90] (3,-1.7);
\draw[-] (5) to [out=90,in=-90] (6,-1.7);
\draw[-] (4) to [out=-90,in=90] (6);
\draw[-] (3) to [out=-90,in=90] (4);
\end{tikzpicture}
\begin{tikzpicture}[x=8pt,y=8pt,thick]\pgfsetlinewidth{0.5pt}
\node[inner sep=1pt] at (0,0) {$\,=\,$}; \node at (0,-5) {};
\end{tikzpicture}
\begin{tikzpicture}[x=8pt,y=8pt,thick]\pgfsetlinewidth{0.5pt}
\node(1) at (1,5) {$\scriptstyle{Y}$};
\node(3) at (3,5) {$\scriptstyle{X}$};
\node[circle,draw, inner sep=1pt](2) at (4.5,1) {$\scriptstyle{\beta}$};
\node(5) at (4.5,-5) {$\scriptstyle{A}$};

\draw[-] (1) to [out=-90,in=90] (1,-1);
\draw[-] (1,-1) to [out=-90,in=-90] (3,-1);
\draw[-] (3) to [out=-90,in=90] (3,-1);
\draw[-] (2) to [out=-90,in=90] (5);
\end{tikzpicture}%
\begin{tikzpicture}[x=8pt,y=8pt,thick]\pgfsetlinewidth{0.5pt}
\node[inner sep=1pt] at (0,0) {$\,=\,$};
\node at (0,-5) {};
\end{tikzpicture}
\begin{tikzpicture}[x=8pt,y=8pt,thick]\pgfsetlinewidth{0.5pt}
\node(1) at (1,5) {$\scriptstyle{Y}$};
\node(3) at (3,5) {$\scriptstyle{X}$};
\node[circle,draw, inner sep=1pt](2) at (2,-1) {$\scriptstyle{\beta}$};
\node(5) at (2,-5) {$\scriptstyle{A}$};

\draw[-] (1) to [out=-90,in=90] (1,1);
\draw[-] (1,1) to [out=-90,in=-90] (3,1);
\draw[-] (3) to [out=-90,in=90] (3,1);
\draw[-] (2) to [out=-90,in=90] (5);
\end{tikzpicture}
\end{center}
which gives the desired equality \eqref{Eq:b}.
\end{proof}

Next result extends a result by Miyashita, see \cite[Proposition 1.1,
implication (ii) imlies (i)]{Mi}.
\begin{theorem}\label{thm:Rp}
Keep the assumptions and notations of Proposition \ref{pro:tau}.
Let $\left( Y',i_{Y'}:Y'\rightarrow A\right) :=\mathrm{Im}\left( \gamma \right) $
be the image of $\gamma $ in $\left( S,R\right) $-bimodules, and let $\left(
R^{\prime },\alpha ^{\prime }:R^{\prime }\rightarrow A\right) :=\mathrm{Im}
( m_{X}^{\prime }) $ be the image of $m_{X}^{\prime }$ in $\left(
R,R\right) $-bimodules, where%
\begin{equation*}
m_{X}^{\prime }:=\Big( X\tensor{S}Y'\overset{X\tensor{S}i_{Y'}}{%
\longrightarrow }X\tensor{S}A\overset{f_{X}}{\longrightarrow }A\Big)
=\Big( X\tensor{S}Y'\overset{i_{X}\tensor{S}i_{Y'}}{\longrightarrow }%
A\tensor{S}A\overset{\mas}{\longrightarrow }A\Big)
\end{equation*}%
Denote by $p_{Y'}:Y\rightarrow Y'$ the canonical projection so that $\gamma
=i_{Y'}\circ p_{Y'}.$ Similarly denote by $m_{X}:X\tensor{S}Y'\rightarrow
R^{\prime }$ the canonical projection such that $\alpha ^{\prime }\circ
m_{X}=m_{X}^{\prime }$. Then

1) $R^{\prime }$ is a monoid in $\left( R,R\right) $-bimodules such that $%
\alpha ^{\prime }:R^{\prime }\rightarrow A$ is a morphism of monoids therein;

2)\ $X \in \mathscr{P}\left( _{R^{\prime }}A_{S}\right) $ has a
right inverse given by $Y' \in \mathscr{P}\left( _{S}A_{R^{\prime
}}\right) $.
\end{theorem}
\begin{proof}
Unfortunately it not possible to present here a smart proof using diagrammatic notation,  since we are dealing with different  type of multiplications $m'_X,m_X, \overline{m_X}$ and it could be difficult to distinguish them properly.\\
1). Set $\overline{m_{X}} :=\left( X\tensor{S}p_{Y'}\right) \circ \mathrm{coev}%
:R\rightarrow X\tensor{S}Y'$ and $\theta :=m_{X}\circ \overline{m_{X}}$.
Note that
\begin{equation*}
\alpha ^{\prime }\circ \theta \,=\,\alpha ^{\prime }\circ m_{X}\circ \overline{m_{X}}\,=\,m_{X}^{\prime }\circ \overline{m_{X}} \, =\, f_{X}\circ \left( X\tensor{S}i_{Y'}\right) \circ \left( X\tensor{S}p_{Y'}\right) \circ \mathrm{coev}
\,=\,f_{X}\circ \left( X\tensor{S}\gamma \right) \circ \mathrm{coev}\overset{(\ref{Eq:a})}{=}\alpha.
\end{equation*}
so that $\alpha ^{\prime }\circ \theta =\alpha$.
Hence we have a commutative diagram
\begin{small}
\begin{equation*}
\xymatrix@R=20pt@C=30pt{ R \ar@{->}^-{\bara{m}_X}[rr]
\ar@{-->}^-{\theta}[rd] \ar@/_2pc/_-{\alpha}[rrdd] & & X\tensor{S}Y'
\ar@{->}^-{m_X}[ld] \ar@{->}^-{m'_X}[dd] \\ & R' \ar@{->}^-{\alpha'}[rd] &
\\ & & A}
\end{equation*}
\end{small}
So the unit of $R^{\prime }$ is $\theta \circ u_{R}$. The construction of
the multiplication is more involved. Let $\left( Q,i_{Q}:Q\rightarrow
Y\right) $ be the kernel of $p_{Y'}$. Tensoring by $X$ on the right we get
the exact sequence%
\begin{equation*}
\xymatrix@C=50pt{Q\tensor{R}X\ar@{->}^-{i_{Q}\tensor{R}X}[r] & Y\tensor{R}X
\ar@{->}^-{p_{Y'}\tensor{R}X}[r] & Y'\tensor{R}X \ar@{->}[r] & 0.}
\end{equation*}%

We then  have
\begin{equation*}
\beta \circ \mathrm{ev}\circ \left( i_{Q}\tensor{R}X\right) \overset{\eqref{Eq:b}}{=}\mar\circ \left( \gamma \tensor{R}i_{X}\right) \circ
\left( i_{Q}\tensor{R}X\right) \,=\, \mar\circ \left( i_{Y'}\tensor{R}i_{X}\right) \circ \left(
p_{Y'}\tensor{R}X\right) \circ \left( i_{Q}\tensor{R}X\right) =0.
\end{equation*}%
Since $\beta $ is a monomorphism we get $\mathrm{ev}\circ \left(
i_{Q}\tensor{R}X\right) =0$ so that there is a morphism of $S$-bimodules $%
\mathrm{ev}^{\prime }:Y'\tensor{R}X\rightarrow S$ such that
\begin{equation}
\mathrm{ev}^{\prime }\circ (p_{Y'}\tensor{R}X)=\mathrm{ev}.  \label{Eq:ev'}
\end{equation}%
Let us check that there is a morphism $m_{R^{\prime }}^{R}$ which turns the
following diagram commutative
\begin{small}
\begin{equation*}
\xymatrix@R=15pt@C=30pt{ (X\tensor{S}Y')\tensor{R}(X\tensor{S}Y')
\ar@{->}^-{m'_X\tensor{R}m'_X}[rr] \ar@{->}_-{X\tensor{S}ev'\tensor{S}Y'}[dd]
\ar@{->>}^-{}[rd] & & A\tensor{R}A \ar@{->}^-{m_A^R}[dddd] \\ & R'\tensor{R}
R' \ar@{->}^-{}[ru] \ar@{-->}^-{m_{R'}^R}[dd] & \\ X\tensor{S}S\tensor{S}Y'
\ar@{->}_-{r_X^S\tensor{S}Y'}^-{\cong}[dd] & & \\ &R' \ar@{->}^-{}[rd] & \\
X\tensor{S}Y' \ar@{->}_-{m_X'}[rr] \ar@{->}^-{}[ru] & & A }
\end{equation*}
\end{small}
On the one hand we have
\begin{equation*}
m_{X}^{\prime }\circ ((\rxs\circ \left( X\tensor{S}\mathrm{ev}%
^{\prime }\right) )\tensor{S}Y')\circ (X\tensor{S}p_{Y'}\tensor{R}X\tensor{S}Y')\overset{\eqref{Eq:ev'}}{=}m_{X}^{\prime }\circ
((\rxs\circ \left( X\tensor{S}\mathrm{ev}\right) )\tensor{S}Y').
\end{equation*}%
Since $X\tensor{S}p_{Y'}\tensor{R}X\tensor{S}Y'$ is an epimorphism, it
suffice to check that
\begin{equation*}
m_{X}^{\prime }\circ ((\rxs\circ \left( X\tensor{S}\mathrm{ev}\right)
)\tensor{S}Y'\,)\,=\,\,\mar\circ (m_{X}^{\prime }\tensor{R}m_{X}^{\prime })\circ \left( X\tensor{S}p_{Y'}\tensor{R}X\tensor{S}Y'\right) ,
\end{equation*}%
which follows from definitions by using equation \eqref{Eq:b}.
From this, one proves that $\mar\circ \left( \alpha ^{\prime }\tensor{R}\alpha ^{\prime }\right) $ factors through a map $\mrrp$ ($%
R^{\prime }$ is a kernel) such that $\alpha ^{\prime }\circ \mrrp=\mar\circ \left( \alpha
^{\prime }\tensor{R}\alpha ^{\prime }\right) $. We have%
\begin{eqnarray*}
\alpha ^{\prime }\circ \mrrp\circ \left( \mrrp\tensor{R}R^{\prime }\right) &=&\mar\circ \left( \alpha
^{\prime }\tensor{R}\alpha ^{\prime }\right) \circ \left(\mrrp\tensor{R}R^{\prime }\right) =\mar\circ \left( \mar\tensor{R}A\right) \circ \left( \alpha ^{\prime }\tensor{R}\alpha ^{\prime
}\tensor{R}\alpha ^{\prime }\right) \\
&=&\mar\circ \left( A\tensor{R}\mar\right) \circ \left( \alpha
^{\prime }\tensor{R}\alpha ^{\prime }\tensor{R}\alpha ^{\prime }\right)
=\alpha ^{\prime }\circ \mrrp\circ \left( R^{\prime }\tensor{R}\mrrp\right) , \\
\alpha ^{\prime }\circ \mrrp\circ \left( \theta \tensor{R}R^{\prime }\right) &=&\mar\circ \left( \alpha ^{\prime }\tensor{R}\alpha ^{\prime }\right) \circ \left( \theta \tensor{R}R^{\prime
}\right) =\mar\circ \left( \alpha \tensor{R}A\right) \circ \left( R\tensor{R}\alpha ^{\prime }\right) \\
&=&\lar\circ \left( R\tensor{R}\alpha ^{\prime }\right)=\alpha ^{\prime }\circ \lrrp
\end{eqnarray*}
and hence $\mrrp\circ \left( \theta \tensor{R}R^{\prime }\right) =\lrrp$. Similarly $\alpha ^{\prime }\circ \mrrp\circ \left( R^{\prime}\tensor{R}\theta \right) =\alpha ^{\prime }\circ \rrrp.$
Since $\alpha ^{\prime }$ is a monomorphism we conclude that $\mrrp$ is associative and unitary.

2) We have to prove that $X \in \mathscr{P}\left( _{R^{\prime
}}A_{S}\right) $. First, let $\left( K,i_{K}\right) $ be the kernel of $%
m_{X}^{\prime }$.
By definition of $m_{X}$ we have the exact sequence%
\begin{equation}\label{SeqRprime}
\xymatrix{ 0\ar@{->}^-{}[r] & K \ar@{->}^-{i_{K}}[r] &   X\tensor{S}Y' \ar@{->}^-{
m_{X}}[r] & R^{\prime} \ar@{->}[r] & 0.}
\end{equation}%
and hence
\begin{equation*}
\xymatrix@C=35pt{ K\tensor{R}X \ar@{->}^-{i_{K}\tensor{R}X}[r] &   X\tensor{S}Y'\tensor{R}X \ar@{->}^-{
m_{X}\tensor{R}X}[r] & R^{\prime}\tensor{R}X \ar@{->}[r] & 0.}
\end{equation*}%
Now,  using the unitality of $\mas$, the definitions of the morphisms   involved  and equation \eqref{Eq:b},  we get
$$
i_{X}\circ \rxs\circ \left( X\tensor{S}\mathrm{ev}^{\prime }\right)
\circ \left( X\tensor{S}p_{Y'}\tensor{R}X\right)  \,=\, \mas\circ \left( A\tensor{S}\mar\right) \circ \left(
i_{X}\tensor{S}i_{Y'}p_{Y'}\tensor{R}i_{X}\right)
$$
so that
\begin{multline*}
i_{X}\circ \rxs\circ \left( X\tensor{S}\mathrm{ev}^{\prime }\right)
= \\ \mas\circ \left( A\tensor{S}\mar\right) \circ \left(
i_{X}\tensor{S}i_{Y'}\tensor{R}i_{X}\right)
=\mar\circ \left( \mas\tensor{R}A\right) \circ \left(
i_{X}\tensor{S}i_{Y'}\tensor{R}i_{X}\right)
=\mar\circ \left( m_{X}^{\prime }\tensor{R}i_{X}\right) .
\end{multline*}%
Hence%
\begin{equation}
i_{X}\circ \rxs\circ \left( X\tensor{S}\mathrm{ev}^{\prime }\right)
=\mar\circ \left( m_{X}^{\prime }\tensor{R}i_{X}\right) .
\label{form:iXev'}
\end{equation}%
Therefore,
\begin{equation*}
i_{X}\circ \rxs\circ \left( X\tensor{S}\mathrm{ev}^{\prime }\right)
\circ \left( i_{K}\tensor{R}X\right) =\mar\circ \left( m_{X}^{\prime
}\tensor{R}i_{X}\right) \circ \left( i_{K}\tensor{R}X\right) =0
\end{equation*}%
and so $\rxs\circ \left( X\tensor{S}\mathrm{ev}^{\prime }\right) \circ \left(
i_{K}\tensor{R}X\right) =0$. By exactness of the last sequence displayed above, there is a morphism $\murpx:R^{\prime }\tensor{R}X\rightarrow X$ such that $
\murpx\circ \left( m_{X}\tensor{R}X\right) =\rxs\circ
\left( X\tensor{S}\mathrm{ev}^{\prime }\right) $. Let us check it is unitary. First, we have%
\begin{equation*}
i_{X}\circ \murpx\circ \left( m_{X}\tensor{R}X\right)
= i_{X}\circ \rxs\circ \left( X\tensor{S}\mathrm{ev}^{\prime
}\right) \overset{(\ref{form:iXev'})}{=}\mar\circ \left( m_{X}^{\prime
}\tensor{R}i_{X}\right)
= \mar\circ \left( \alpha ^{\prime }\tensor{R}i_{X}\right) \circ
\left( m_{X}\tensor{R}X\right) .
\end{equation*}%
Since $m_{X}\tensor{R}X$ is an epimorphism, we obtain%
\begin{equation}
i_{X}\circ \murpx=\mar\circ \left( \alpha ^{\prime
}\tensor{R}i_{X}\right) .  \label{form:imu}
\end{equation}%
We get%
\begin{gather*}
i_{X}\circ \murpx \circ \left( \theta \tensor{R}X\right)
\overset{(\ref{form:imu})}{=}\mar\circ \left( \alpha ^{\prime }\tensor{R}i_{X}\right) \circ \left( \theta \tensor{R}X\right) =\mar\circ
\left( \alpha \tensor{R}A\right) \circ \left( R\tensor{R}i_{X}\right) \\
=\lar \circ \left( R\tensor{R}i_{X}\right)=i_{X}\circ \lxr
\end{gather*}%
and hence $\murpx\circ \left( \theta \tensor{R}X\right)
=\lxr.$ Let us check that $\murpx$ is associative:%
\begin{gather*}
i_{X}\circ \murpx\circ \left( R^{\prime }\tensor{R}\mu
_{X}^{R^{\prime }}\right) \overset{(\ref{form:imu})}{=} \mar\circ \left(
\alpha ^{\prime }\tensor{R}i_{X}\right) \circ \left( R^{\prime }\tensor{R}\murpx\right)
\overset{(\ref{form:imu})}{=} \mar\circ \left( A\tensor{R}\mar\right) \circ \left( \alpha ^{\prime }\tensor{R}\alpha
^{\prime }\tensor{R}i_{X}\right) \\
=\mar\circ \left( \mar\tensor{R}A\right) \circ \left( \alpha
^{\prime }\tensor{R}\alpha ^{\prime }\tensor{R}i_{X}\right)
=\mar\circ \left( \alpha ^{\prime }\tensor{R}i_{X}\right) \circ
\left( \mrrp\tensor{R}X\right) \overset{(\ref{form:imu})}{=}%
i_{X}\circ \murpx\circ \left( \mrrp \tensor{R}X\right)
\end{gather*}%
so that $\murpx\circ \left( R^{\prime }\tensor{R}\murpx\right) =\murpx\circ \left( \mrrp\tensor{R}X\right) .$ The properties we proved imply that $\left(
X,i_{X}\right) \in \mathscr{P}\left( _{R^{\prime }}A_{S}\right) .$

Next aim is to check that $Y'\in \mathscr{P}\left(
_{S}A_{R^{\prime }}\right) $ and it is a right inverse of $X$. We need a morphism
$\mu_{\scriptscriptstyle{Y'}}^{\scriptscriptstyle{R'}}:Y'\tensor{R}R^{\prime}\rightarrow Y'$. Consider again the exact sequence \eqref{SeqRprime} and the induced one
\begin{equation*}
\xymatrix@C=40pt{ Y'\tensor{R}K\ar@{->}^-{Y'\tensor{R}i_{K}}[r] & Y'\tensor{R}X\tensor{S}Y' \ar@{->}^-{Y'\tensor{R}m_{X}}[r] &  Y'\tensor{R}R^{\prime } \ar@{->}^-{}[r] & 0.}
\end{equation*}%
As before  using the unitality of $\mas$, the definitions of the morphisms  involved  and equation \eqref{Eq:b},  we get
$$
i_{Y'}\circ \lys\circ \left( \mathrm{ev}^{\prime }\tensor{S}Y'\right)
\circ \left( p_{Y'}\tensor{R}X\tensor{S}Y'\right) \,=\,  \mas\circ \left( \mar\tensor{S}A\right) \circ \left(i_{Y'}p_{Y'}\tensor{R}i_{X}\tensor{S}i_{Y'}\right)
$$
and hence
\begin{multline*}
i_{Y'}\circ \lys\circ \left( \mathrm{ev}^{\prime }\tensor{S}Y'\right) \\
=\mas\circ \left( \mar\tensor{S}A\right) \circ \left(
i_{Y'}\tensor{R}i_{X}\tensor{S}i_{Y'}\right)
=\mar\circ \left( A\tensor{R}\mas\right) \circ \left(
i_{Y'}\tensor{R}i_{X}\tensor{S}i_{Y'}\right)
=\mar\circ \left( i_{Y'}\tensor{R}m_{X}^{\prime }\right) .
\end{multline*}
Thus we get
\begin{equation}
i_{Y'}\circ \lys\circ \left( \mathrm{ev}^{\prime }\tensor{S}Y'\right)
=\mar\circ \left( i_{Y'}\tensor{R}m_{X}^{\prime }\right) .
\label{form:iY'ev'}
\end{equation}%
Coming back to the exact sequence, we compute%
\begin{equation*}
i_{Y'}\circ \lys\circ \left( \mathrm{ev}^{\prime }\tensor{S}Y'\right)
\circ \left( Y'\tensor{R}i_{K}\right) \overset{(\ref{form:iY'ev'})}{=}%
\mar\circ \left( i_{Y'}\tensor{R}m_{X}^{\prime }\right) \circ \left(
Y'\tensor{R}i_{K}\right) =0.
\end{equation*}%
Therefore,  there is a unique morphism $\mu_{\scriptscriptstyle{Y'}}^{\scriptscriptstyle{R'}}:Y'\tensor{R}R^{\prime }\rightarrow Y'$ such that $\mu_{\scriptscriptstyle{Y'}}^{\scriptscriptstyle{R'}}\circ \left( Y'\tensor{R}m_{X}\right) =\lys\circ \left( \mathrm{ev}^{\prime}\tensor{S}Y'\right) $. Using this equality we compute%
\begin{equation*}
i_{Y'}\circ \mu_{\scriptscriptstyle{Y'}}^{\scriptscriptstyle{R'}}\circ \left( Y'\tensor{R}m_{X}\right)
\,=\,i_{Y'}\circ \lys\circ \left( \mathrm{ev}^{\prime }\otimes
_{S}Y'\right) \overset{(\ref{form:iY'ev'})}{=}\mar\circ \left(
i_{Y'}\tensor{R}m_{X}^{\prime }\right) \,=\,\mar\circ \left( i_{Y'}\tensor{R}\alpha ^{\prime }\right) \circ
\left( Y'\tensor{R}m_{X}\right).
\end{equation*}%
Since $Y'\tensor{R}m_{X}$ is an epimorphism, we obtain%
\begin{equation}
i_{Y'}\circ \mu_{\scriptscriptstyle{Y'}}^{\scriptscriptstyle{R'}} =\mar\circ \left( i_{Y'}\otimes
_{R}\alpha ^{\prime }\right) .  \label{form:imuY'}
\end{equation}%
Using this formula, as above, one proves that $\mu_{\scriptscriptstyle{Y'}}^{\scriptscriptstyle{R'}}$ is
unitary and associative and hence $Y' \in \mathscr{P}\left(
_{S}A_{R^{\prime }}\right) $.

Let us check that $Y'$ is a right inverse for $
X$. Consider the coequalizer
\begin{equation*}
\xymatrix@C=50pt{ Y'\tensor{R}R^{\prime }\tensor{R}X \ar@<0.5ex>[r]^-{\mu _{Y'}^{R^{\prime }}\tensor{R}X}  \ar@<-0.5ex>[r]_-{Y'\tensor{R}\murpx} &
Y'\tensor{R}X \ar@{->}^-{Y'\tensor{\theta}X}[r] & Y'\tensor{R^{\prime }}X  \ar@{->}[r] & 0.}
\end{equation*}
We  have
\begin{multline*}
\mathrm{ev}^{\prime }\circ \left(\mu_{\scriptscriptstyle{Y'}}^{\scriptscriptstyle{R'}}\tensor{R}X\right) \circ \left( Y'\tensor{R}m_{X}\tensor{R}X\right)
=\mathrm{ev}^{\prime }\circ \left( \lys\tensor{R}X\right) \circ
\left( \mathrm{ev}^{\prime }\tensor{S}Y'\tensor{R}X\right)
=\mathrm{ev}^{\prime }\circ l_{\scriptscriptstyle{Y'\tensor{R}X}}^{\scriptscriptstyle{S}}\circ \left( \mathrm{ev}
^{\prime }\tensor{S}Y'\tensor{R}X\right)
\\  =l_{\scriptscriptstyle{S}}^{\scriptscriptstyle{S}}\circ \left( S\tensor{S}\mathrm{ev}^{\prime }\right) \circ
\left( \mathrm{ev}^{\prime }\tensor{S}Y'\tensor{R}X\right)
=r_{\scriptscriptstyle{S}}^{\scriptscriptstyle{S}}\circ \left( \mathrm{ev}^{\prime }\tensor{S}S\right) \circ
\left( Y'\tensor{R}X\tensor{S}\mathrm{ev}^{\prime }\right)
=\mathrm{ev}^{\prime }\circ r_{Y'\tensor{R}X}^{S}\circ \left( Y'\tensor{R}X\tensor{S}\mathrm{ev}^{\prime }\right) \\
=\mathrm{ev}^{\prime }\circ \left( Y'\tensor{R}\rxs\right) \circ
\left( Y'\tensor{R}X\tensor{S}\mathrm{ev}^{\prime }\right)
=\mathrm{ev}^{\prime }\circ \left( Y'\tensor{R}\mu_{\scriptscriptstyle{X}}^{\scriptscriptstyle{R'}}\right) \circ \left( Y'\tensor{R}m_{X}\tensor{R}X\right)
\end{multline*}
so that $\mathrm{ev}^{\prime }\circ \left( \mu_{\scriptscriptstyle{Y'}}^{\scriptscriptstyle{R'}}\tensor{R}X\right)
=\mathrm{ev}^{\prime }\circ \left( Y'\tensor{R}\murpx\right)$
and hence there exists $m_{Y'}:Y'\tensor{R^{\prime }}X\rightarrow S$ such
that $m_{Y'}\circ \left( Y'\otimes _{\theta }X\right) =\mathrm{ev}^{\prime }$.   Therefore,
\begin{multline*}
\beta \circ m_{Y'}\circ \left( Y'\tensor{\theta }X\right) \circ \left(
p_{Y'}\tensor{R}X\right) =\beta \circ \mathrm{ev}^{\prime }\circ \left(
p_{Y'}\tensor{R}X\right)
=\beta \circ \mathrm{ev}\overset{(\ref{Eq:b})}{=}\mar\circ \left(
\gamma \tensor{R}i_{X}\right)
=\mar\circ \left( i_{Y'}\tensor{R}i_{X}\right) \circ \left(
p_{Y'}\tensor{R}X\right) \\
=m_{\scriptscriptstyle{A}}^{\scriptscriptstyle{R'}}\circ \left( A\tensor{{\theta }}A\right) \circ
\left( i_{Y'}\tensor{R}i_{X}\right) \circ \left( p_{Y'}\tensor{R}X\right)
=m_{\scriptscriptstyle{A}}^{\scriptscriptstyle{R'}}\circ \left( i_{Y'}\tensor{R^{\prime }}i_{X}\right)
\circ \left( Y'\tensor{\theta }X\right) \circ \left( p_{Y'}\tensor{R}X\right)
\end{multline*}%
and hence $\beta \circ m_{Y'}=m_{\scriptscriptstyle{A}}^{\scriptscriptstyle{R'}}\circ \left( i_{Y'}\tensor{R^{\prime
}}i_{X}\right)$. Moreover $\alpha ^{\prime }\circ m_{X}=m_{X}^{\prime }=\mas\circ \left(i_{X}\tensor{S}i_{Y'}\right)$.

It remains to check that $m_{X}$ is an isomorphism. It is an epimorphism by
construction. Since $\alpha ^{\prime }\circ m_{X}=m_{X}^{\prime }$ and $%
\alpha ^{\prime }$ is a monomorphism by construction, we have that $m_{X}$
is a monomorphism if and only if $m_{X}^{\prime }$ is. Now $m_{X}^{\prime
}=f_{X}\circ \left( X\tensor{S}i_{Y'}\right) $ and $f_{X}$ is an
isomorphism by assumption. Thus $m_{X}^{\prime }$ is a monomorphism if and
only if $X\tensor{S}i_{Y'}$ is a monomorphism. Since by Proposition \ref{prop:adjunction}, we know that the functor $X\tensor{S}\left( -\right) $ is a right adjoint,
hence  $X\tensor{S}i_{Y'}$ is a
monomorphism. Thus $m_{X}$ is both an epimorphism and a monomorphism and
hence it is an isomorphism, and this completes the proof.
\end{proof}

\section{Internals hom for modules and bimodules in monoidal categories}\label{ssec:appendix2}
In this appendix we  show the main steps  to construct  the internal homs  functors which were  implicitly used  in Subsections \ref{ssec:central} and \ref{ssec:azumaya}.
To this aim, consider a symmetric monoidal abelian bicomplete category $(\cat{M},\tensor{},\I)$ with right exact tensor products. Let $(A,m,u)$ be a monoid in $\cat{M}$ and denote by $A^e:=A\tensor{}A^o$ its enveloping monoid, where $A^o$ is the opposite monoid. Assume that the functor $A\tensor{}-: \cat{M} \to \cat{M}$ has a right adjoint functor $[A,-]: \cat{M} \to \cat{M}$ with unit and counit
$$ \B{\varepsilon}^{A}_Y: A\tensor{}[A,Y] \longrightarrow Y \in \cat{M}, \quad
\B{\zeta}^{A}_X: \cat{M} \ni X \longrightarrow [A,A\tensor{}X].$$ The functor $[A,-]$ is referred to as the \emph{left internal hom functor}.
It is clear that $[A\tensor{}\cdots\tensor{}A,-]$ also exists and so is in particular  for $[A^e,-]$. We will use similar notation for the unit and counit of the corresponding adjunctions. Given any two objects $X, Y$ in $\cat{M}$, one can define a  morphism $[m,X]: [A,X] \to [A\tensor{}A,X]$ which naturally turns commutative the following diagram:
\begin{small}
$$
\xymatrix@C=50pt@R=20pt{  \hom{\cat{M}}{A\tensor{}A\tensor{}X}{Y}   \ar@{->}^-{\Phi^{A\tensor{}A}_{X,\,Y}}_-{\cong}[rr]& & \hom{\cat{M}}{X}{[A\tensor{}A,Y]}
\\  \hom{\cat{M}}{A\tensor{}X}{Y} \ar@{->}^{\Psi^{A}_{X,Y}}_-{\cong}[rr] \ar@{->}^-{\hom{\cat{M}}{m\tensor{}X}{Y}}[u] & & \hom{\cat{M}}{X}{[A,Y]}   \ar@{-->}_-{\hom{\cat{M}}{X}{[m,X]}}[u] }
$$
\end{small}

On the other hand, for any object $X$ in $ \cat{M}$ we can define the following morphism:
\begin{small}
\begin{equation*}
\xymatrix@C=45pt@R=20pt{ [A,X] \ar_-{\B{\zeta}^{A^e\tensor{}A}_{[A,X]}}[dr] \ar@{-->}^-{\Omega^{A^e}_{A,\, X}}[rr] & &  [A^e\tensor{}A,A^e\tensor{}X] \\  & \Lr{A^e\tensor{}A,A^e\tensor{}A\tensor{}[A,X]}. \ar_-{[A^e\tensor{}A,\,A^e\tensor{}\B{\varepsilon}^A_X]}[ru] & }
\end{equation*}
\end{small}
Clearly $ \Omega^{A^e}_{A,\, -}$ is a natural transformation as a composition of natural transformations.

Now,  for every  left $A^e$-module $M$ (i.e.~ an $A$-bimodule) with action $\lambda_M: A^e\tensor{}M \to M$ we can consider the morphism
$
\chi_{A,\, M}:= [A^e\tensor{}A,\lambda_M] \circ \Omega^{A^e}_{A,\, M},
$
and so the following equalizer:
\begin{equation*}
\xymatrix@C=50pt{0 \ar@{->}[r] & {}_{A^e}[A,M] \ar@{->}^-{\fk{eq}}[r] &  [A,M] \ar@<.5ex>[r]^-{[m,M]}  \ar@<-.5ex>[r]_-{\chi_{A,\,M}} & [A\tensor{}A,M],}
\end{equation*}
which leads to a natural monomorphism  $ {}_{A^e}[A,-] \hookrightarrow [A,\mathscr{O}(-)]$, where $\mathscr{O}: {}_{A^e}\cat{M} \to \cat{M}$ is the forgetful functor.

The following result whose proof is similar to that of \cite[Proposition 3.10]{PareigisI}, summarize the relation between the functors constructed above. Indeed, its shows that, if the functor $A\tensor{}-:\cat{M} \to \cat{M}$ has a right adjoint, then so is $A\tensor{}-: \cat{M} \to {}_{A^e}\cat{M}$.

\begin{proposition}\label{prop:appII}
Let $(A,m,u)$ be a monoid in a symmetric monoidal abelian and bicomplete category $(\cat{M},\tensor{},\I)$ whose tensor products are right exact.  Assume that there is an adjunction  $A\tensor{}-\dashv [A,-]$, where $\xymatrix{  A\tensor{}-: \cat{M} \ar@<.5ex>[r]  &  \cat{M}: [A,-]. \ar@<.5ex>[l] }$ Then there is a commutative diagram of natural transformations:
\begin{small}
$$
\xymatrix@C=50pt@R=20pt{  \hom{\cat{M}}{A\tensor{}X}{M}   \ar@{->}^-{\Phi^A_{X,\,M}}_-{\cong}[rr]& & \hom{\cat{M}}{X}{[A,M]}
\\  \lhom{A^e}{A\tensor{}X}{M} \ar@{-->}^{\Psi^A_{X,M}}_-{\cong}[rr] \ar@{^{(}->}[u] & & \hom{\cat{M}}{X}{{}_{A^e}[A,M]}   \ar@{_{(}->}[u] }
$$
\end{small}
where the vertical arrows denote the canonical injections and $X \in \cat{M}$, $M \in {}_{A^e}\cat{M}$. That is, we have an adjunction  $A\tensor{}-\dashv {}_{A^e}[A,-]$, where $\xymatrix{  A\tensor{}-: \cat{M} \ar@<.5ex>[r]  &  {}_{A^e}\cat{M}:  {}_{A^e}[A,-]. \ar@<.5ex>[l] }$
\end{proposition}

\smallskip

\textbf{Acknowledgements}  We would like to thank Claudia Menini for helpful discussions on a preliminary version of the paper. We are also grateful to the referee for carefully reading the manuscript and for several useful comments.
L. El Kaoutit would like to thank all the members of the  Department of Mathematics of University of Ferrara
for a warm hospitality during his visit.

\end{document}